\documentclass[a4paper,fleqn]{cas-sc}

\usepackage[numbers]{natbib}
\DeclareMathOperator{\e}{e}
\def\NN{{\mathbb N}}

\def\R{{\mathbb R}}
\def\RR{{\mathbb R}}
\def\P{{\mathbb P}}

\def\d{\,\mathrm d}	
\def\x{\,\mathrm x}

\newcommand{\mrm}[1]{\mathrm{#1}}
\renewcommand{\div}{\mrm{div}}
\newcommand{\mL}{\mrm{L}}
\newcommand{\mH}{\mrm{H}}

\newcommand{\mD}{\mrm{D}}

\newcommand{\mW}{\mrm{W}}
\newcommand{\mK}{\mrm{K}}
\newcommand{\mC}{\mrm{C}}

\newtheorem{theorem}{Theorem}[section]
\newtheorem{lemma}{Lemma}[section]
\newtheorem{rmk}{Remark}[section]
\newtheorem{definition}{Definition}[section]
\newtheorem{corollary}{Corollary}[section]
\newtheorem{proposition}{Proposition}[section]

\newtheorem{Assumption and notation}{Assumption and notation}[section]
\newtheorem{General assumptions}{General assumptions}[section]

\newtheorem{Definition and notation}{Definition and notation}[section]
\newproof{pf}{proof}
\def\tsc#1{\csdef{#1}{\textsc{\lowercase{#1}}\xspace}}
\tsc{WGM}
\tsc{QE}
\tsc{EP}
\tsc{PMS}
\tsc{BEC}
\tsc{DE}
\begin{document}
\let\WriteBookmarks\relax
\def\floatpagepagefraction{1}
\def\textpagefraction{.001}
\shorttitle{Optimal control problem with state constraints
for cancer chemotherapy and treatment optimization}
\shortauthors{D. Lassounon et~al.}

\title [mode = title]{Optimal Control Problem with Mixed Control and State Constraints for Cancer Chemotherapy and Treatment Optimization}                    


\author[1,2]{David Lassounon}
[
 orcid=0000-0002-3011-4020
 ]
\cormark[1]
\ead{david.lassounon@uni-lorraine.fr}


\affiliation[1]{organization={Université de Lorraine, CNRS, CRAN, F-54000 Nancy, France}
            }

\affiliation[2]{organization={Univ Rennes, INSA, CNRS, IRMAR-UMR 6625, F-35000, France},
postcode={35708},
city={Rennes},
country={France}}

\author[2]{Aziz Belmiloudi}
\ead{aziz.belmiloudi@math.cnrs.fr}
\author[2]{Mounir Haddou}
\ead{mounir.haddou@insa-rennes.fr}

\cortext[cor1]{Corresponding author}


\begin{abstract}
The success of chemotherapy depends on the effectiveness of the drug delivery strategy and its ability to destroy cancer cells while minimizing damage to healthy tissues. The main objective of this work is to minimize the density of invasive tumour cells by controlling the chemotherapeutic agents. For this, we address an optimal control problem with mixed control and state constraints. The concentration of chemotherapeutic drugs is represented as a control variable. 
We use a nonlinear reaction-diffusion equation to describe the effect of drugs on the progression of invasive tumours. We start with the mathematical analysis of this initial boundary value problem. Then, we formulate the optimal control problem, explain the role of different constraints, and derive first-order necessary conditions of optimality. Finally, in order  to demonstrate the efficiency of the proposed strategy, numerical simulations in the case of the eradication  of malignant lung tumours, are presented and analysed.
\end{abstract}



\begin{keywords}
Optimal control problem\sep
State and control constraints\sep
Cancer chemotherapy\sep 
Nonlinear Boundary Condition\sep
First-order optimality conditions\sep 
Optimal therapeutic strategy\sep
Numerical simulations\sep
Lung tumours\sep
\end{keywords}
\maketitle
\section{Introductions and Problem Statement}
\subsection{Motivation}

Cancer is a disease characterized by uncontrolled cell growth, leading to the formation of tumors that can invade other parts of the body. Cancer cells can detach from the original tumor, spreading through blood and lymphatic vessels to form new tumors in a process known as metastasis. The World Health Organization (WHO) identifies cancer as one of the leading causes of death globally. Moreover, from  the International Agency for Research on Cancer
(IARC), lung cancer was the most frequently diagnosed cancer in 2022, responsible for almost 2.5 million new cases, or one in eight cancers world-wide ($12.4\%$ of all cancers globally), and  followed by cancers of the female breast ($11.6\%$) (see \cite
{BLS2024}). In particular, lung cancer is the third common cancer and the leading cause of cancer death in France (among men). Between 2010 and 2023 (52,777 new cases estimated in 2023 with 33,438 men and 19,339 women), while the incidence of this cancer decreased slightly in men (-0.5\% per year), it increased sharply in women (+4.3\% per year). Because of this, developing efficient, precise and personalized treatment strategies  for cancer patients is  a key challenge.\\
To better understand the dynamic and complex process of tumor proliferation, several mathematical models have been developed over the years to describe the mechanisms governing tumor evolution. For 
models displayed by ordinary differential equations we can cite e.g., \cite{VlaGo2004, EfB2021,  RHM2014, SKMa2023} and references therein, and for models using partial differential equations (called the proliferation-invasion models) see e.g., \cite{Bel2010, aB2017, bulai2025numerical, fife2013, FLB2019, HWBa2017, RHSwa2019, SMGh2019}, and references therein.  

Various therapeutic methods are available for cancer treatment, including surgery, chemotherapy, radiotherapy, and immunotherapy. However, these approaches often face clinical challenges. Tumour cells can develop resistance to treatments. Drugs may cause acute toxicity to normal cells, and secondary cancers can emerge following treatment. These issues are prevalent across all types of cancer.  In particular, in  chemotherapy, the treatment response models have been proposed (see, e.g. \cite{aB2017, CR2020, GLG2022, JaHB2018, santra2024modeling, WeMA2013,WMAr2015,bao2024mathematical,kumar2024impact}) to integrate clinical and experimental data to assess tumour response to treatment, thereby evaluating the efficacy of various therapeutic approaches. The key challenge in chemotherapy treatment aims to find an optimal balance between the effective destruction of cancer cells and the minimization of damage to healthy tissue.
Today, oncology experts increasingly apply optimal control theory, a powerful framework for addressing this challenge, to optimize and personalize therapeutic strategies based on chemotherapy  treatment response models. This approach allows clinicians to optimize dosing protocols and compare different treatment scenarios. Unfortunately, most of these studies neglect constraints on the state, even though accounting for them can significantly improve proposed treatment strategies. 
 Without taking these state constraints into account, the cancer chemotherapy treatment through optimal control process can be unpredictable, fluctuating between increases and decreases, even when control constraints are applied (see, e.g. the recent papers \cite{ABAM, LSB, LBH2, LHJ2021,musleh2025constrained}, and the references therein).
So, it is important to impose, in addition to the constraint on drug dosage, a constraint on tumour density  itself to be respected at each treatment time. This latter ensures that the total tumour burden does not exceed a critical threshold throughout the duration of treatment, and this is the context in which this work fits.

 This work presents an optimal treatment strategy using an optimal control framework with mixed control and state constraints, as well as a nonlinear partial differential equation describing the effect of chemotherapy on invasive tumours.
Our goal is to reduce the density of malignant tumour cells while minimizing damage to normal cells, thus proposing optimal treatment strategies. To achieve this, we control the concentration of chemotherapeutic agents, thereby limiting drug toxicity in the patient's body.
In our optimal control problem, the state constraint represents a temporal limitation on tumour density during treatment aimed at reducing any resistance and progression of tumour cells. These constraints are involved in managing acute drug toxicity in patients and controlling the evolution of tumour density throughout treatment. We, therefore, model the evolution of tumour invasion using the nonlinear partial differential equations in the following form
 \begin{equation}
  \begin{cases}
 \label{Eqstate}
\displaystyle{\frac{\partial \mathbf{u}}{\partial t}+\mathcal{A}(\mathbf{u})=f_1(t,\x,\mathbf{u},\varphi)\mbox{ in $\mathcal{Q}:=(0, T)\times\Omega$ },}\\
 \displaystyle{-\frac{\partial \mathbf{u}}{\partial \mathbf{n}_{\mathcal{A}}}:=-\mD(t,\x)\nabla{{\mathbf{u}}}\cdot\mathbf{n}=f_2(t,\x,\mathbf{u})\mbox{ on $\Sigma:=(0, T)\times\Gamma$ },}\\
\displaystyle{ \mathbf{u}(0,\x)=u_{0}(\x) \mbox{ in $\Omega,$ }}
\end{cases}
\end{equation}
under the pointwise constraints 
\begin{equation}\label{PC}
\mathbf{a}(\x)\leq \varphi(t,\x)\leq \mathbf{b}(\x) \quad a.e.\hspace{0.1cm} (t,\x)\in \cal{Q},
\end{equation}
where
\begin{gather}
\begin{split}
\label{termes sources principales}
 f_1(t,\x,\mathbf{u},\varphi)&=\mK_2(t,\x)\mathbf{u}(t,\x)-\mK_1(t,\x)\text{h}_{1}(\mathbf{u}(t,\x))-\alpha_{0}\varphi(t,\x)\mathbf{u}(t,\x) \mbox{ in $\Bar{\cal{Q}},$ }\\  
f_{2}(t,\x,\mathbf{u})&=\mK_{3}(t,\x)\text{h}_{2}(\mathbf{u}(t,\x)) \mbox{ on ${\Sigma,}$ }\\
 \mathcal{A}(\mathbf{u})&=-\div(\mD(t,\x)\nabla{\mathbf{u}}(t,\x)),
 \end{split}
\end{gather}
with (for $i=1,2$) 
\begin{equation}\label{hi}
 \text{h}_{i}:z\in\RR\longmapsto \text{h}_{i}(z)=
\begin{array}{|ll}
    z^{\rho_i}\quad\mbox{ if $\rho_{i}\in \NN\backslash\lbrace{1}\rbrace,$ }\\
     z|z|^{\rho_{i}-1}\mbox{ if $\rho_{i}\in \RR\setminus\NN, \rho_{i}>1,$ and $\rho_1 +1 \ge  2\rho_2.$ }\\
\end{array}
\end{equation}
Let us describe this system. The region $\Omega$  is an open bounded domain in $\R^{m}, m\ge 1$ representing the tumour region, with a smooth boundary $\Gamma=\partial \Omega,$ $T>0$ is finite and fixed time horizon (the final treatment time in days $(d)$), and $\mathbf{n}$ is the outward normal to $\Gamma$. The value $\mathbf{u}(t,\x)$ is the density of tumour cells at time $t$ and at position $\x$ in millimeter $(mm)$ of the region $\Omega.$ 
 The first two terms of the nonlinear operator $f_1$ represent the proliferation and multiplicity of tumour cells with an intrinsic growth rate and proliferation rate per day denoted respectively $\mK_{1}$ $(d^{-1}),$ and $\mK_2$ $(d^{-1})$, with a multiplicity coefficient $\rho_{1}.$ As the chemotherapy protocol is temporal and spatial, the control variable $\varphi(t,\x)$ in micromolar ($\mu$M) is the concentration of drugs at time $t$ and position $\x$, with a calibration coefficient $\alpha_{0}$ ($\mu$M$^{-1}\cdot d^{-1}$).\\
 The last term of the operator $f_1$ describes the mortality of tumour cells due to the effect of chemotherapy. Moreover, considering the malignancy of tumour cells, the nonlinear operator $f_2$ describes the invasion of tumour cells to other organs of the body beyond their sources, with an invasion rate denoted $\mK_3$ $(mm\cdot d^{-1})$ and a growth coefficient $\rho_2.$ This invasive nature of tumour cells towards other nearby organs is neglected in certain models (see, e.g., \cite{aB2017,JaHB2018, LHJ2021,WeMA2013, WMAr2015} and the references therein), where a homogeneous Neumann condition is often considered. We take this into account in the hope of obtaining a much more general model. The function $ u_{ 0}$ is the initial density of tumour cells at time $t=0$ and $\mD$ $(mm^{2}\cdot d^{-1})$ is the diffusivity function of tumour cells. The control bounds $\mathbf{a},\mathbf{b}$ are regular functions defined by clinicians based on the clinical conditions of the patient. The constraint \eqref{PC} imposed on the drug makes it possible to avoid any risk of acute toxicity in the body. For the parameters of the model we assume that
$\mu_1\ge\mD(t,\x)\ge\mu_0>0,$ and that $0<\underline{\mK_{i}}\leq\mK_{i}(t,\x)\leq\overline{\mK_{i}}$ for all $(t,\x)\in\Bar{\mathcal{Q}}$, with $\overline{\mK_i},\underline{\mK_i}\in\RR^{*}_{+}$, for $i=1,3$.\\

Our work yields the following contributions.
\begin{itemize}
\item \textbf{Tumour invasion and chemotherapy.} We consider a nonlinear reaction–diffusion equation describing tumour response under treatment. tumour invasion into other organs is described by a nonlinear boundary condition. A rigorously mathematical analysis of the equation is provided, including the global existence of solutions (see Theorem \ref{Théorème principal pour le bien posé}), the uniqueness and  stability results (see Theorem \ref{Théorème principal pour le bien posé 2}). The positivity and boundedness properties of the solution are derived in Proposition \ref{principe de maximum}.
\item \textbf{Optimal control problem with state and control constraints.} We propose and describe an optimal control problem with mixed control and state constraints to minimize tumour density while controlling the systemic toxicity of chemotherapeutic agents. A key feature of our approach is the explicit integration of a time-dependent state constraint on the tumour density. This constraint ensures that the tumour density remains below a prescribed threshold throughout the treatment period. The existence of an optimal solution to the control problem is established in Theorem \ref{existence solution of P}, together with a differentiability result for the solution mapping of the state equation (see Theorem \ref{proposition Gdiff}). First-order optimality conditions are then derived and characterized by an adjoint equation involving a Borel measure term (see Theorem \ref{theorem optimality}). Finally, the existence and uniqueness of the solution to this adjoint equation are proved in Proposition \ref{welpadjsys}.
\item \textbf{Numerical simulations.} 
We present numerical techniques for computing optimal treatment solutions. Numerical simulations are performed for the treatment of two cases of malignant lung tumours, one located in the left lobe and the other in the right lobe, near the main bronchus within a realistic lung geometry. The numerical results demonstrate the effectiveness of the proposed treatment strategy.
\end{itemize}

The paper is organized as follows. First, in Section \ref{setting1}, we present the general assumptions regarding the data and parameters of the system \eqref{Eqstate}. In Section \ref{setting2}, we discuss the well-posedness and stability results of this system. In Section \ref{setting3}, we formulate the optimal control problem with mixed constraints and prove the existence of optimal solutions. After stating and proving a differentiability result for the solution operator of the system \eqref{Eqstate}, we describe the necessary optimality conditions for the optimal control problem under certain conditions. In Section \ref{setting4}, we use a constraint penalization approach to numerically solve the control problem and present numerical simulations of solutions to eradicate two cases of infiltrative lung tumours using drugs. Finally, the conclusion and future work are discussed in Section \ref{setting5}.
\subsection{Definitions and General Assumptions}\label{setting1}
In our theoretical results, we  examine cases where the initial tumour density is either regular or not regular, and for this, we make the following assumptions:\\
\textit{ $\textbf{(H1)}:$ The initial tumour density $u_{0}$ is in $\mW^{r,q}(\Omega)$ with $r\in\{0,1\}$ and $1<q<\infty$ be fixed.  The tumour growth coefficients $\rho_i$ for $i=1,2$, witch depend on $r$ and $q$ are also fixed such that \\ 
$
(I_1): \begin{cases}
 \rm{ (i)}  \mbox{ if $r=0,$ then  $\rho_1\in]1;\frac{m+2q}{m}]$ and $\rho_2\in]1;\frac{m+q}{m}].$}\\
\rm{(ii)}  \mbox{ if $r=1$  and $q<m,$ then $\rho_1\in]1;\frac{m+q}{m-q}]$ and $\rho_2\in]1;\frac{m}{m-q}].$}\\ 
\rm{(iii)} \mbox{ if $r=1$ and $q\ge m,$ then $\rho_i\in]1;\infty[$ for $i=1,2.$}
\end{cases}
$
}
\\\\
\textit{ $\textbf{(H2)}:$  The bounds of the control $\mathbf{a},\mathbf{b}\in\mL^{\mathbf{p}_c}(\Omega)$ are given functions in $\Omega,$ such that $0\leq \mathbf{a}\leq\mathbf{b},$ and $\mathbf{p}_c$ satisfies
\\
 $
(I_2): \begin{cases} 
 \rm{(i)} \mbox{ for $r=0:$ }\\
 \mbox{ $\mathbf{p}_c\in]\frac{\rho_1}{\rho_1 -1};\frac{q\rho_1}{\rho_1 -1}]$ if $\rho_1\ge q,$ } and 
\mbox{ $\mathbf{p}_c\in[\frac{q}{\rho_1 -1};\frac{q\rho_1}{\rho_1 -1}]$ if $\rho_1< q.$ }
\\
\rm{ (ii)} \mbox{ for $r=1$ and $q<m:$ } \\
 \mbox{ $\mathbf{p}_c\in]\frac{mq}{(m-q)(\rho_1 -1)};\frac{q\rho_1}{\rho_1 -1}]$ if $\frac{m}{m-q}<\rho_1,$ } and  
\mbox{ $\mathbf{p}_c\in[\frac{\rho_1}{\rho_1 -1};\frac{q\rho_1}{\rho_1 -1}]$ if $\rho_1\leq\frac{m}{m-q}.$  }\\
\rm{ (iii)}  \mbox{ for $r=1$ and $q\ge m,$ $\mathbf{p}_c=\frac{q\rho_1}{\rho_1 -1}.$}
\end{cases}
$
}\\
We denote by $\mathcal{U}_{ad}$ the set of admissible controls defined by
\begin{gather*}
\mathcal{U}_{ad}:=\big\{\varphi\in \mL^{\mathbf{p}_{c}}({\cal Q}):\hspace{0.01cm} \mathbf{a}(\x)\leq \varphi(t,\x)\leq \mathbf{b}(\x),\hspace{0.01cm} a.e. \hspace{0.1cm} (t,x)\in{\cal Q}\big\}.
\end{gather*}
Although $\mathcal{U}_{ad}$ is non-empty and is a closed convex subset of the space $\mL^{\infty}(0, T;\mL^{\mathbf{p}_{c}}(\Omega))$, we prefer to use the standard norms of the space $\mL^{\mathbf{p}_{c}}(\mathcal{Q})$.\\
We denote by $\mathbf{B}$ the following beta-function
 $$\forall \alpha_{1},\alpha_{2}>0,\quad \mathbf{B}(\alpha_{1};\alpha_{2}):=\displaystyle \int^{1}_{0}(1-s)^{\alpha_{1}-1}s^{\alpha_{2}-1}\d s. $$
From now on, we assume that the assumptions \textit{\textbf{(H1)}} and \textit{\textbf{(H2)}} are satisfied. For the linear operator $\mathcal{A}$ defined in our model, we define its restriction denoted $A_q$ by
\begin{gather*}
A_{q}v:=\mathcal{A}v \quad \forall v\in\mathcal{E}^1_q(\Omega),
\end{gather*}
where $
 \mathcal{E}^1_q(\Omega):=\{v\in \mW^{2,q}(\Omega); \nabla{v}\cdot\mathbf{n}=0 \mbox{ on } \Gamma \}.$\\
The operator $A_q$ can be seen as an unbounded operator in $\mathcal{E}^0_q(\Omega):=\mL^{q}(\Omega).$ From \cite[Theo 1.3.4]{H1981}, $A_{q}$ is the generator of an analytic semigroup denoted $\{\e^{A_{q}t},t\geq 0\}$ on $\mathcal{E}^0_q(\Omega).$ Then, according to \cite[Sec.6]{Am1993}, the scale of fractional powers spaces and operators $(A_{\alpha,q};\mathcal{E}^{\alpha}_q(\Omega))_{\alpha\in \RR}$ generated by $A_q$ satisfy $ \mathcal{E}^{-{\alpha}}_{q^{*}}(\Omega)=:(\mathcal{E}^{\alpha}_q(\Omega))',$ $\mathcal{E}^{\frac{1}{2}}_q(\Omega)=\mW^{1,q}(\Omega)$ and the following continuous injections
 \begin{gather}
     \label{injections continues}
     \begin{cases}
     \mathcal{E}^{\alpha}_q(\Omega) \hookrightarrow \mH^{2\alpha}_q(\Omega),\quad \alpha\ge 0,\\
      \mathcal{E}^{\alpha}_q(\Omega) \hookrightarrow \mL^{\ell}(\Omega) \mbox{ for } \quad \ell\leq \frac{mq}{m-2\alpha q}, \quad  0\leq\alpha<\frac{m}{2q},\\
   \mL^s(\Omega)\hookrightarrow \mathcal{E}^{\alpha}_q(\Omega)\mbox{ for } \quad s\geq \frac{mq}{m-2\alpha q},\quad -\frac{m}{2q^{*}}<\alpha\leq 0,\quad q^{*}=\frac{q}{q-1},
   \end{cases}
 \end{gather}
 where $\mH^{\gamma}_q(\Omega), \gamma\in \mathbb{R}$ are the spaces of distribution $u$ such that $ \displaystyle  (I-\Delta)^{\frac{\gamma}{2}}u \in\mathcal{E}^{0}_{q}(\Omega)$ (in particular we have, $\forall\gamma\in\NN, \forall q\ge 1$ $\mH^{\gamma}_q (\Omega)=\mW^{\gamma,q}(\Omega),$  and $\forall \gamma\notin\NN,$ $\mH^{\gamma}_2 (\Omega)=\mW^{\gamma,2}(\Omega)=\mH^{\gamma}(\Omega)$). The operator  $A_{\alpha,q}: \mathcal{E}^{\alpha+1}_q(\Omega)\subset \mathcal{E}^{\alpha}_q(\Omega)\longrightarrow \mathcal{E}^{\alpha}_q(\Omega)$ is the restriction of $A_q$ on $\mathcal{E}^{\alpha}_{q}(\Omega)$ if $\alpha>0,$ and the continuous extension of $A_q$ in $\mathcal{E}^{\alpha}_{q}(\Omega)$ if $\alpha<0$. \\
 Similar, for the nonlinear operators in the model, we  define Banach spaces based on the regularity of the initial data $u_0\in\mW^{r,q}(\Omega)$ with  $r\in\{0,1\}$  as follows.
 \begin{itemize}
 \item For $r=0,$ we denote by $\mathcal{X}^{\beta}_{q}(\Omega):=\mathcal{E}^{\beta-1}_{q}(\Omega)$ and let  $A_{-1,q}: \mathcal{X}^{1}_{q}(\Omega)\subset\mathcal{X}^{0}_{q}(\Omega)\longrightarrow \mathcal{X}^{0}_{q}(\Omega).$ From \eqref{injections continues}, we have the following continuous injections
 \begin{gather}
 \label{injection xlq}
  \begin{cases}
      \mathcal{X}^{\beta}_q(\Omega) \hookrightarrow \mL^\ell(\Omega),\quad \mbox{ for } \ell\leq \frac{mq}{m+2q-2q\beta},\quad 1\leq\beta<1+\frac{m}{2q},\\
      \mL^s(\Omega)\hookrightarrow \mathcal{X}^{\beta}_q(\Omega) ,\quad \mbox{ for } s\geq \frac{mq}{m+2q-2q\beta},\quad 1-\frac{m}{2q^{*}}<\beta\leq 1,\\
      \mathcal{X}^{-{\beta}}_{q^{*}}(\Omega)=(\mathcal{X}^{\beta}_q(\Omega))'.
      \end{cases}
 \end{gather}
\item For $r=1,$ we denote by $\mathcal{X}^{\beta}_{q}(\Omega):=\mathcal{E}^{\beta-\frac{1}{2}}_{q}(\Omega)$ and let  $A_{-\frac{1}{2},q}: \mathcal{X}^{1}_{q}(\Omega)\subset\mathcal{X}^{0}_{q}(\Omega)\longrightarrow \mathcal{X}^{0}_{q}(\Omega).$ We have 
  \begin{gather}
  \label{injection wlq}
    \begin{cases}
      \mathcal{X}^{\beta}_q (\Omega)\hookrightarrow \mL^\ell(\Omega),\quad \mbox{ for } \ell\leq \frac{mq}{m+q-2\beta q},\quad \frac{1}{2}\leq\beta<\frac{1}{2}+\frac{m}{2q},\\
      \mL^s(\Omega)\hookrightarrow \mathcal{X}^{\beta}_q ,\quad \mbox{ for } s\geq \frac{mq}{m+q-2\beta q},\quad \frac{1}{2}-\frac{m}{2q^{*}}<\beta\leq \frac{1}{2},\\
      \mathcal{X}^{-{\beta}}_{q^{*}}(\Omega)=(\mathcal{X}^{\beta}_q (\Omega))'.
     \end{cases}
 \end{gather}
 \end{itemize}
 Note that in both cases, we have $\mW^{r,q}(\Omega)=\mathcal{X}^{1}_q(\Omega).$
From \cite[Theo 2.4, Prop 1.1]{A1987}, $A_{\alpha,q}$ generates an analytic semigroup $\{e^{A_{\alpha,q }t},t\geq 0\}$ in  $\mathcal{E}^{\alpha}_q(\Omega)$ ( which is the restriction of $ \{e^{A_q t},t\geq 0\}$ to $ \mathcal{E}^{\alpha}_q(\Omega)$ if $\alpha>0$ and the continuous extension of $\{e^{A_q t},t\geq 0\}$ over $ \mathcal{E}^{\alpha}_q(\Omega)$ if $\alpha<0$), and there is a positive constant $M(\alpha,\beta):=M$ such that 
 \begin{gather}
 \label{defM}
    \lVert \e^{A_{\alpha,q}t}\rVert_{\mathcal{L}(\mathcal{E}^{\alpha}_q(\Omega),\mathcal{E}^{\beta}_q(\Omega))}\leq Mt^{\alpha-\beta}, \quad  \alpha\leq \beta, \quad t>0.
\end{gather}
Noting by $\gamma_{\Gamma}$ the trace operator on $\mH^{r}_{q}(\Omega),$ then for $rq>1$ it is well known that (see, e.g.  \cite{H1981}) $\gamma_{\Gamma}$ is continuous from $\mH^{r}_{q}(\Omega)$ to $\mL^{\ell}(\Gamma)$ for all
\begin{gather}
\label{opérateur trace}
\begin{cases}
\ell\leq \infty,  \mbox{ if }  rq>m\\
\ell<\infty,      \mbox{ if } rq=m \\
\ell\leq \frac{(m-1)q}{m-rq},  \mbox{ if } rq< m. 
\end{cases}
\end{gather}
In the sequel, $C$ is a positive generic constant whose exact value is not important and may change from line to line. Now, we collect some known definitions of a $\epsilon-regular$ function and $\epsilon-regular$ mild solutions of \eqref{Eqstate} for some $\epsilon>0.$ For more details, we refer the reader to \cite{JAC1999,ACRB1999}.
\begin{definition}
\label{def epreg}
Let $\epsilon\geq 0,$ a map $f$ is $\epsilon-regular$ relative to $(\mathcal{X}^1_q(\Omega),\mathcal{X}^0_q(\Omega))$ if there exists a constant $\rho>1$, a real $\gamma\in[\rho\epsilon;1[,$ a constant $C>0,$ and a positive increasing function \(\nu\) defined on \((0, +\infty)\) with the property \(\lim_{t \to 0^+} \nu(t) = 0\) such that for all $u,v\in \mathcal{X}^{1+\epsilon}_q(\Omega)$ and $t>0,$
\begin{gather}
\begin{split}
\label{epregu}
 \lVert  f(t,\x,u)-  &f(t,\x,v)\lVert _{\mathcal{X}^{\gamma}_q(\Omega)} \leq C\lVert u-v\lVert  _{\mathcal{X}^{1+\epsilon}_q(\Omega)} \big(\lVert   u\lVert^{\rho-1}_{\mathcal{X}^{1+\epsilon}_q(\Omega)}+\lVert  v\lVert  ^{\rho-1}_{\mathcal{X}^{1+\epsilon}_q(\Omega)}+\nu(t)t^{\epsilon-\gamma}\big)\\
 & \lVert f(t,\x,u)\lVert _{\mathcal{X}^{\gamma}_q(\Omega)} \leq C\big(\lVert u\lVert^{\rho}_{\mathcal{X}^{1+\epsilon}_q(\Omega)}+\nu(t)t^{-\gamma}\big).
\end{split}
\end{gather}
\end{definition}
\begin{definition}
\label{epregsol}
We say that a function $u: [0;\tau]\longrightarrow\mW^{r,q}(\Omega)$ ($r\in\{0,1\}$) is an $\epsilon-regular$ mild solution  ($\epsilon-solution$ for short) for some $\tau>0$ to the problem \eqref{Eqstate},  if $u\in \mC([0,\tau];\mW^{r,q}(\Omega))\cap \mC((0,\tau];\mathcal{X}^{1+\epsilon}_q(\Omega)),$ and satisfying
\begin{gather}
\label{mild form}
u(t,\x)=\e^{\tilde{A}t}u_{0}(\x)+\int^{t}_{0}\e^{\tilde{A}(t-s)}\Big(f_{1}(s,\x,u,\varphi)+\tilde{f}_{2}(s,\x,u)\Big)\d s,
\end{gather}
where $\tilde{A}$ denotes either $A_{-1,q}$ or $A_{-\frac{1}{2},q},$ and $\tilde{f}_{2}$ is the trace lift of the nonlinear operator $f_2$ in $\Bar{\mathcal{Q}}.$ More precisely, $f_2$ is the restriction of $\tilde{f}_2$ on $\Sigma.$
\end{definition}
\begin{definition}
We say that a function $u: [0;\tau]\longrightarrow\mW^{r,q}(\Omega)$ $(r\in\{0,1\})$ is a classical solution of \eqref{Eqstate}
on $[0; \tau],$ if $u\in \mC^{1}((0,\tau];\mathcal{X}^{\epsilon}_q(\Omega))\cap\mC([0,\tau];\mW^{r,q}(\Omega))\cap \mC ((0,\tau];\mathcal{X}^{1+\epsilon}_q(\Omega))$ and satisfying 
\begin{gather*}
\begin{cases}
\partial_t u(t,\x)+\tilde{A}u(t,\x)=f_{1}(t,\x,\varphi,u)+\tilde{f}_{2}(t,\x,u),\\
u(0,\cdot)=u_{0},\quad 0< t\leq \tau.
\end{cases}
\end{gather*}
\end{definition}
 \section{Well-Posedness and Stability}
 \label{setting2}
In this section, we present some existence, uniqueness, and stability results for the problem defined in \eqref{Eqstate}. We suppose that the assumptions \textit{(\textbf{H1})} and \textit{(\textbf{H2})} are satisfied.
\begin{theorem}\label{Théorème principal pour le bien posé}(Existence)\\
For all $u_0\in\mW^{r,q}(\Omega)$ $(r\in\{0,1\}, q>1),$ and $\varphi\in\mathcal{U}_{ad},$ there exists $\underline{\gamma}>0,$ $\overline{\epsilon}\in]0;\underline{\gamma}[,$ a time $\tau_0>0$ and a unique $\overline{\epsilon}-regular$ mild solution $\mathbf{u}$ on $[0;\tau_0]$ to the problem \eqref{Eqstate}. Moreover, if $\rho_1 +1>2\rho_2$ or $\rho_1 +1 =2\rho_2$ with $4\nu_0(q-1)\inf_{\mathcal{Q}}{\mK_{1}}\ge c_{\Omega}(\inf_{\Sigma}{\mK_{3}})^{2}(\rho_2 +q-1)^{2}$, then $\mathbf{u}$ is the classical solution globally defined in all time $t\in (0;  T)$ for the problem \eqref{Eqstate}.
\end{theorem}
The proof of this theorem can be found in Appendix \ref{setting61}.\\
The following theorem concerns the uniqueness and stability results for the system \eqref{Eqstate} as follows:
\begin{theorem}\label{Théorème principal pour le bien posé 2} (Uniqueness and stability)\\
Let $\mathbf{u}_1,$ resp. $\mathbf{u}_2$ be the classical solution of \eqref{Eqstate} with respect to $(u_{0,1},\varphi_{1}),$ resp. $(u_{0,2},\varphi_{2})$ belonging to  $\mW^{r,q}(\Omega)\times\mathcal{U}_{ad}$ $(r\in\{0,1\}, q>1).$ Then, there exists $C>0$ such that for all $t\in(0; T),$ $\theta\in[0;\underline{\gamma}],$ $\forall 0\leq\epsilon\leq\theta$ the following relations hold
\begin{gather}
\begin{split}
 \label{stabilité dt}
  \lVert \mathbf{u}_{1}(t,\cdot) -\mathbf{u}_{2}(t,\cdot) \lVert_{\mathcal{X}^{1+\theta}_q(\Omega)}\leq Ct^{-\underline{\gamma}}\Big(\lVert u_{0,1}-u_{0,2}\lVert_{\mW^{r,q}(\Omega)}+\lVert \varphi_{1} -\varphi_{2} \Vert_{\mathcal{U}_{ad}}\Big),
  \\
   \lVert \frac{\partial\mathbf{u}_{1}}{\partial t}(t,\cdot)-\frac{\partial \mathbf{u}_2}{\partial t}(t,\cdot)\lVert_{\mathcal{X}^{\theta-\epsilon}_q(\Omega)}\leq Ct^{-\underline{\gamma}+\epsilon}\Big(\lVert u_{0,1}-u_{0,2}\lVert_{\mW^{r,q}(\Omega)}+\lVert \varphi_{1} -\varphi_{2} \Vert_{\mathcal{U}_{ad}}\Big). 
   \end{split}
\end{gather}
\end{theorem}
The proof of this result is included in Appendix \ref{preuve bien pose2}.\\
Since the solution \(\mathbf{u}\) of the problem \eqref{Eqstate} represents tumour density, the following proposition establishes its positivity and boundedness; the proof can be found in Appendix \ref{preuve principe maximum}.
\begin{proposition}\label{principe de maximum}(Maximum principle)\\
For $u_0 \in\mW^{r,q}(\Omega)$ ($r\in\{0,1\}$, $q>1$) such that $0\leq u_{0}\leq c_{s}$ and $\varphi\in\mathcal{U}_{ad}$, the solution $\mathbf{u}$ of \eqref{Eqstate} is bounded in $\mathcal{Q}$ and verifies
\begin{gather}
\label{borneu}
0\leq \mathbf{u}(t,\cdot)\leq c_{s}, \mbox{ for all $t\in (0;T)$ and $a.e.$ in } \Omega,
\end{gather}
where $c_{s}=\bigg{(}\frac{\overline{\mK_{2}}}{{\underline{\mK_{1}}}}\bigg{)}^{\frac{1}{\rho_{1}-1}}$. \end{proposition}

\section{Optimal Control Problem }\label{setting3}
In this section, we consider the case where $q=2,$ $\theta=0,$ $m\leq 3$, $\rho_1\ge 2$, $\rho_2 =2$, and $\mathbf{p}_c\ge 2$ (the existence of such $\mathbf{p}_c$ is garanted when $\rho_1\ge 2$ according to $(I_2)$).
\subsection{Problem Setting}\label{hypnot} 

We introduce the following spaces
$\mathcal{H}_{0}:=\mH^{1}(\Omega),$  ${\cal V}_{0}:=\mL^{\infty}(0,T;{\cal H}_{0}),$
${\cal W}_{0}:=\mH^1(0,T;{\mathcal{H}_{0}}'),$
$\mathcal{\mathbf{U}}_{0}:={\cal V}_{0}\cap {\cal W}_{0},$ and 
we assume that $u_{0}\in\mathcal{H}_{0}$ and satisfies $0\leq u_{0}\leq c_{s}.$ Based on the previous section, the state problem defined in \eqref{Eqstate} has a unique classical solution $\mathbf{u}\in\mathcal{\mathbf{U}}_{0}$ which is bounded and positive, corresponding to the control $\varphi\in\mathcal{U}_{ad}.$ Moreover, we have the following results
\begin{corollary}
\label{Existence and stability}
Let $(u_{0,1}, \varphi_{1}),$ $resp.$ $(u_{0,2}, \varphi_{2})\in \mathcal{H}_{0}\times\mathcal{U}_{ad},$  then there exists $\mathbf{u}_1,$ $resp.$ $\mathbf{u}_2$ the unique solution of  \eqref{Eqstate} that belong to $\mathcal{\mathbf{U}}_{0},$ corresponding to $(u_{0,1}, \varphi_{1}),$ $resp.$ $(u_{0,2}, \varphi_{2})$ such that 
 \begin{gather}
 \label{estimations stabilite}
 \begin{split}
 \lVert \mathbf{u}_{1}-\mathbf{u}_{2}\lVert_{\mathcal{\mathbf{U}}_{0}}\leq C\Big(\lVert u_{0,1}-u_{0,2}\lVert_{\mathcal{H}_{0}}+\lVert \varphi_{1} -\varphi_{2} \Vert_{\mathcal{U}_{ad}}\Big).
 \end{split}
\end{gather}
\end{corollary}
\begin{pf}
This corollary is a direct consequence of Theorem \ref{Théorème principal pour le bien posé 2}.
\end{pf}
We now present some notations that will be useful later on. 
Let  $\mathcal{C}$ be a closed convexe subset of some Banach space $\mathcal{B}$ The normal cone to $\mathcal{C}$ at a point $\bar{c}\in\mathcal{C}$ is defined by
\begin{gather}\label{conenormal}
N_{\mathcal{C}}(\bar{c}):=\{v\in\mathcal{B}'; \langle v, c-\bar{c}\rangle_{\mathcal{B}',\mathcal{B}}\leq 0, \mbox{ for all } c\in\mathcal{C}\}.
\end{gather}
where $\mathcal{B}'$ is the dual of $\mathcal{B}$.\\
Let $\mathcal{K}:=\mC([0; T])_{-}$ be the cone of the negative continuous functions on $[0; T].$ Its negative polar cone is defined by $\mathcal{K}^- :=\text{M}([0; T])_{+},$ the set of positive Borel measures on $[0; T].$ Since $\mathcal{K}$ is a closed convex cone then, $N_{\mathcal{K}}(v)=\mathcal{K}^{-}\cap v^{\perp}$ for all $v\in \mathcal{K},$ see e.g. \cite{BS2000}.
 In the sequel, we assume that linear forms over $\mC([0; T])$ are Stieltjes integrals associated with $\mu\in BV(0, T)$ (bounded variation space) such that $\mu(T) = 0$ and 
  we denote by $v^{+}:=\max(v; 0)$,  $v^{-}:=-\min(v; 0)$ ($\forall v\in \RR$).\\
We can now introduce the optimal control problem.   Let's define the solution operator, of problem \eqref{Eqstate}, $\mathcal{F}: \varphi\in \mathcal{U}_{ad}\longmapsto \mathcal{F}(\varphi)\in\mathcal{\mathbf{U}}_{0}$ such that $\mathcal{F}(\varphi)=\mathbf{u}.$ 
 On the state variable $\mathbf{u},$ we consider the following state constraint
\begin{gather}\label{state constraint}
\mathcal{G}(\varphi)=\lVert\mathcal{F}(\varphi)(t,\cdot)\lVert^{2}_{\mL^{2}(\Omega)}-\zeta(t)\leq 0  \quad a.e. \hspace{0.1cm} t\in (0;T).
\end{gather}
 We introduce the following objective (or cost) functional $J$ by 
\begin{gather}\label{Cost func}
    J(\varphi)=\frac{1}{2}\int_{\cal Q}\mathbf{u}^{2}\d\x\d t +\frac{\lambda}{\mathbf{p}_{c}}\int_{\cal Q} | \varphi|^{\mathbf{p}_{c}}\d\x\d t, \quad \lambda>0.
\end{gather}
The control problem aims to obtain a minimizer of the functional $J$ with respect to the control $\varphi$.  More concretely, we study the following problem $({\P})$
\begin{gather}
   \begin{split}
    \label{control problem}
 &\inf_{ \varphi\in \mathcal{U}_{ad}} J(\varphi); \\
& s.t., \quad \mathcal{G}(\varphi)\in \mathcal{K}.
\end{split}
\end{gather}
We now give some details about the data, parameters, and objectives of this optimal control problem.
    \begin{itemize}
        \item The problem $(\P)$ aims to minimize the malignant tumour density with less damage caused by drugs to normal cells. The control constraint helps limit the acute toxicity of drugs on normal cells. In contrast, the state constraint focuses on decreasing the density of tumour cells during treatment. 

        \item The parameter $\zeta$ represents the maximum allowable tumour density. If the tumour density surpasses $\zeta$ during treatment, this signifies a violation of the state constraint, which may endanger the patient's survival.
        \item The first part of the cost function aims to minimize tumour density, while the second part represents the cost of control.
    \end{itemize}

\subsection{Existence of an Optimal Solution} 
In this subsection, we  prove the existence of an optimal solution to the problem $(\P)$. For this, we need the following
\begin{lemma}
\label{faible continute de F}
    The solution operator $\mathcal{F}:\mathcal{U}_{ad}
\longrightarrow\mathcal{\mathbf{U}}_{0}$  is sequentially weakly continuous.
\end{lemma}
\begin{pf}
 Let a sequence $(\varphi_n)_{n\in \NN}$ of $\mathcal{U}_{ad}$ such that $\varphi_{n} \rightharpoonup {\varphi}^*$  weakly in $\mathcal{U}_{ad},$ then $\varphi_n$ is bounded in $\mathcal{U}_{ad}.$ We shall prove that $\mathbf{u}_n \rightharpoonup {\mathbf{u}}^{*} $  weakly in $\mathcal{\mathbf{U_{0}}}$, where
$\mathbf{u}_n = \mathcal{F}(\varphi_{n})$ and  ${u}^*=\mathcal{F}({\varphi}^*).$ To do this, we prove that we can pass to the limit
in each term of the system \eqref{Eqstate}.
\begin{itemize}
 \item By the Corollary \ref{Existence and stability}, $\mathcal{F}(\varphi_n)$ is bounded in $\mathcal{\mathbf{U}}_{0},$ so extracting a subsequent, we may assume that $\mathcal{F}(\varphi_n)\rightharpoonup \hat{u}$ weakly in $\mathcal{\mathbf{U}}_{0}$. By Aubin-Lions Lemma, the injection of $\mathcal{\mathbf{U}}_{0}$ into $\mL^2(\cal{Q})$ is compact. Consequently, we have that 
$u_{n} \longrightarrow  \hat{u}$ strongly in $\mL^2(\cal{Q})$, this implies that the convergence occurs almost everywhere in \(\cal{Q}\). Moreover, $\mathbf{u}_n(0,\cdot)\longrightarrow  \hat{u}(0,\cdot)$ strongly in $\mC([0,T];\mL^{2}(\Omega))$ and almost everywhere in $\Omega.$ Furthermore, according to \textit{\textbf{(H1)}}, we have 
 $ \text{h}_1(\mathbf{u}_n)$ is bounded in $\mL^{\infty}(0,T;\mL^{\frac{5}{\rho_1}}(\Omega))$ (since $m\leq 3$ and by $(I_1)-\rm{(ii)}$, in the worst case ($r=1,$ $m=3$), $\rm{h}_{1}(\mathbf{u}_n)\in\mL^{\frac{5}{\rho_{1}}}(\Omega)$ with $\rm{h}_1$ definied in \eqref{hi}). We may suppose that $ \text{h}_1(\mathbf{u}_n)\rightharpoonup \tilde{w}$ weakly star  in $\mL^{\infty}(0,T;\mL^{\frac{5}{\rho_1}}(\Omega))$. Moreover, we have that $ \text{h}_1(\mathbf{u}_n)$ is bounded in $\mL^{\frac{5}{\rho_1}}(\cal{Q})$ since $\mathbf{u}_n \in\mC(0,T;\mathcal{H}_{0}).$ So, we obtain $\text{h}_1(\mathbf{u}_n) \longrightarrow\text{h}_1 ( \hat{u})=v$ almost everywhere in $\mathcal{Q},$ and $ \text{h}_1(\mathbf{u}_n)\rightharpoonup \tilde{w}$ weakly in $\mL^{\frac{5}{\rho_1}}(\cal{Q}).$
 Therefore, we have $\tilde{w}=v=\text{h}_1( \hat{u})$, and that $ \mK_{1}\text{h}_1(\mathbf{u}_n)\rightharpoonup \mK_{1}\text{h}_1 ( \hat{u})$  weakly star in $\mL^{\infty}(0,T;\mL^{\frac{5}{\rho_1}}(\Omega))$, and $ \mK_{2}\mathbf{u}_n\longrightarrow \mK_{2}{ \hat{u}}$ strongly in $\mL^{2}(\cal Q )$ (since $\mK_1 ,\mK_{2}\in \mL^{\infty}(\cal{Q})$). 
\item  Let $v\in\mL^{\infty}(\cal{Q}),$ then we have $\int_{\cal Q }{\varphi_n}(\mathbf{u}_n - \hat{u})v \longrightarrow 0,$ and $\int_{\cal Q }({\varphi_n} -{{\varphi}^*}) \hat{u} v\longrightarrow 0.$  Indeed, since ${\varphi_n}$ is bounded in $\mathcal{U}_{ad},$ then $\varphi_{n} \rightharpoonup {\varphi}^{*}$ weakly star in $\mathcal{U}_{ad}.$ Using item 1, we have $\mathbf{u}_{n} \longrightarrow  \hat{u}$ in $\mL^2(\cal{Q})$ strongly, then  $\int_{\cal Q}{\varphi_n}\mathbf{u}_{n}v\longrightarrow \int_{\cal Q}{{\varphi}^*} \hat{u}v$ for all $v\in\mL^{\infty}(\cal{Q}).$ Since $\mL^{\infty}(\cal{Q})$ is dense in $\mL^2(\cal{Q}),$ then ${\varphi_n}\mathbf{u}_{n}\rightharpoonup {{\varphi}^*} \hat{u}$ weakly in $\mL^2(\cal{Q}).$ 
\item $\mathbf{u}_{n}\rightharpoonup \hat{u}$ weakly in $\mL^2(0,T;\mH^{\frac{1}{2}}(\Gamma))$ and since the injection of $\mH^{\frac{1}{2}}(\Gamma)$ is compact in $\mL^2(\Gamma),$ then $\mathbf{u}_{n}\longrightarrow { \hat{u}}$ strongly in $\mL^2(\Sigma),$ it almost converges everywhere in $\Sigma.$  Moreover, according to \textit{\textbf{(H1)}} and since $m\leq 3$ et by $(I_1)-\rm{ii.)}$, in the worst case ($r=1,$ $m=3$), $\rm{h}_2(\mathbf{u}_n)\in\mL^{2}(\Gamma)$ with $\rm{h}_2$ definied in \eqref{hi}, we have $\text{h}_2(\mathbf{u}_{n})$ is bounded in $\mL^{\infty}(0,T;\mL^{2}(\Gamma))$. We may suppose that $\text{h}_2(\mathbf{u}_{n})\rightharpoonup z_0$ weakly star in $\mL^{\infty}(0,T;\mL^{\frac{3}{\rho_2}}(\Gamma)).$ Since $\text{h}_2(\mathbf{u}_{n})$ is bounded in $\mL^{2}(\Sigma),$ we have $\text{h}_2(\mathbf{u}_{n}) \longrightarrow\text{h}_2( \hat{u})=z_1,$ almost everywhere in $\Sigma,$ and $\text{h}_2(\mathbf{u}_{n})\rightharpoonup z_0$ weakly  in $\mL^{2}(\Sigma).$ Therefore, $z_0 =z_{1}=\text{h}_{2}( \hat{u})$ in $\Sigma$ and since $\mK_{3}\in\mL^{\infty}(\Bar{\mathcal{Q}}),$ then
   $\mK_{3}\text{h}_2(\mathbf{u}_{n})\rightharpoonup \mK_{3}\text{h}_{2}( \hat{u})$ weakly star in $\mL^{\infty}(0,T;\mL^{2}(\Gamma)).$
\end{itemize}
Using items 1, 2, 3, and passing to limit in the system \eqref{Eqstate}, we obtain $ \hat{u}={\mathbf{u}^*}$ by the uniqueness of solutions of the system \eqref{Eqstate}.
\end{pf}
We state the existence of an optimal solution for the problem $(\P)$ as follows.
\begin{theorem}\label{existence solution of P}
  $(\P)$ admits at least one optimal solution ${\varphi}_* \in \mathcal{U}_{ad}.$ 
    \end{theorem}
\begin{pf}
    Let $\varphi_{k}\in \mathcal{U}_{ad}$ be a minimizing sequence of $J,$ $i.e.,$
    \begin{gather*}  
\liminf_{k\longrightarrow\infty} J(\varphi_{k})=\inf_{\varphi\in\mathcal{U}_{ad}} J(\varphi),
    \end{gather*}
 such that $\mathcal{G}(\varphi_k)\in \mathcal{K}.$    
According to the nature of the cost function $J,$ we can deduce that $\varphi_{k}$
is uniformly bounded in $\mathcal{U}_{ad}$ and we can extract from $\varphi_{k}$ a subsequence also
denoted by $\varphi_{k}$ such that $\varphi_{k}\rightharpoonup\varphi_{*}$ weakly in $\mathcal{U}_{ad}$ with $\mathcal{F}(\varphi_{k})=\mathbf{u}_{k}$ and $\mathcal{F}(\varphi_* )=\mathbf{u}_*$. Therefore, using the Lemma \ref{faible continute de F}, we have $\mathbf{u}_{k}\rightharpoonup \mathbf{u}_*$ weakly in $\mathcal{\mathbf{U_{0}}}.$ Moreover,
the map $\mathcal{G}(\Phi): t\mapsto \mathcal{G}(\Phi)(t)$ is continuous, then $\mathcal{G}(\varphi_k)\rightharpoonup\mathcal{G}(\varphi_*)$ weakly star in $\mC([0;T])$ and $\mathcal{G}(\varphi_*)(t)\leq \liminf_{k}\mathcal{G}(\varphi_k)(t)\leq 0$ for all $t\in[0;T].$ Finally, since the norm is lower semi-continuous, then the control $\varphi_*$ is an optimal solution, $i.e.,$
\begin{gather*}
\inf_{\phi\in\mathcal{U}_{ad}}J(\phi)=\liminf_{k\longrightarrow\infty}J(\varphi_k)=J(\varphi_*), \mbox{ and $\mathcal{G}(\varphi^*)\in \mathcal{K}.$ }
\end{gather*}
\end{pf}
In the sequel, we need the well-posedness and stability of the following linear parabolic problem:
   \begin{equation}
   \begin{cases}
  \label{linear problem}
  \frac{\partial v}{\partial t}(t,\x)- \div(\mD(t,\x)\nabla{v(t,\x)})+k_1(t,\x) v(t,\x)=k(t,\x) \mbox{ in $\cal Q,$ }\\
 -\mD(t,\x)\nabla{ v(t,\x)}\cdot \mathbf{n}-g_1(t,\x) v(t,\x)=g(t,\x)\mbox{ on $\Sigma,$ }\\
 v(0,\x)=0 \mbox{ in $\Omega.$ }
 \end{cases}
 \end{equation}
 \begin{proposition}\label{probleme Lineaire}
For all map's $(k_1,g_1,k,g) \in \mL^{\infty}(0,T;\mL^{2}(\Omega))\times\mL^{\infty}(0,T;\mL^{4}(\Gamma))\times\mL^2(\mathcal{Q})\times\mL^{2}(\Sigma)$ such that $\lVert k_1\lVert_{\mL^{\infty}(0,T;\mL^{2}(\Omega))}:=\gamma_1$ and $\lVert g_1\lVert_{\mL^{\infty}(0,T;\mL^{4}(\Gamma))}:=\gamma_2,$ there exists a unique solution $v\in\mathcal{W}_{0}\cap\mC([0, T];\mL^{2}(\Omega))$ of the problem \eqref{linear problem} in the weak formulation sense such that, for all $(k,g)$ $resp.$ $(\tilde{k},\tilde{g})\in \mL^2(\mathcal{Q})\times\mL^2(\Sigma)$ which correspond to the solution $v_1$ $resp.$ $v_2$ of  \eqref{linear problem}, we have
 \begin{gather}
 \label{stablin}
 \begin{split}
  \lVert v_1 -v_2\lVert^{2}_{\mathcal{W}_{0}\cap\mC([0, T];\mL^{2}(\Omega))}\leq  C\Big( \lVert k-\tilde{k}\lVert^{2}_{\mL^2(\mathcal{Q})}+\lVert g-\tilde{g}\lVert^{2}_{\mL^{2}(\Sigma) }\Big).
 \end{split}
 \end{gather}
 \end{proposition}
 \begin{pf}
The proof of the existence of a solution uses standard techniques, notably the Faedo-Galerkin method, and some classical compactness results. 
The uniqueness and stability can be obtained by a classical analysis. We omit the details.
 \end{pf}
\subsection{Differentiability of the Map Solution}
In this subsection, we  first state and establish a differentiability result of the map solution $\mathcal{F}$ given by the following.
\begin{theorem}
\label{proposition Gdiff}
The solution operator $\mathcal{F}$  is continuously differentiable from $\mathcal{U}_{ad}$ to $\mathcal{\mathbf{U}}_{0}$ such that the derivative $\mathcal{F}'(\varphi): \phi\longmapsto w = \mathcal{F}'(\varphi)\cdot \phi$  at point $\varphi \in \mathcal{U}_{ad}$ is the unique solution be in $\mathcal{W}_{0}\cap \mC([0; T],\mL^{2}(\Omega))$ in the weak formulation sense of the following problem
\begin{equation}	
\begin{cases}
 \label{Gdiff problem}
	\displaystyle   \frac{\partial w}{\partial t}-\div(\mD\nabla{w})+\big(\mK_{1} {\rm{h}}'_{1}(\mathbf{u})-\mK_{2}+\alpha_{0}\varphi\big)w=-\alpha_{0}\phi\mathbf{u} \mbox{ in } \cal Q,  \\
\displaystyle   -\mD\nabla{ w}\cdot\mathbf{n}-\mK_{3}{\rm{h}}'_{2}(\mathbf{u})w=0 \mbox{ on $\Sigma,$ }\\
  w(0,\x)=0 \mbox{ in $\Omega.$ }
\end{cases}
\end{equation}
Moreover, we have the estimates $(\forall\varphi\in \mathcal{U}_{ad}):$
\begin{gather}
 \lVert\mathcal{F}(\varphi+\phi)-\mathcal{F}(\varphi)-\mathcal{F}'(\varphi)\cdot \phi\rVert^{2}_{\mathcal{W}_{0}\cap \mC([0,T];\mL^{2}(\Omega))}\leq C\lVert \phi\rVert^{4}_{\mL^{\mathbf{p}_{c}}(\cal Q)}, 
\end{gather}
where $\phi$ is such that $\varphi+\phi\in \mathcal{U}_{ad}.$
 \begin{gather}
\lVert\mathcal{F}'(\varphi_{1})\cdot \phi-\mathcal{F}'(\varphi_{2})\cdot \phi\rVert^{2}_{\mathcal{W}_{0}\cap \mC([0,T];\mL^{2}(\Omega))}\leq C\lVert \varphi_{1}-\varphi_{2}\rVert^{2}_{\mL^{\mathbf{p}_{c}}(\cal Q)}\lVert \phi\rVert^{2}_{\mL^{\mathbf{p}_{c}}(\cal Q)} 
\end{gather}
for all $(\varphi_{1},\varphi_{2},\phi)\in\mathcal{U}_{ad}\times\mathcal{U}_{ad}\times\mL^{\mathbf{p}_{c}}(\cal Q).$
\end{theorem}
\begin{pf} 
The problem \eqref{Gdiff problem} is similar to the problem \eqref{linear problem}. Indeed, since $\mathbf{u}\in\mL^{\infty}(\mathcal{Q})$ and $\mathbf{p}_c\ge 2,$ then for all $\phi\in \mL^{\mathbf{p}_{c}}(\cal Q),$ we have 
  $-\alpha_{0}\phi\mathbf{u}\in \mL^2(\mathcal{Q})$ and that $-\alpha_{0}\phi\mathbf{u}$ can play the role of map $k$ in \eqref{linear problem}. In addition, according to the boundedness of $\mathbf{u}$ in $\cal Q,$ it follows that $\rm{h}'_1(\mathbf{u})$ is bounded in $\cal Q$, and $\mK_1\rm{h}'_1(\mathbf{u})-\mK_{2}+\alpha_{0}\varphi\in \mL^{\infty}(0,T;\mL^{2}(\Omega)).$ Thus, $\mK_1\rm{h}'_1(\mathbf{u})-\mK_{2}+\alpha_{0}\varphi$ can play the role of $k_1.$\\
Moreover, $\mK_3\text{h}'_2(\mathbf{u})\in \mL^{\infty}(0,T;\mL^{4}(\Gamma))$ (since $\mK_{3}\in\mL^{\infty}(\Sigma)$ and $\mathbf{u}\in\mathcal{H}_{0}$ with $\gamma_{\Gamma}(\mathbf{u})\in\mL^{4}(\Gamma)$ according to \eqref{opérateur trace}), then $\mK_3\text{h}'_2(\mathbf{u})$ can play the role of $g_1.$ Therefore, we deduce that the well-posedness and stability of the system \eqref{Gdiff problem} is a consequence of Proposition \ref{probleme Lineaire}, and we obtain
  \begin{gather}
  \label{estim w}
      \lVert w\lVert^{2}_{\mathcal{W}_{0}\cap\mC([0,T];\mL^{2}(\Omega))}\leq C \lVert \phi\rVert^{2}_{\mL^{\mathbf{p}_{c}}(\cal Q)}.
  \end{gather}
   We show that the solution operator $\mathcal{F}$ is Fréchet differentiable and the F-derivative is continuous with respect to the control $\varphi\in \mathcal{U}_{ad}.$
\begin{enumerate}
    \item  Prove that  $\mathcal{F}$ is Fréchet differentiable.\\
Let $(\varphi,\phi)\in\mathcal{U}_{ad}\times\mL^{\mathbf{p}_{c}}(\cal Q)$  such that $\varphi+\phi\in  \mathcal{U}_{ad}$ 
 and $v_{\phi}=\mathbf{u}_{\phi}-\mathbf{u}=\mathcal{F}(\varphi+\phi)-\mathcal{F}(\varphi)$ then  according to Corollary \ref{Existence and stability} we obtain
 \begin{gather}
 \label{estim vh}
     \lVert v_{\phi}\rVert^{2}_{\mathcal{\mathbf{U}}_0}\leq C\lVert \phi\rVert^{2}_{\mL^{\mathbf{p}_{c}}(\cal Q)}.
 \end{gather} 
Let $  z:=v_{\phi}-w$ then $ z$ satisfies the system
 \begin{gather}
  \label{problem vphi}
 \begin{cases}
\displaystyle  \frac{\partial z}{\partial t}\!-\!\div(\mD\nabla{ z})\!+\!\big(\mK_1\text{h}'_{1}(u_{m_1})\!-\!\mK_{2}\!+\!\alpha_{0}\varphi\big)z=-\alpha_{0}\phi v_{\phi}\!-\!\mK_{1}\big(\text{h}'_{1}(u_{m_1})\!-\!\text{h}'_{1}(\mathbf{u})\big)\phi w \mbox{ in $\cal Q,$ }\\
\displaystyle -\mD\nabla{z}\cdot\mathbf{n}-\mK_{3}\text{h}'_{2}(u_{m_2})z=\mK_{3}\big(\text{h}'_{2}(u_{m_2})-\text{h}'_{2}(\mathbf{u})\big)v_{\phi} w \mbox{ on $\Sigma,$ }\\
 \displaystyle  z(0,\cdot)=0 \mbox{ in $\Omega,$ }
 \end{cases}
 \end{gather}
 where, from the mean value theorem, $\text{h}'_{1}(u_{m1})v_{\phi}=\text{h}_{1}(\mathbf{u}_{\phi})-\text{h}_{1}(\mathbf{u})$ with $u_{m_1}(t,\x)\in[\mathbf{u}_{\phi}(t,\x);\mathbf{u}(t,\x)]$ on $\mathcal{Q}$ and $\text{h}'_{2}(u_{m_2})v_{\phi}=\text{h}_{2}(\mathbf{u}_{\phi})-\text{h}_{2}(\mathbf{u})$ with $u_{m_2}(t,\x)\in[\mathbf{u}_{\phi}(t,\x);\mathbf{u}(t,\x)]$ on $\Sigma$ (since $\text{h}_{i}$ for $i=1,2$ are functions of class $\mC^{2}$ because $\rho_i\ge 2).$\\
The problem \eqref{problem vphi} is well-posed according to Proposition \ref{probleme Lineaire} then, multiplying its first equation by $z,$ integrating on $\cal Q,$ we have
\begin{gather}\label{nik0}
\begin{split}
  \frac{\d }{2\d t}&\lVert  z(t,\cdot)\rVert^{2}_{\mL^{2}(\Omega)}+\int_{\Omega} \mD|\nabla{  z}|^2\d\x =-
     \int_{\Omega}\Big(\mK_1{\rm{h}}'_{1}(u_{m_1})-\mK_{2}+\alpha_{0}\varphi\Big)z^{2}\d\x-\int_{\Omega}\alpha_{0}\phi v_{\phi}z\d\x\\
  & -\!\int_{\Omega}\!\!\mK_{1}\Big({\text{h}}'_{1}(u_{m_1})-\text{h}'_{1}(\mathbf{u})\Big)\phi w z\d\x\!-\!\int_{\Gamma}\!\!\mK_{3}\Big(\text{h}'_{2}(u_{m_2})\!-\!{\rm{h}}'_{2}(\mathbf{u})\Big)v_{\phi} wz\d\x
\!-\!\int_{\Gamma}\!\!\mK_{3}{\rm{h}}'_{2}(u_{m_2})z^2 \d\x
\\
&\leq C_1\!\!\int_{\Omega}\!z^{2}\d\x+\alpha_{0}\!\!\int_{\Omega}\!\phi v_{\phi}z\d\x+C_2\!\!\int_{\Omega}\!\phi wz\d\x+C_3\!\!\int_{\Gamma}\!v_{\phi}wz\d\x+C_4\!\!\int_{\Gamma}z^2\!\d\x.
  \end{split}
\end{gather}
We consider separately the terms on the right-hand side of \eqref{nik0}.\\
Using Hölder's and Young’s inequality, and \eqref{estim vh}, and the fact that $\mathbf{p}_c\ge 2,$ we have
\begin{gather}\label{nik1}
\begin{split}
\alpha_{0}\!\!\int_{\Omega}\!\!\phi v_{\phi}z\d\x &\leq C\lVert\phi\rVert_{\mL^{2}(\Omega)}\lVert v_{\phi}\rVert_{\mL^{4}(\Omega)}\lVert z\rVert_{\mL^{4}(\Omega)}\\
&\leq C\lVert\phi\rVert_{\mL^{\mathbf{p}_c}(\Omega)}\lVert\phi\rVert_{\mL^{\mathbf{p}_c}(\cal Q)}\lVert z\rVert_{\mL^{4}(\Omega)} \mbox{ (since $v_{\phi}\in\mathcal{\mathbf{U}}_{0}\subset\mL^{\infty}(0, T;\mathcal{H}_{0})$) }\\
&\leq C\big(\lVert z\rVert^{2}_{\mL^{4}(\Omega)}+\lVert\phi\rVert^{2}_{\mL^{\mathbf{p}_c}(\Omega)}\lVert\phi\rVert^{2}_{\mL^{\mathbf{p}_c}(\cal Q)}\big).
\end{split}
\end{gather}
Again, using Hölder's and Young’s inequality and \eqref{estim w}, we obtain
\begin{gather}\label{nik2}
\begin{split}
\alpha_{0}\!\!\int_{\Omega}\!\!\phi w z\d\x\leq C\big(\lVert z\rVert^{2}_{\mL^{4}(\Omega)}+\lVert\phi\rVert^{2}_{\mL^{\mathbf{p}_c}(\Omega)}\lVert\phi\rVert^{2}_{\mL^{\mathbf{p}_c}(\cal Q)}\big),
\end{split}
\end{gather}
then using \eqref{estim vh}-\eqref{estim w}, we have
\begin{gather}\label{nik3}
\begin{split}
\int_{\Gamma}\!\!v_{\phi} w z\d\x\leq C\big(\lVert z\rVert^{2}_{\mL^{2}(\Gamma)}+\lVert\phi\rVert^{4}_{\mL^{\mathbf{p}_c}(\cal Q)}\big).
\end{split}
\end{gather}
For the last term, using Gagliardo-Nirenberg's and Young's inequality, we obtain ($\forall\eta>0$):  
\begin{gather}\label{nirberg}
\begin{split}
     \lVert   z\lVert^{2}_{\mL^4(\Omega)} \leq\frac{3\eta C }{4}\lVert \nabla{  z}\lVert^{2}_{\mL^{2}(\Omega)}+\frac{C}{4\eta}\lVert   z\lVert^{2}_{\mL^{2}(\Omega)}\\
     \lVert  z\lVert^{2}_{\mL^2(\Gamma)} \leq \frac{\eta C}{2}\lVert \nabla{  z}\lVert^{2}_{\mL^{2}(\Omega)}+\frac{C}{2\eta}\lVert z\lVert^{2}_{\mL^{2}(\Omega)}.
     \end{split}
\end{gather}
Therefore, by replacing the estimates \eqref{nirberg}-\eqref{nik3}-\eqref{nik2}-\eqref{nik1} in \eqref{nik0}, we obtain ($\forall\eta>0$):
\begin{gather*}
\begin{split}
   \frac{\d }{2\d t}\lVert z(t,\cdot)\rVert^{2}_{\mL^{2}(\Omega)}\!+\!\nu_{0}\!\!\int_{\Omega}\!\!|\nabla{  z}|^2 \leq  \frac{5\eta C }{4}\lVert \nabla{z}\lVert^{2}_{\mL^{2}(\Omega)}\!+\!\frac{3C}{4\eta}\lVert z\lVert^{2}_{\mL^{2}(\Omega)}+C\lVert \phi\rVert^{4}_{\mL^{\mathbf{p}_{c}}(\cal Q)}\!+\!C\lVert\phi\rVert^{2}_{\mL^{\mathbf{p}_c}(\Omega)}\lVert\phi\rVert^{2}_{\mL^{\mathbf{p}_c}(\cal Q)}.
\end{split}
\end{gather*}
Choosing $\eta>0$ such that  $\nu_0 -\frac{5\eta C }{4}=\frac{\nu_{0}}{2},$ $i.e.,$ $\eta=\frac{2 \nu_0 }{5C}>0,$ and according to Grönwall's Lemma with ($\mathbf{p}_c\ge 2$), we have 
\begin{gather*}
\begin{split}
   \lVert  z\rVert^{2}_{\mL^{\infty}(0,T;\mL^{2}(\Omega))}+\nu_{0}\int_{\cal Q}|\nabla{  z}|^2\d\x\d t &\leq C\left(\lVert \phi\rVert^{4}_{\mL^{\mathbf{p}_{c}}(\cal Q)}+\lVert \phi\rVert^{2}_{\mL^{\mathbf{p}_{c}}(\cal Q)}\int^{T}_{0}\lVert \phi\rVert^{2}_{\mL^{\mathbf{p}_{c}}(\Omega)}\d t\right)\\
   &\leq C\left(\lVert \phi\rVert^{4}_{\mL^{\mathbf{p}_{c}}(\cal Q)}+\lVert \phi\rVert^{4}_{\mL^{\mathbf{p}_{c}}(\cal Q)}\right).
\end{split}
\end{gather*}
By similar processing, we get
\begin{gather*}
\begin{split}
   \lVert\frac{\partial  z}{\partial t}\rVert^{2}_{\mL^{2}(0,T;\mathcal{H}_{0}')} \leq C\lVert \phi\rVert^{4}_{\mL^{\mathbf{p}_{c}}(\cal Q)}.
\end{split}
\end{gather*}
Thus, we obtain
\begin{gather*}
\lVert   z\rVert^{2}_{\mathcal{W}_{0}\cap \mC([0; T],\mL^{2}(\Omega))}\leq C\lVert \phi\rVert^{4}_{\mL^{\mathbf{p}_{c}}(\cal Q)},
\end{gather*}
  $i.e.,$
\begin{gather*}
 \lVert\mathcal{F}(\varphi+\phi)-\mathcal{F}(\varphi)-\mathcal{F}'(\varphi)\cdot \phi\rVert^{2}_{\mathcal{W}_{0}\cap \mC([0; T],\mL^{2}(\Omega))}\leq C\lVert \phi\rVert^{4}_{\mL^{\mathbf{p}_{c}}(\cal Q)}, \hspace{0.02cm} \forall (\varphi,\phi)\in \mathcal{U}_{ad}\times {\mL^{\mathbf{p}_{c}}(\cal Q)}.
\end{gather*}
\item  We prove that the F-derivative $\mathcal{F}'$ is continuous.\\
   Define $\mathbf{u}_{1}=\mathcal{F}(\varphi_{1})$, $\mathbf{u}_{2}=\mathcal{F}(\varphi_{2})$, and let $\mathbf{u}=\mathbf{u}_{1}-\mathbf{u}_{2}$. Also, define $w_{1}=\mathcal{F}'(\varphi_{1})\cdot \phi$, $w_{2}=\mathcal{F}'(\varphi_{2})\cdot \phi$ and let $\tilde{w}=w_{1}-w_{2}$ for all $\phi\in\mL^{\mathbf{p}_{c}}(\mathcal{Q})$, with $(\varphi_{1},\varphi_{2})\in\mathcal{U}_{ad}\times \mathcal{U}_{ad}.$ Then, $\tilde{w}$ verifies the following
 \begin{gather}	
 \label{problem wbar}
\begin{cases}
 \frac{\partial\tilde{w}}{\partial t}\!-\!\div(\mD\nabla{\tilde{w}})\!+\!\big(\!\mK_{1}\text{h}'_{1}(\mathbf{u}_{1})\!\!-\!\!\mK_{2}\!+\!\alpha_{0}\varphi_{1}\!\big)\tilde{w}\!=\!-\!\alpha_{0}\phi\mathbf{u}\!\!-\!\!\mK_{1}\text{h}''_{1}(u_{m_3})\mathbf{u}w_{2}\!+\!\alpha_{0}(\varphi_1\!\!-\!\!\varphi_2)w_{2}\mbox{ \!in\! $\cal Q,$ } \\
- \mD\nabla{\tilde{w}}\cdot\mathbf{n}-\mK_{3}\text{h}'_{2}(\mathbf{u}_{1})\tilde{w}=\mK_{3}\text{h}''_{2}(u_{m_4})\mathbf{u}w_{2} \mbox{ on $\Sigma,$ }\\
\tilde{w}(0,\x)=0 \mbox{ in $\Omega,$ }
\end{cases}
\end{gather} 
where, from the mean value theorem again, $\text{h}''_{1}(u_{m_3})(\mathbf{u}_1-\mathbf{u}_2)=\text{h}'_{1}(\mathbf{u}_1)-\text{h}'_{1}(\mathbf{u}_2)$ with $u_{m_3}(t,\x)\in[\mathbf{u}_{1}(t,\x);\mathbf{u}_2(t,\x)]$ on $\mathcal{Q}$ and $\text{h}''_{2}(u_{m_4})(\mathbf{u}_1-\mathbf{u}_2)=\text{h}'_{2}(\mathbf{u}_1)-\text{h}'_{2}(\mathbf{u}_2)$ with $u_{m_4}(t,\x)\in[\mathbf{u}_{1}(t,\x);\mathbf{u}_2(t,\x)]$ on $\Sigma$ (since $\text{h}_{i}$ for $i=1,2$ are functions of class $\mC^{2}$ because $\rho_i\ge 2$).\\
The well-posedness of the problem \eqref{problem wbar} is a consequence of Proposition \ref{probleme Lineaire} and Corollary \ref{Existence and stability}. Indeed, setting
\begin{gather*}
\begin{split}
 &k=-\alpha_{0}\phi\mathbf{u}
 -\mK_{1}\text{h}''_{1}(u_{m_3})\mathbf{u}w_{2}+\alpha_{0}(\varphi_1-\varphi_2)w_{2}\in \mL^{2}(\mathcal{Q}),\\
&    k_{1}=\mK_{1}\text{h}'_{1}(\mathbf{u}_{1})-\mK_{2}+\alpha_{0}\varphi_{1}\in \mL^{\infty}(0,T;\mL^{2}(\Omega)),\\
&    g_{1}=\mK_{3}\text{h}'_{2}(\mathbf{u}_{1})_{\Gamma}\in \mL^{\infty}(0,T;\mL^{4}(\Gamma)),\\
 &   g=\mK_{3}\text{h}''_{2}(u_{m_4})\mathbf{u}{w_{2}}_{\Gamma}\in \mL^{2}(\Sigma),
    \end{split}
\end{gather*} 
and according to the Proposition \ref{Existence and stability} and the estimate \eqref{estim w}, we have
\begin{gather*}
    \lVert g\rVert^{2}_{\mL^{2}(\Sigma)}\leq C \lVert \varphi_{1}-\varphi_{2}\rVert^{2}_{\mathcal{U}_{ad}}\lVert \phi\rVert^{2}_{\mL^{\mathbf{p}_{c}}(\cal Q)},\\ 
    \lVert k\rVert^{2}_{\mL^{2}(\mathcal{Q})}\leq C\lVert \varphi_{1}-\varphi_{2}\rVert^{2}_{\mathcal{U}_{ad}}\lVert \phi\rVert^{2}_{\mL^{\mathbf{p}_{c}}(\cal Q)}.
\end{gather*}
 Thus, we can deduce that the system \eqref{problem wbar} admits a unique solution in the weak formulation sense such that
 \begin{gather*}
  \lVert\tilde{w}\rVert^{2}_{\mathcal{W}_{0}\cap \mC([0,T];\mL^{2}(\Omega))}\leq C\lVert \varphi_{1}-\varphi_{2}\rVert^{2}_{\mathcal{U}_{ad}}\lVert \phi\rVert^{2}_{\mL^{\mathbf{p}_{c}}(\cal Q)},
\end{gather*}
$i.e.,$
\begin{gather*}
  \lVert\mathcal{F}'(\varphi_{1})\cdot\phi-\mathcal{F}'(\varphi_{2})\cdot\phi\rVert^{2}_{\mathcal{W}_{0}\cap \mC([0,T];\mL^{2}(\Omega))}\leq C\lVert \varphi_{1}-\varphi_{2}\rVert^{2}_{\mathcal{U}_{ad}}\lVert \phi\rVert^{2}_{\mL^{\mathbf{p}_{c}}(\cal Q)}.
\end{gather*}
 \end{enumerate}
\end{pf}
\subsection{First-order Optimality Conditions}
The following discussion aims to characterize the qualification conditions of the state constraint \eqref{state constraint} and formulate the optimality conditions of problem $(\P)$. For this, we introduce the generalized Lagrangian associated with problem $(\P)$ and define the set of Lagrange multipliers.\\
 The generalized Lagrangian for the problem $(\P)$ is formulated by choosing a multiplier for the  problem, denoted $\tilde{\mathbf{u}} \in \mL^2(\mathcal{Q})$, with $\mathbf{u} = \mathcal{F}(\varphi),$ the parameter $\beta \in \mathbb{R}_{+}$ and $\mu \in \mathcal{K}^{-}$ (see, e.g., \cite[Sec. 3.1.2]{BS2000}).
\begin{gather}\label{lagrangiang}
    \begin{split}
    \mathcal{L}(\beta,\tilde{\mathbf{u}},\mu,\varphi,\mathbf{u})=\beta J(\varphi)+\!\!\int_{\cal Q}\!\!\tilde{\mathbf{u}}\Big(\div(\mD\nabla \mathbf{u})-\mK_1 \text{h}_1(\mathbf{u}) +\mK_{2}\mathbf{u}-\alpha_{0}\varphi\mathbf{u}-\frac{\partial\mathbf{u}}{\partial t}\Big)\d \x \d t 
    +\!\!\int^{T}_{0}\!\!\mathcal{G}(\varphi)\d\mu(t).
    \end{split}
\end{gather}
\begin{definition} 
The set of generalized Lagrangian multipliers associated with $\varphi_* \in\mathcal{U}_{ad}$ is given by
\begin{gather*}
   \Lambda_{g}(\varphi_*):= \bigg{\{}(\beta,\mu_{*})\in \RR_{+}\times N_{\mathcal{K}}(\mathcal{G}(\varphi_*));\hspace{0.1cm}  (\beta,\mu_{*})\ne (0,0);\hspace{0.1cm}  -\frac{\partial \mathcal{L}}{\partial \varphi}(\beta,\tilde{\mathbf{u}}_{*},\mu_{*},\varphi_{*},\mathbf{u}_{*})\in N_{\mathcal{U}_{ad}}(\varphi_*)\bigg\},
\end{gather*}
and $\Lambda_{s}(\varphi_*)$ and $\Lambda_{l}(\varphi_*)$ are respectively the set of singular multipliers and Lagrange multipliers at $\varphi_*:$
\begin{gather*}
\begin{split}
  & \Lambda_{s}(\varphi_*):=\bigg{\{}\mu_{*}\in N_{\mathcal{K}}(\mathcal{G}(\varphi_*));\hspace{0.1cm} \mu_{*}\ne 0;\hspace{0.1cm}  -\frac{\partial \mathcal{L}}{\partial \varphi}(0,\tilde{\mathbf{u}}_{*},\mu_{*},\varphi_{*},\mathbf{u}_{*})\in N_{\mathcal{U}_{ad}}(\varphi_*)\bigg{\}},\\
&   \Lambda_{l}(\varphi_*):=\bigg{\{}\mu_{*}\in N_{\mathcal{K}}(\mathcal{G}(\varphi_*));\hspace{0.1cm} -\frac{\partial \mathcal{L}}{\partial \varphi}(1,\tilde{\mathbf{u}}_{*},\mu_{*},\varphi_{*},\mathbf{u}_{*})\in N_{\mathcal{U}_{ad}}(\varphi_*)\bigg{\}}.  
\end{split}
\end{gather*}
\end{definition}
Since the subset $\mathcal{K}$ is a convex cone, the condition $\mu_{*} \in N_{\mathcal{K}}(\mathcal{G}(\varphi_*))$ can be rewritten as
\begin{gather}
\label{equation de mualpha}
    \mathcal{G}(\varphi_*)\in \mathcal{K},\hspace{0.1cm} \int^{T}_{0}\mathcal{G(\varphi_*)}\d\mu_{*}(t)=0 \mbox{ with }\hspace{0.1cm}  \mu_{*}\in \mathcal{K}^{-}.
\end{gather}
\paragraph{Constraint Qualification Condition.}\hspace{-0.2cm}The qualification condition of \eqref{state constraint} is the condition implying that Lagrange multipliers associated with a solution to problem $(\P)$ exist. According to Theorem \ref{proposition Gdiff}, the map $\mathcal{G}$ is continuously differentiable because it is the composition of continuously differentiable functions. Given that $\mathcal{K}$ is a convex cone with a non-empty interior, we say that the state constraint \eqref{state constraint} is qualified at the point $\varphi_*$ if the following Robinson constraint qualification condition is satisfied (\cite[Sec. 3.2]{BonSh1998}): 
\begin{gather}
\label{regularité contrainte}
\begin{cases}
\mbox{ There exists $\varphi\in\mathcal{U}_{ad}$ such that $\Phi:=\varphi-\varphi_*$ verifies }\\ \mathcal{G}(\varphi_*)+D\mathcal{G}(\varphi_*)\cdot\Phi\in int(\mathcal{K}),
     \end{cases}
\end{gather}
where $int(\mathcal{K})$ is the interior of $\mathcal{K}$ and $D\mathcal{G}$ is the derivative of $\mathcal{G}.$\\

The condition \eqref{regularité contrainte} is equivalent to the following qualification condition:
\begin{gather}
   \begin{cases}
   \label{constraint qualif condition}
   \mbox{ There exists  $\varphi\in \mathcal{U}_{ad}$ such that } \\
  \varpi(t):=\mathcal{G}(\varphi_*)(t)+\int_{\Omega}\mathbf{u}_*(t,\x)w_{\Phi}(t,\x)\d\x<0, \forall t\in(0; T) \\
   \mathbf{u}_*=\mathcal{F}(\varphi_*) \mbox{ and $w_{\Phi}$ solution of \eqref{Gdiff problem} for $\Phi=\varphi-\varphi_*$ }.
\end{cases}
\end{gather}
To characterize the optimality conditions, we assume that the variable $\tilde{u}$ is the solution of the following adjoint system:
\begin{gather}\label{adjsys}
   \begin{cases}
 - \frac{\partial\tilde{{u}}_{}}{\partial t}-\div(\mD\nabla{\tilde{{u}}_{}})+\big(\mK_1\text{h}'_{1}(\mathbf{u}_{})-\mK_{2}+\alpha_{0}\varphi_{}\big)\tilde{{u}}_{}=\mathbf{u}_{}+ 2\mathbf{u}\d\mu(t) \mbox{ in }\cal Q,\\
-\mD\nabla{\Tilde{u}_{}}\cdot \mathbf{n}-\mK_3\text{h}'_{2}(\mathbf{u}_{})\tilde{{u}}_{}=0 \mbox{ on } \Sigma, \\
\tilde{{u}}_{}(t=T,\cdot)=0 \mbox{ in } \Omega.   
   \end{cases}
 \end{gather}
 The well-posedness of the problem \eqref{adjsys} is given by the following proposition.
 \begin{proposition}\label{welpadjsys}
Let $\mu\in BV(0, T)$ such that $\d\mu=:\omega\in M([0, T]),$ and let $\mathbf{u}=\mathcal{F}(\varphi)\in\mathcal{\mathbf{U}}_{0}\cap\mL^{\infty}_{+}(\mathcal{Q}),$ where $\varphi\in\mathcal{U}_{ad}.$ Then, the problem \eqref{adjsys} has a unique solution $\tilde{u}\in\mL^{\infty}([0, T];\mL^{2}(\Omega))\cap\mL^{2}([0, T];\mathcal{H}_{0}),$ such that 
 \begin{gather}\label{Bbornetilde3} 
 \lVert\tilde{u}\lVert_{\mL^{2}([0, T];\mathcal{H}_{0})\cap\mL^{\infty}([0, T];\mL^{2}(\Omega))}\leq C\Big(\lVert\mathbf{u}\lVert_{\mL^{\infty}([0, T];\mL^{2}(\Omega))}+\lVert\mathbf{u}\lVert_{\mL^{\infty}([0, T];\mL^{2}(\Omega))}\lVert \omega\lVert_{M([0, T])}\Big).
 \end{gather}
 \end{proposition}
 \begin{pf} 
 The proof of this proposition can be obtained using the same technique as that developed in \cite{casas1997pontryagin}.\\ Since the right side of the problem \eqref{adjsys} is a measure, the existence of a solution can be obtained by defining an approximate system for which we know the existence of the solution. Then, we show that the approximate solution is bounded and, therefore, converges weakly towards a limit, which is a solution to the adjoint problem. Here, we only show that the approximate solution is bounded and weakly convergent towards a limit $\tilde{u}$ verifying \eqref{Bbornetilde3}. The rest of the proof uses a classical analysis, which we omit.\\
 Let $\mu\in BV(0, T)$ such that $\d\mu=\omega\in M([0, T]),$ and let $\varphi\in\mathcal{U}_{ad}$ with $\mathbf{u}=\mathcal{F}(\varphi)\in\mathcal{\mathbf{U}}_{0}\cap\mL^{\infty}_{+}(\mathcal{Q})$ solution of \eqref{Eqstate} corresponding to the control $\varphi.$
  We choose a sequence $(\omega_n)_{n}\subset\mC^{\infty}_{c}([0, T])$ that converges weakly star to $\omega$ in $M([0, T])$ in the distribution sense and satisfies
\begin{gather} \label{omom}
\lVert\omega_{n}\lVert_{\mL^{1}([0, T])}\leq\lVert\omega\lVert_{M([0, T])}.
 \end{gather}
Now, let $\tilde{u}_{n}\in\mL^{2}([0, T];\mathcal{H}_{0})$ such that 
 \begin{gather}\label{adjsys2}
   \begin{cases}
 - \frac{\partial\tilde{{u}}_{n}}{\partial t}-\div(\mD\nabla{\tilde{{u}}_{n}})+\big(\mK_1\text{h}'_{1}(\mathbf{u}_{})-\mK_{2}+\alpha_{0}\varphi_{}\big)\tilde{{u}}_{n}=\mathbf{u}_{}+ 2\mathbf{u}_{}\omega_{n} \mbox{ in }\cal Q,\\
-\mD\nabla{\Tilde{u}_{n}}\cdot \mathbf{n}-\mK_3\text{h}'_{2}(\mathbf{u}_{})\tilde{{u}}_{n}=0 \mbox{ on } \Sigma, \\
\tilde{{u}}_{n}(t=T,\cdot)=0 \mbox{ in } \Omega.   
   \end{cases}
 \end{gather}
 By reversing time $t=T-t$ in \eqref{adjsys2}, we obtain a linear system with zero initial condition, and according to the regularity of $\mathbf{u}$, using Proposition \ref{probleme Lineaire}, the system \eqref{adjsys2} admits a unique solution in $\mathcal{W}_{0}\cap \mC([0, T];\mL^{2}(\Omega))$ in the sense of the weak formulation.\\
 Let $f\in\mC^{\infty}_{c}(\mathcal{Q}),$ and denote by $v_{f}$ the solution in $\mathcal{W}_{0}\cap \mC([0, T];\mL^{2}(\Omega))$ (according to a similar result as Proposition \ref{probleme Lineaire}, and using $\mC^{\infty}_{c}(\mathcal{Q})\subset\mL^{2}(0, T;\mathcal{H}'_{0}))$ is dense) of the following problem
 \begin{gather}\label{pbpadj}
  \begin{cases}
 \frac{\partial v}{\partial t}-\div(\mD\nabla{v})+\big(\mK_1\text{h}'_{1}(\mathbf{u}_{})-\mK_{2}+\alpha_{0}\varphi_{}\big)v=f \mbox{ in }\cal Q,\\
-\mD\nabla{v}\cdot \mathbf{n}-\mK_3\text{h}'_{2}(\mathbf{u}_{})v=0 \mbox{ on } \Sigma, \\
v(t=0,\cdot)=0 \mbox{ in } \Omega,  
   \end{cases}
 \end{gather}
such that
 \begin{gather}\label{vfest}
 \lVert v_{f} \lVert_{\mathcal{W}_{0}\cap\mL^{\infty}([0, T];\mL^{2}(\Omega))}\leq C\lVert f\lVert_{\mL^{2}(0,T; \mathcal{H}'_{0})}.
   \end{gather} 
 Multiplying \eqref{pbpadj} by $\tilde{{u}}_{n}$, and integrating by parts over $\cal Q$ (by using Green's formula), and using \eqref{adjsys2}, we obtain
 \begin{gather}\label{futildeest}
 \begin{split}
 \int_{\mathcal{Q}}f\tilde{u}_{n}\d\x\d t &=\int_{\mathcal{Q}}\Big(\frac{\partial v}{\partial t}-\div(\mD\nabla{v})\Big)\tilde{u}_{n}\d\x\d t+\int_{\mathcal{Q}}\Big(\mK_1\text{h}'_{1}(\mathbf{u}_{})-\mK_{2}+\alpha_{0}\varphi\Big)v\tilde{u}_{n}\d\x\d t\\
 &=\int_{\mathcal{Q}}\Big(-\frac{\partial \tilde{u}_{n}}{\partial t}-\div(\mD\nabla{\tilde{u}_{n}})+(\mK_1\text{h}'_{1}(\mathbf{u}_{})-\mK_{2}+\alpha_{0}\varphi)\tilde{u}_{n}\Big)v\d\x\d t\\
 &=\int_{\mathcal{Q}}(\mathbf{u}_{}+ 2\mathbf{u}_{}\omega_{n})v\d\x\d t\\
 &\leq C\bigg(\!\!\lVert\mathbf{u}\lVert_{\mL^{\infty}(0,T; \mL^{2}(\Omega))}\lVert v \lVert_{\mL^{\infty}(0,T; \mL^{2}(\Omega))}+\lVert\omega_{n}\lVert_{\mL^{1}([0, T])}\lVert\mathbf{u}\lVert_{\mL^{\infty}(0,T; \mL^{2}(\Omega))}\lVert v \lVert_{\mL^{\infty}(0,T; \mL^{2}(\Omega))}\!\!\bigg).
 \end{split}
 \end{gather}
According to \eqref{omom} and \eqref{vfest}, we have
 \begin{gather}\label{bornetilde1} 
 \int^{T}_{0}\langle f, \tilde{u}_{n}\rangle_{\mathcal{H}'_{0},\mathcal{H}_{0}}\d t\leq C\lVert f\lVert_{\mL^{2}(0,T; \mathcal{H}'_{0})}\Big(\lVert\mathbf{u}\lVert_{\mL^{\infty}(0,T; \mL^{2}(\Omega))}+\lVert\mathbf{u}\lVert_{\mL^{\infty}(0,T; \mL^{2}(\Omega))}\lVert \omega\lVert_{M([0, T])}\Big).
 \end{gather}
 Moreover, using standard results for linear parabolic systems, we can obtain
  \begin{gather}\label{vfest2}
 \lVert v \lVert_{\mC([0, T];\mL^{2}(\Omega))}\leq C\lVert f\lVert_{\mL^{1}(0,T; \mL^{2}(\Omega))}.
   \end{gather} 
  So according to \eqref{omom}, \eqref{futildeest} and \eqref{vfest2}, we also have
 \begin{gather} \label{bornetilde2}
 \int^{T}_{0}\langle f, \tilde{u}_{n}\rangle_{\mL^{2}(\Omega),\mL^{2}(\Omega)}\d t\leq C\lVert f\lVert_{\mL^{1}(0,T; \mL^{2}(\Omega))}\Big(\lVert\mathbf{u}\lVert_{\mL^{\infty}(0,T; \mL^{2}(\Omega))}+\lVert\mathbf{u}\lVert_{\mL^{\infty}(0,T; \mL^{2}(\Omega))}\lVert \omega\lVert_{M([0, T])}\Big).
 \end{gather}
From \eqref{bornetilde1}, we deduce that the sequence $(\tilde{u}_n)_{n}$ is bounded in $\mL^{2}(0, T;\mathcal{H}_{0}),$ and according to \eqref{bornetilde2}, we have the boundedness of $(\tilde{u}_n)_{n}$ in $\mL^{\infty}(0, T;\mL^{2}(\Omega)).$ Therefore, extracting a subsequent, we can assume that $\tilde{u}_{n}\rightharpoonup\tilde{u}$ weakly in $\mL^{2}(0, T;\mathcal{H}_{0}),$ we have
 \begin{gather}\label{bornetilde3} 
 \lVert\tilde{u}\lVert_{\mL^{2}(0,T; \mathcal{H}_{0})\cap\mL^{\infty}(0,T; \mL^{2}(\Omega))}\leq C\Big(\lVert\mathbf{u}\lVert_{\mL^{\infty}(0,T; \mL^{2}(\Omega))}+\lVert\mathbf{u}\lVert_{\mL^{\infty}(0,T; \mL^{2}(\Omega))}\lVert \omega\lVert_{M([0, T])}\Big).
 \end{gather}
 The rest of the proof uses a classical analysis, which we omit.
 \end{pf} 
Now, we can derive the first-order optimality conditions for the solutions of $(\P)$ by the following
\begin{theorem}
\label{theorem optimality}
   Let ${\varphi}_{*}\in\mathcal{U}_{ad}$ a solution of $(\P)$ and $\mathbf{u}_{*}=\mathcal{F}(\varphi_{*})$ such that the condition \eqref{constraint qualif condition} holds. Then there exist Lagrange multipliers  $(\tilde{\mathbf{u}}_{*},\mu_{*})\in\mL^{2}(0,T;\mathcal{H}_{0})\cap \mL^{\infty}(0,T;\mL^{2}(\Omega))\times BV(0, T)_{+},$ such that $\mu_{*}\ne 0,$ and that $\tilde{\mathbf{u}}_{*}=\tilde{\mathcal{F}}(\varphi_{*})$ is the solution of adjoint problem \eqref{optim utilde theo}
   \begin{gather}
   \begin{cases}
   \label{optim utilde theo}
 - \frac{\partial\tilde{\mathbf{u}}_{*}}{\partial t}-\div(\mD\nabla{\tilde{\mathbf{u}}_{*}})+\big(\mK_1\rm{h}'_{1}(\mathbf{u}_{*})-\mK_{2}+\alpha_{0}\varphi_{*}\big)\tilde{\mathbf{u}}_{*}=\mathbf{u}_{*}+ 2\mathbf{u}_{*}\d\mu_{*}(t) \mbox{ in }\cal Q,\\
-\mD\nabla{\Tilde{ \mathbf{u}}_{*}}\cdot \mathbf{n}-\mK_3\rm{h}'_{2}(\mathbf{u}_{*})\tilde{\mathbf{u}}_{*}=0 \mbox{ on } \Sigma, \\
\tilde{\mathbf{u}}_{*}(t=T,\cdot)=0 \mbox{ in } \Omega,   
   \end{cases}
 \end{gather}
 the  multiplier $\mu_{*}$ verifies
   \begin{gather}
   \label{complementary condition}
    \mathcal{G}(\varphi_{*})\in \mathcal{K}; \hspace{0.1cm} \int^{T}_{0}\mathcal{G(\varphi_{*})}\d\mu_{*}(t)=0; \hspace{0.1cm} \mu_{*}\in \mathcal{K}^{-}; \hspace{0.1cm} \mu_{*}(T)=0,
\end{gather}
 and we have the following optimality condition
   \begin{gather}
   \label{optimality cond utilde1}
     \big\langle\lambda|\varphi_*|^{\mathbf{p}_{c}-1}-\alpha_{0}\mathbf{u}_* \tilde{\mathbf{u}}_*;\varphi-\varphi_*\big\rangle_{\mL^{\mathbf{p}^{*}_{c}}(\mathcal{Q}),\mL^{\mathbf{p}_{c}}(\mathcal{Q})}\ge 0 \mbox{ for all $\varphi\in\mathcal{U}_{ad}$ }.
   \end{gather}
\end{theorem}
 \begin{pf}
 Let $\varphi_{*}\in\mathcal{U}_{ad}$ be an optimal control of the problem $(\P)$ and $\mathbf{u}_{*}=\mathcal{F}(\varphi_*)$ the solution of \eqref{Eqstate} corresponding to the control $\varphi_{*}$ such that the condition \eqref{constraint qualif condition} holds at $\varphi_{*}.$ Since the cone $\mathcal{K}$ has a non-empty interior, then according to \cite[Prop 3.18]{BS2000}, the set $\Lambda_{g}(\varphi_{*})$ of generalized multipliers is nonempty. Moreover, given that \eqref{constraint qualif condition} holds, then from \cite[Prop 3.16]{BS2000} there is no singular multiplier, that is, $\Lambda_{l}(\varphi_*)$ is nonempty. Therefore, there exists a Lagrange multiplier $\mu_{*}\in BV(0, T)$ such that
 \begin{gather}\label{condmu}
\mu_{*}\ne 0, \mbox{ and } \mu_{*}\in N_{\mathcal{K}}(\mathcal{G}(\varphi_{*})), 
 \end{gather}
  and that
   \begin{gather}\label{condL}
  -\frac{\partial \mathcal{L}}{\partial \varphi}(1,\tilde{\mathbf{u}}_{*},\mu_{*},\varphi_{*},\mathbf{u}_{*})\in N_{\mathcal{U}_{ad}}(\varphi_*).
  \end{gather}
Since the subset $\mathcal{K}$  is a closed convex cone, then the relation \eqref{condmu} implies that $\mu_{*}\ne 0$ and verifies
\begin{gather*}
  \int^{T}_{0}\mathcal{G(\varphi_{*})}\d\mu_{*}(t)=0; \hspace{0.1cm} \mu_{*}\in \mathcal{K}^{-}, \hspace{0.1cm} \mu_{*}(T)=0,  \hspace{0.1cm}\mathcal{G}(\varphi_{*})\in \mathcal{K}. 
\end{gather*}
Now let's calculate $\frac{\partial \mathcal{L}}{\partial \varphi}(1,\tilde{\mathbf{u}}_{*},\mu_{*},\varphi_{*},\mathbf{u}_{*}).$\\
 Let $\phi\in \mL^{\mathbf{p}_{c}}(\cal{Q}),$ using $\mathbf{u}_{*}=\mathcal{F}(\varphi_{*})$ and the fact that $w=\mathcal{F}'(\varphi)\cdot \phi$ is the solution of problem \eqref{Gdiff problem}, we have 
\begin{gather}
\label{derL}
 \frac{\partial \mathcal{L}}{\partial \varphi}(1,\tilde{\mathbf{u}}_{*},\mu_* ,\varphi_*,\mathbf{u}_*)\cdot \phi=\int_{\cal Q}\mathbf{u}_*w\d\x\d t+\lambda\int_{\cal Q}|\varphi|^{\mathbf{p}_{c}-1} \cdot \phi\d\x\d t +2\int^{T}_{0}\int_{\Omega}\mathbf{u}_{*}w\d\x\d\mu_{*}(t).
\end{gather}
Let $\omega_{*}\in M([0, T])$ such that $\d\mu_{*}=:\omega_{*},$ then \eqref{derL} becomes
\begin{gather}\label{Lphi}
 \frac{\partial \mathcal{L}}{\partial \varphi}(1,\tilde{\mathbf{u}}_{*},\mu_* ,\varphi_*,\mathbf{u}_*)\cdot \phi=\int_{\cal Q}(\mathbf{u}_{*}+2\omega_{*}\mathbf{u}_{*})w\d\x\d t+\lambda\int_{\cal Q}|\varphi|^{\mathbf{p}_{c}-1} \cdot \phi\d\x\d t.
\end{gather}
 Multiplying \eqref{Gdiff problem} by $\tilde{{u}}$, where $\tilde{{u}}$ is a sufficiently regular function, and integrating by parts over $\cal Q$ (by using Green's formula), we obtain (we denote by $\mathbf{u} =\mathcal{F}(\varphi)$):
\begin{gather}\label{adjw}
\begin{split}	
	\displaystyle{ \int^{T}_{0}\int_{\Omega}\Big(-\frac{\partial \tilde{u}}{\partial t}-\div(\mD\nabla{\tilde{u}})+(\mK_{1} \text{h}'_{1}(\mathbf{u})-\mK_{2}+\alpha_{0}\varphi)\tilde{u}\Big)w\d\x\d t+\int_{\Omega}w(T)\tilde{u}(T)\d\x}\\
\displaystyle{	+\int^{T}_{0}\int_{\Gamma}\Big(\mD\nabla{\tilde{u}}\cdot\mathbf{n}+\mK_{3}\text{h}'_{2}(\mathbf{u})\tilde{u}\Big)w\d\x\d t=-\alpha_{0}\int^{T}_{0}\int_{\Omega}\mathbf{u}\tilde{u}\phi\d\x\d t.}
\end{split}
\end{gather}
Now, since $\tilde{u}$ is the solution of system \eqref{adjsys}, we obtain
\begin{gather*}
    \int_{\cal Q}(\mathbf{u}+2\omega\mathbf{u})w\d\x\d t=-\alpha_{0}\int_{\cal Q}\mathbf{u}\tilde{{u}}\phi \d\x\d t.
\end{gather*}
Then from \eqref{Lphi}, we have ($\forall \phi\in \mL^{\mathbf{p}_{c}}(\cal{Q})$): 
\begin{gather*}
\frac{\partial \mathcal{L}}{\partial \varphi}(1,\tilde{{u}},\mu,\varphi,\mathbf{u})\cdot \phi=\int_{\mathcal{Q}}(\lambda|\varphi|^{\mathbf{p}_{c}-1}-\alpha_{0}\mathbf{u}\tilde{{u}})\cdot\phi\d\x\d t. 
\end{gather*}
Therefore (since $\varphi_{*}$ is an optimal solution of $(\P)$, and using \eqref{conenormal} and \eqref{condL}):
 \begin{gather}
 \label{optimality cond utilde}
 \big\langle \lambda|\varphi_{*}|^{\mathbf{p}_{c}-1}-\alpha_{0}\mathbf{u}_{*}\tilde{\mathbf{u}}_{*};\varphi-\varphi_{*}\big\rangle_{\mL^{\mathbf{p}^{*}_{c}}(\mathcal{Q}),\mL^{\mathbf{p}_{c}}(\mathcal{Q})}\ge 0 \mbox{ for all $\varphi\in\mathcal{U}_{ad}$ },
\end{gather}
where $\tilde{\mathbf{u}}_{*}$ is the solution of the adjoint system \eqref{adjsys} corresponding to the  solution $\mathbf{u}_{*}=\mathcal{F}(\varphi_{*})$ of \eqref{Eqstate}.
 \end{pf}
\section{Application and Numerical Simulations}\label{setting4}
For the application of our approach, we consider the problem of the eradication of non-small cell lung cancer,  a highly malignant lung tumour frequently observed in smokers. More precisely, we numerical study the optimal treatment of two types of malignant lung tumours: one located in the left lobe of the lung and the other situated in the right lobe, near the bronchus, within a 2D domain of the lung.

The numerical resolution of optimality system \eqref{optim utilde theo}-\eqref{complementary condition}-\eqref{optimality cond utilde1} presents some difficulties, as the numerical construction of the measure $\mu_{*}$ in the adjoint system \eqref{optim utilde theo} and verifying \eqref{complementary condition} on singular arcs of the form $[\tau_{en}; \tau_{ex}]$ where the constraint \eqref{state constraint} is active, with $\tau_{en},$ $\tau_{ex}$ the entry and exit point. Some work suggested an approach based on the shooting method as in \cite{BHerm2008}. We attempted it, but it did not work for our problem. Then, to solve the constrained control problem $(\P)$ numerically, we use penalty methods (see e.g., \cite{ACH1997,BHK2000}). The penalization method consists of replacing the problem $(\P)$ with a series of unconstrained problems which are formed by adding some penalty functions to the initial objective function.
\subsection{Penalisation Approach} 
This approach consists of formally penalizing the state and control constraints, and solving the unconstrained control problem $(\overline{\P})$
\begin{gather}
    \begin{split}                             
    \label{control problem modif}
\inf_{ \varphi\in\mL^{\mathbf{p}_{c}}(\cal{Q})}\bar{J}(\varphi),
\end{split}
\end{gather}
where 
\begin{gather}
\label{Jtilde}
\begin{split}
 \bar{J}(\varphi)=J(\varphi)
 +\frac{\beta_{1}}{\mathbf{p}_{c}}\int_{\cal Q}(\mathbf{a}-\varphi)^{\mathbf{p}_{c}}_{+} \d\x\d t+\frac{\beta_{2}}{\mathbf{p}_{c}}\int_{\cal Q}(\varphi-\mathbf{b})^{\mathbf{p}_{c}}_{+}
 \d\x\d t
 +\frac{\beta_{3}}{4}\int^{T}_{0}(\mathcal{G}(\varphi))^{2}_{+}\d t
\end{split}
\end{gather}
with $\beta_1,\beta_2$ the prices to pay for the control constraint, and $\beta_3$ the price to pay for the state constraint.\\ 
 By the same minimizing sequence argument as in the proof of  Theorem \ref{existence solution of P}, we easily show that $(\bar{\P})$ admits at least one solution in $\mL^{\mathbf{p}_{c}}(\cal{Q}).$
To characterize the optimal control of $\overline{\P}$, we introduce the following adjoint problem corresponding to the problem \eqref{Eqstate} and the penalized cost \eqref{Jtilde} by
        \begin{gather}
   \begin{cases}
   \label{utilde penal}
 - \frac{\partial\tilde{v}_{}}{\partial t}-\div(\mD\nabla{\tilde{v}_{}})+(\mK_1\text{h}'_{1}(\mathbf{u})-\mK_{2}+\alpha_{0}\varphi_{})\tilde{v}_{}=\mathbf{u}_{}+ \beta_{3}\mathbf{u}_{}\mathcal{G}(\varphi)_{+} \mbox{ in }\cal Q,\\
-\mD\nabla{\Tilde{v}_{}}\cdot \mathbf{n}-\mK_3\text{h}'_{2}(\mathbf{u})\Tilde{v}_{}=0 \mbox{ on } \Sigma, \\
\tilde{v}_{}(T,\x)=0 \mbox{ in } \Omega,  
   \end{cases}
 \end{gather}
 where $\mathbf{u}=\mathcal{F}(\varphi)$ is the solution corresponding to $\varphi.$\\
Reversing time ($t:=T-t$), we can see that the linear system \eqref{utilde penal} is well-posed in the week formulation sense (Theorem \ref{probleme Lineaire}) and that the week solution $\tilde{v}$ is in $\mathcal{\mathbf{U}}_{0}\cap \mC([0; T],\mL^{2}(\Omega)).$
\begin{proposition}
\label{gradJtildeeps}
(i) The cost function $\Bar{J}_{}$ is continuously differentiable and its partial derivative is given by
     \begin{gather}
     \label{gradJtild}
        \nabla{\bar{J}}(\varphi)=\lambda|\varphi|^{\mathbf{p}_{c}-1} -\alpha_{0}\tilde{v}_{}\mathbf{u}-\beta_{1}(\mathbf{a}-\varphi)^{\mathbf{p}_{c}-1}_{+} +\beta_{2}(\varphi-\mathbf{b})^{\mathbf{p}_{c}-1}_{+},
     \end{gather}
  where $\tilde{v}$ verifies \eqref{utilde penal}.\\
(ii) If $\varphi_{*}\in\mL^{\mathbf{p}_{c}}(\cal{Q})$ is a solution of $(\bar{\P})$ and $\mathbf{u}_{*}=\cal{F}(\varphi_{*})$ is the associated optimal state, then
\begin{gather}\label{cequi0}
\lambda|\varphi_{*}|^{\mathbf{p}_{c}-1}-\alpha_{0}\mathbf{u}_{*}{\Tilde{v}}_{*}-\beta_{1}(\mathbf{a}-\varphi_{*})^{\mathbf{p}_{c}-1}_{+}+\beta_{2}(\varphi_{*}-\mathbf{b})^{\mathbf{p}_{c}-1}_{+}=0,
\end{gather}
where $\tilde{v}_{*}$ verifies 
        \begin{gather}
        \label{ceqi}
   \begin{cases}
 - \frac{\partial\tilde{v}_{*}}{\partial t}-\div(\mD\nabla{\tilde{v}_{*}})+(\mK_1{\rm{h}}'_1(\mathbf{u}_*)-\mK_{2}+\alpha_{0}\varphi_{*})\tilde{v}_{*}=\mathbf{u}_{*}+ \beta_{3}\mathbf{u}_{*}\mathcal{G}(\varphi_{*})_{+} \mbox{ in }\cal Q,\\
-\mD\nabla{\Tilde{v}_{*}}\cdot \mathbf{n}-\mK_3{\rm{h}}'_2(\mathbf{u}_{*})\Tilde{v}_{*}=0 \mbox{ on } \Sigma, \\
\tilde{v}_{*}(T,\x)=0 \mbox{ in } \Omega.  
   \end{cases}
 \end{gather}
 \end{proposition}
\begin{pf}
 (i)  From the Theorem \ref{proposition Gdiff}, the operator $\mathcal{F}$ is continuously differentiable. Then, $\bar{J}$ is continuously differentiable as a composition of continuously differentiable functions.\\
     Let $\phi$ be in $\mL^{\mathbf{p}_{c}}(\cal{Q})$ and $w=\mathcal{F}'(\varphi)\cdot \phi$  be a solution of \eqref{Gdiff problem}.
\begin{gather*}
 \frac{\partial \bar{J}_{}}{\partial \varphi}\cdot \phi=\int_{\mathcal{Q}} \Big(\!\lambda\varphi^{\mathbf{p}_{c}-1}-\beta_{1}(\mathbf{a}-\varphi)^{\mathbf{p}_{c}-1}_{+}+\beta_{2}(\varphi-\mathbf{b})^{\mathbf{p}_{c}-1}_{+}\!\Big)\phi\d\x\d t+\int_{\cal{Q}}\Big(\!\mathbf{u}+\beta_{3}\mathbf{u}\mathcal{G}(\varphi)_{+}\!\Big) w\d\x\d t. 
\end{gather*}
We calculate the term $\int_{\cal{Q}}\Big(\mathbf{u}+\beta_{3}\mathbf{u}\mathcal{G}(\varphi)_{+}\Big) w\d\x\d t.$\\
Multiplying by $w$ the first equation of  \eqref{utilde penal}, and integrating by parts over $\cal{Q}$, and using \eqref{Gdiff problem}, we obtain
\begin{gather*}
\int_{\cal{Q}}\Big(\mathbf{u}+\beta_{3}\mathbf{u}\mathcal{G}(\varphi)_{+}\Big) w\d\x\d t =-\alpha_{0}\int_{\cal{Q}}\mathbf{u}\tilde{v}_{}\phi\d\x\d t.
\end{gather*}
Then, we have (for all $\phi\in \mL^{\mathbf{p}_{c}}(\cal{Q})$):
\begin{gather*}
 \frac{\partial \bar{J}}{\partial \varphi}\cdot \phi=\int_{Q}\bigg(\lambda|\varphi|^{\mathbf{p}_{c}-1} -\alpha_{0}\tilde{v}_{}\mathbf{u}-\beta_{1}(\mathbf{a}-\varphi)^{\mathbf{p}_{c}-1}_{+} +\beta_{2}(\varphi-\mathbf{b})^{\mathbf{p}_{c}-1}_{+}\bigg)\cdot\phi\d\x\d t.
\end{gather*}
Since $\bar{J}$ is continuously differentiable, then
\begin{gather*}
\nabla{\bar{J}}(\varphi)=\lambda|\varphi|^{\mathbf{p}_{c}-1} -\alpha_{0}\tilde{v}_{}\mathbf{u}-\beta_{1}(\mathbf{a}-\varphi)^{\mathbf{p}_{c}-1}_{+} +\beta_{2}(\varphi-\mathbf{b})^{\mathbf{p}_{c}-1}_{+}.
\end{gather*}
(ii) Let $\varphi_{*}\in\mL^{\mathbf{p}_{c}}(\cal{Q})$ be a solution of the unconstrained problem $(\bar{\P})$ and $\mathbf{u}_{*}=\cal{F}(\varphi_{*})$ the associated optimal state. Since $\bar{J}$ is continuously differentiable, we can deduce that $\nabla{\bar{J}}(\varphi_{*})=0,$ and then from \eqref{gradJtild}
\begin{gather}
\lambda|\varphi_{*}|^{\mathbf{p}_{c}-1}-\alpha_{0}\mathbf{u}_{*}{\Tilde{v}}_{*}-\beta_{1}(\mathbf{a}-\varphi_{*})^{\mathbf{p}_{c}-1}_{+}+\beta_{2}(\varphi_{*}-\mathbf{b})^{\mathbf{p}_{c}-1}_{+}=0,
\end{gather}
where $\tilde{v}_{*}$ verifies the system \eqref{utilde penal} with $\mathbf{u}_{*}=\mathcal{F}(\varphi_{*})$ the solution of \eqref{Eqstate} corresponding to $\varphi_{*}.$
\end{pf} 
Now, suppose that $\beta_i$ are large for $i=1,2,3$. Let $\epsilon$ be a  positive real  parameter, and consider the following approximate control problem $(\bar{\P}_{\epsilon})$:
      \begin{gather*}
    \begin{split}
    \label{control problem modif}
\inf_{ \varphi\in \mL^{\mathbf{p}_{c}}(\cal{Q})}\bar{J}_{\epsilon}(\varphi),
\end{split}
\end{gather*}
where
\begin{gather*}
\begin{split}
 \bar{J}_{\epsilon}(\varphi)=J(\varphi)
 +\frac{1}{\epsilon\mathbf{p}_{c}}\int_{\cal Q}(\mathbf{a}-\varphi)^{\mathbf{p}_{c}}_{+} \d\x\d t+\frac{1}{\epsilon\mathbf{p}_{c}}\int_{\cal Q}(\varphi-\mathbf{b})^{\mathbf{p}_{c}}_{+}
 \d\x\d t
 +\frac{1}{4\epsilon}\int^{T}_{0}(\mathcal{G}(\varphi))^{2}_{+}\d t .
\end{split}
\end{gather*}
We can state the following convergence results.
\begin{proposition}
Every limit point $\varphi_{**}$ of the sequence of solutions $(\varphi_{\epsilon})_{\epsilon>0}$ of $(\bar{\P}_{\epsilon})$ satisfies  $\varphi_{**}\in\mathcal{U}_{ad},$ $\mathcal{G}(\varphi_{**})\in\mathcal{K}$, and $J(\varphi_{**})=\inf_{\varphi\in\mathcal{U}_{ad}} J(\varphi).$
\end{proposition}
  \begin{pf}
Let us consider  the following function $H$ 
     \begin{gather*}
        H(\Phi)=\frac{1}{\mathbf{p}_{c}}\int_{\cal{Q}}\!(\mathbf{a}-\Phi)^{\mathbf{p}_{c}}_{+}\d\x\d t+\frac{1}{\mathbf{p}_{c}}\int_{\cal{Q}}\!(\Phi-\mathbf{b})^{\mathbf{p}_{c}}_{+}\d\x\d t
 +\frac{1}{4}\int^{T}_{0}\!(\mathcal{G}(\Phi))^{2}_{+}\d t, \mbox{ with } \bar{J}_{\epsilon}(\Phi)=J(\Phi)
 + \frac{H(\Phi)}{\epsilon}.
     \end{gather*}
We have that $H({\varphi_{*}})=0$, at the solution $\varphi_{*}$ of $(\P)$, and so
     \begin{gather}
     \label{ineg suite}
         J(\varphi_{\epsilon})\leq \bar{J}_{\epsilon}(\varphi_{\epsilon})\leq \bar{J}_{\epsilon}({\varphi_{*}})=J({\varphi_{*}}).
     \end{gather}
     Then, $\varphi_{\epsilon}$ is bounded in $\mL^{\mathbf{p}_{c}}(\cal{Q}),$ we can therefore extract
of the family $(\varphi_{\epsilon})$ a sequence $(\varphi_{{\epsilon}_k})$ which converges weakly to a limit $\varphi_{**}\in \mL^{\mathbf{p}_{c}} (\mathcal{Q})$ when $\epsilon_{k}\longrightarrow 0.$ Moreover, we have according to \eqref{ineg suite}
\begin{equation}
     0\leq H(\varphi_{\epsilon_{k}})= \epsilon_{k}(\bar{J}_{\epsilon_{k}}(\varphi_{\epsilon_{k}})-J(\varphi_{\epsilon_{k}}))\leq \epsilon_{k}(J({\varphi_{*}})-J(\varphi_{\epsilon_{k}})).
  \end{equation}    
Passing to the limit $\epsilon_{k}\longrightarrow 0,$ and using the fact that $H$ is continuous (since $H$ is a sum of continuous functions), and using \eqref{ineg suite}, we deduce that $H(\varphi_{**})=0,$ which implies $\varphi_{**}\in\mathcal{U}_{ad},$ and $\mathcal{G}(\varphi_{**})\in\mathcal{K}.$ Again, according to \eqref{ineg suite},  we have $J(\varphi_{**})\leq J({\varphi_{*}})$ and we  deduce that $J(\varphi_{**})=\inf_{\varphi\in\mathcal{U}_{ad}} J,$ which concludes the proof.
  \end{pf}
\begin{rmk}
(i) A brief discretization of this approach and numerical simulations of optimal solutions for eradicating malignant breast tumours were recently proposed in \cite{LBH2, LBH1} with some realistic data.  \\
(ii) Note that since the optimal solutions of $(\P)$ are not unique, every limit point of the sequence of solutions of $(\bar{\P}_\epsilon)$ belongs to the set of optimal solutions of $(\P)$.
     \end{rmk}
\subsection{Illustrative Simulations}
We use $2D$ lung domain for $\Omega$ ($m=2$) and consider $\rho_1=3$ and $\rho_2=2$ in the problem \eqref{Eqstate}. Then from the point $(I_2)-\rm{(iii)}$ of \textit{\textbf{(H2)}}, we have $\mathbf{p}_{c}=2\rho_1/(\rho_1 -1)=3.$ From the Proposition \ref{gradJtildeeps}, the optimal control $\varphi_{*}$ and state $\mathbf{u}_*$ satisfy (with $\mathbf{p}_{c}=3$) the optimality conditions \eqref{cequi0}-\eqref{ceqi}.\\
 In all our simulations, we considered a metronomic treatment, see, e.g., \cite{BocKer2016} spanning $ T = 20 $ days, where a low dose of drugs is administered continuously throughout the treatment period.
For the discretization of \eqref{Eqstate}, we used the implicit Euler method with the Newton method, and $\text{P}^{1}$ Lagrange finite elements. The numerical simulations presented below were obtained using \texttt{FreeFem++} \cite{Hecht2012}. To solve the penalized control problem, we use the \texttt{BFGS} (Broyden-Fletcher-Goldfarb-Shanno) algorithm available on \texttt{FreeFem++}. 
In our numerical experiments, the \texttt{BFGS} function is configured with the following parameters: \texttt{eps = 1e-9}, \texttt{nbiter = 10}, and \texttt{nbiterline = 1}.

 To set up the state problem \eqref{Eqstate}, we used some data from \cite{EfB2021,SKMa2023} (for $(t,\x)\in (0; T)\times \Bar{\Omega}$):\\
 $ \mD=\mD_{0}\e^{-\gamma_{0}\sigma}$, $\sigma=\sigma_{0}\e^{-\delta_{0}|{\x}|^{2}t}$, $\mK_{1}=k_{1}\e^{-\delta_{0}|{\x}|^{2}t}$, $\mK_{2}=\theta^{2}\mK_{1}$, $\mathbf{a}(\x)=0$, $\mathbf{b}(\x)=b_{0}\e^{- b_{1}[(x-x_{1})^{2}+(y-y_{1})^{2}]}$, $\mK_{3}=k_{3}{\mK_{1}}_{\Gamma}$, $\mD_{0}=d_{0}\eta_{1}$, $\eta_{1}=1+\varsigma_{1}\cos(\varsigma_{2}\eta_{2})$, $\eta_{2}=\varsigma_{3}\arctan\bigg(\frac{y-y_{0}}{\varsigma_{4}+\sqrt{(x-x_{0})^{2}+(y-y_{0})^{2}}}\bigg)$.\\
The discretization informations are:
 $\Delta t=0.2$ days (the time step),  $\x=(x,y)\in\RR^{2}$ (the $2D$-space position), $\Delta x =0.48mm$,  $\Delta y=0.51 mm$. Table \ref{datatab1} presents the data parameters considered in our simulations.

\begin{table}[htbp]
     \caption{Data parameters}
    \centering
   \begin{tabular}{|c|c|c|c|c|c|c|c|}
 \hline
Parameters & D$_{0}$ & $\delta_{0}$ & $k_{1}$ & $\alpha_{0}$ & $\sigma_{0}$&  $b_{0}$ &  $b_{1}$\\  
  \hline 
Values & 0.2 $mm^{2}/d$ &$5.55\cdot 10^{-5}/mm^{2}\cdot d$& $0.44/d$ & $1/\mu\text{M}\cdot d$\!\! & $25$ mmHg& 100 $\mu$M & 11.11$mm^{2}$\!\!\\
  \hline  
\end{tabular} 
    \centering
   \begin{tabular}{|c|c|c|c|c|c|c|c|c|c|c|c|c|c|c|c|c|}
\hline
\hline
$\gamma_{0}$ & $k_{3}$& $T$&  $\varsigma_1$ &  $\varsigma_{2}$&  $\varsigma_3$ &  $\varsigma_{4}$& $\theta$\\  
  \hline 
$ 1/mm$Hg & 3 $mm$& 20 $d$& $0.7$& $50$& $2$& $10^{-4}$ & $1.4$\\
  \hline  
\end{tabular} 
\label{datatab1}
\end{table}
In Table \ref{datatab1}, the millimeter of mercury ($mm$Hg) is the unit of measurement for pulmonary arterial pressure, the micro-molar ($\mu$M) is the unit of chemotherapeutic drug concentration and the millimeter ($mm$) is the unit of length, and $d$ means days.
 \subsubsection{Optimal solutions for the eradication of left lobe lung tumours}
We consider  
$u_{0}({\x})=\eta_{1}\e^{-\frac{1}{\varepsilon}[(x-x_{0})^{2}+(y-y_{0})^{2}]} \mbox{ in $\Omega,$ }$ representing the initial tumour density in the lung, centred in $(x_{0},y_{0})=(1.34,3.8)$ with $\varepsilon=0.1.$ 
The mesh consists of $n_{me} = 1492$ elements and $n_{mp} = 821$ points.
The control  parameters are fixed as follows: $\lambda = 10^{-2}$, $\beta_1 = 10^3$, $\beta_2 = 3 \cdot 10^4$, $\beta_3 = 5 \cdot 10^6$, $(x_1, y_1) = (1.3, 3.5)$.\\
 For better clarity of the figures, different scales have been used for all figures. Figure \ref{tumst} shows the progression of tumour density in the absence of treatment, revealing a marked increase in density that extends further into the bronchi, confirming the aggressive nature of the cancer cells. Figure \ref{figerac1} demonstrates the feasibility of constraints applied to both the optimal control and state variables. Figure \ref{contAC} depicts the optimal drug concentration $\varphi_{*}$ and the optimal tumour density $\mathbf{u_{*}}$, where $\mathbf{u_{*}}$ exhibits a significant reduction, aligning with our expectations. In a realistic treatment protocol, therapy ceases once tumour density becomes undetectable. However, Figure \ref{1dec} indicates a substantial drug concentration in the lung, even as tumour density declines, due to the continuous 20-day treatment. As anticipated, the optimal tumour density decreases concurrently with the drug concentration.

\begin{figure}
  \centering
        \includegraphics[width=0.26\textwidth]{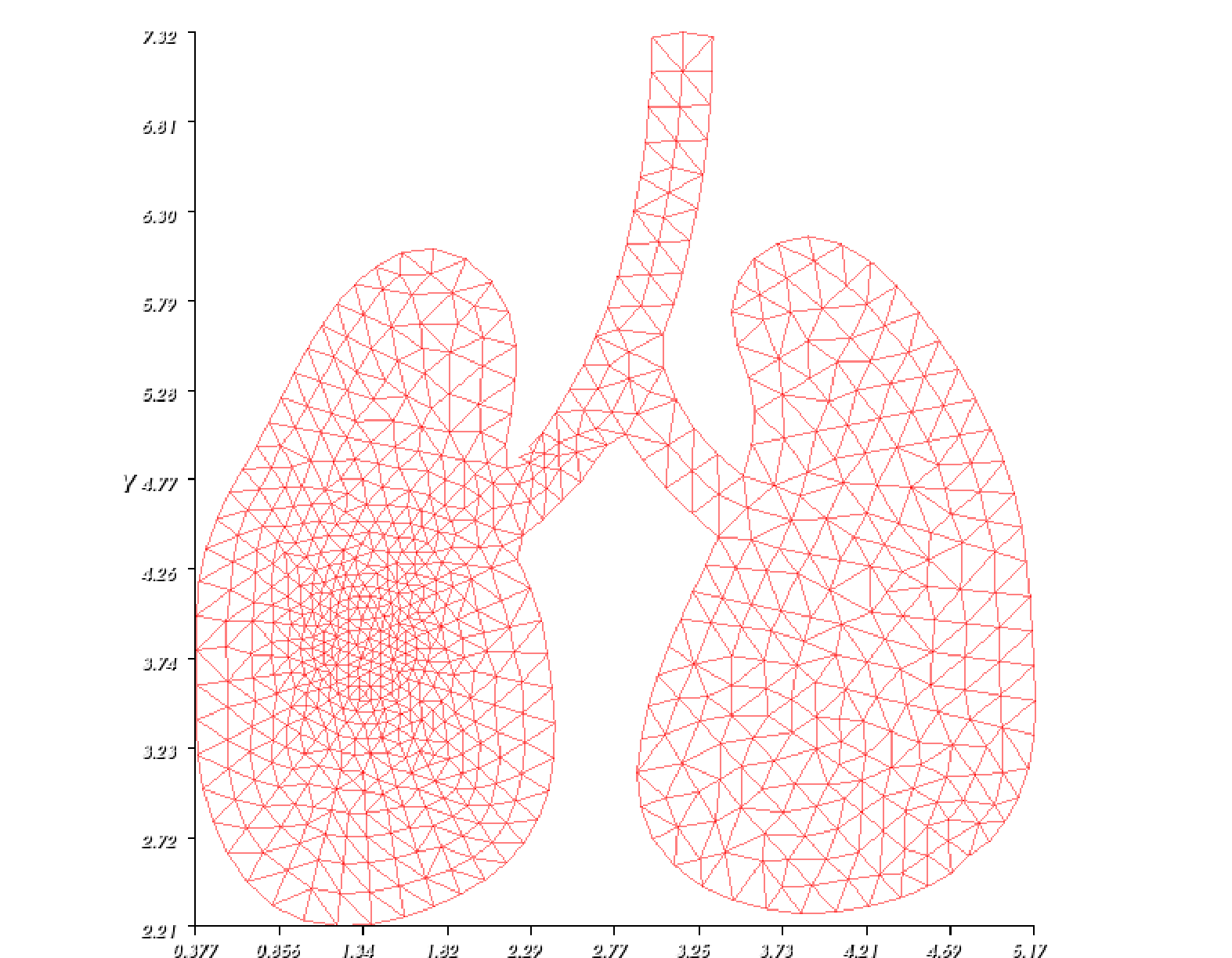}
        \includegraphics[width=0.26\textwidth]{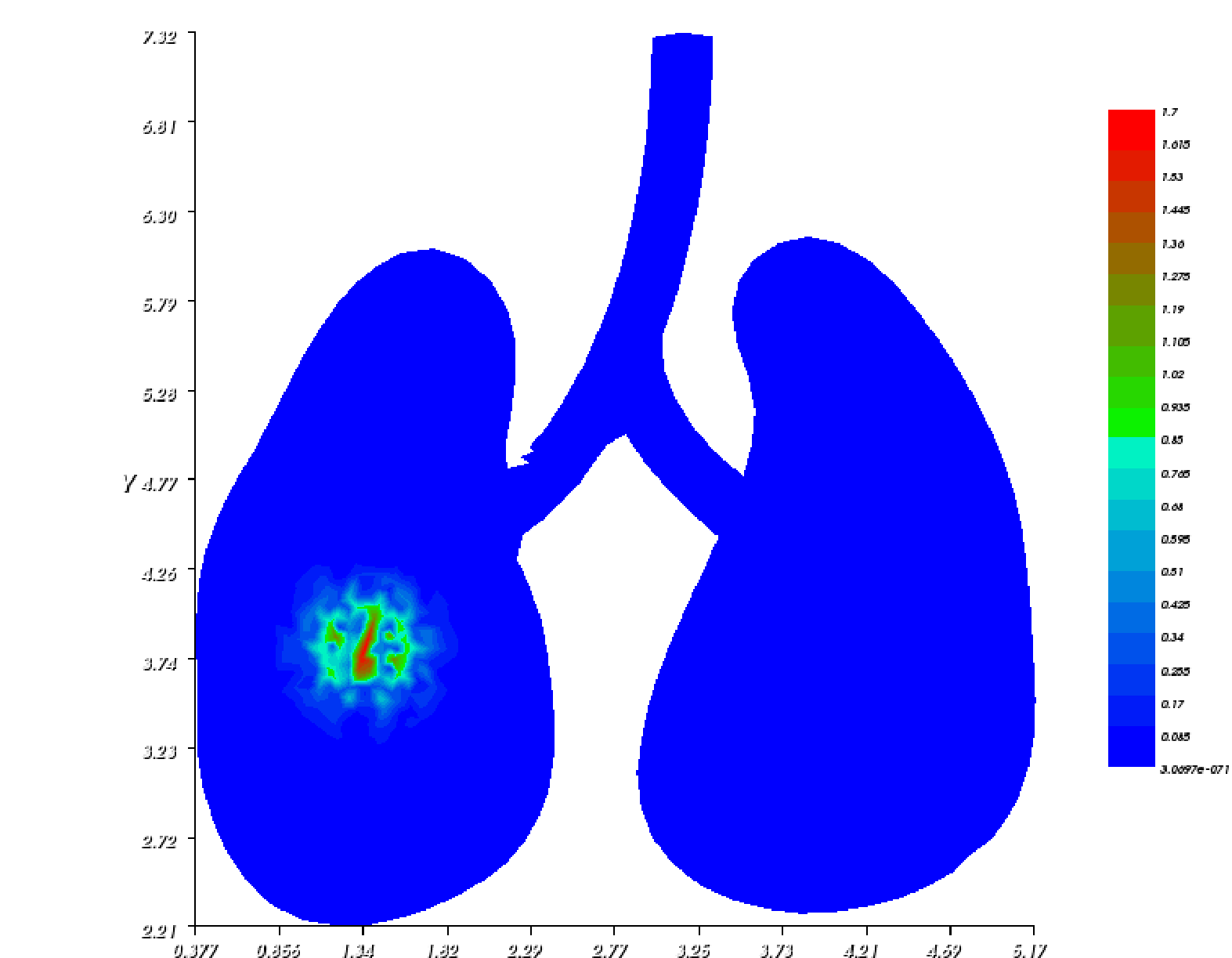} 
         \caption{The lung domain $\Omega$ on the left and the initial tumour density $u_{0}$ on the right.} 
       \label{u01}
    \end{figure}
\begin{figure} 
\centering
        \includegraphics[width=0.245\textwidth]{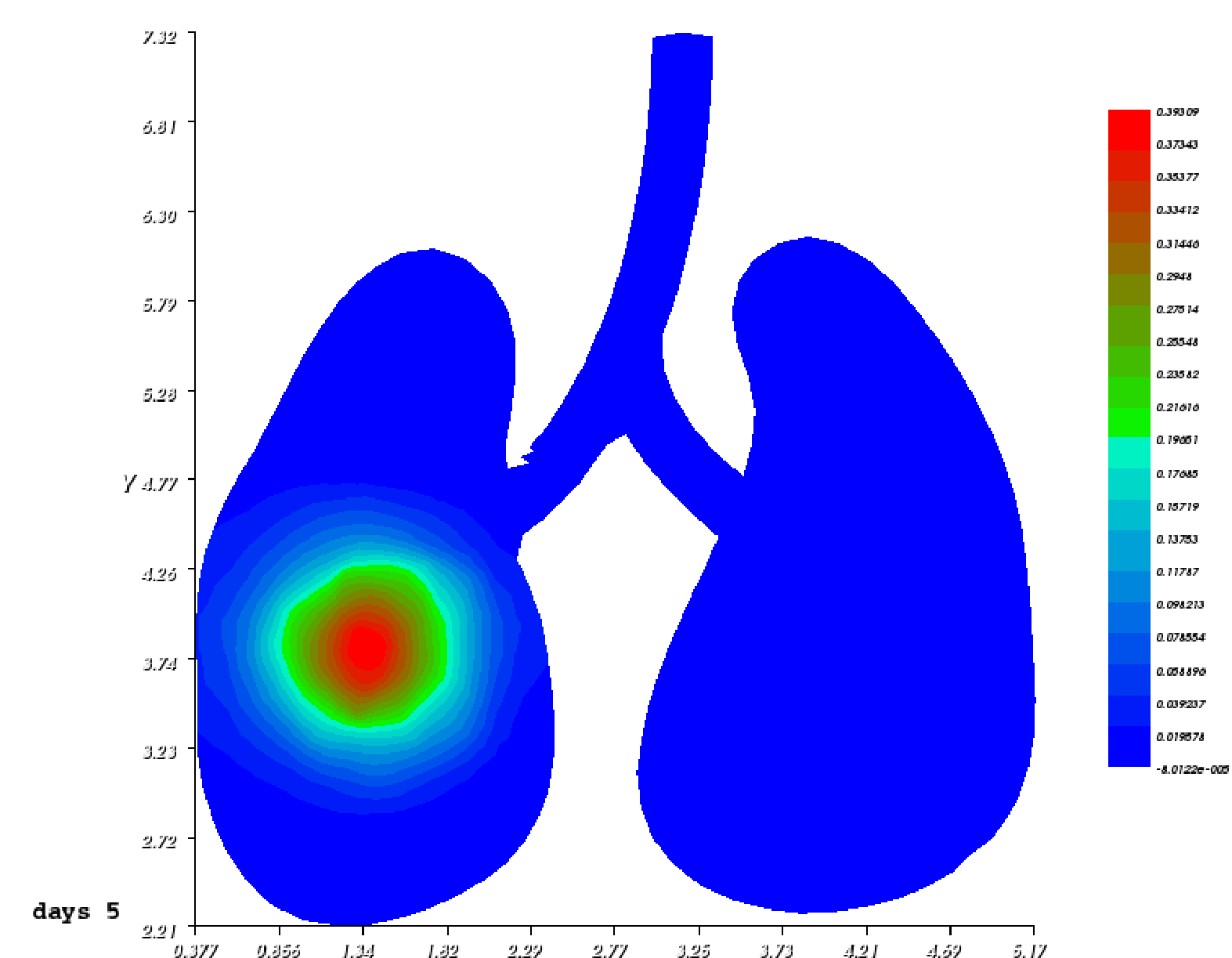}
        \includegraphics[width=0.245\textwidth]{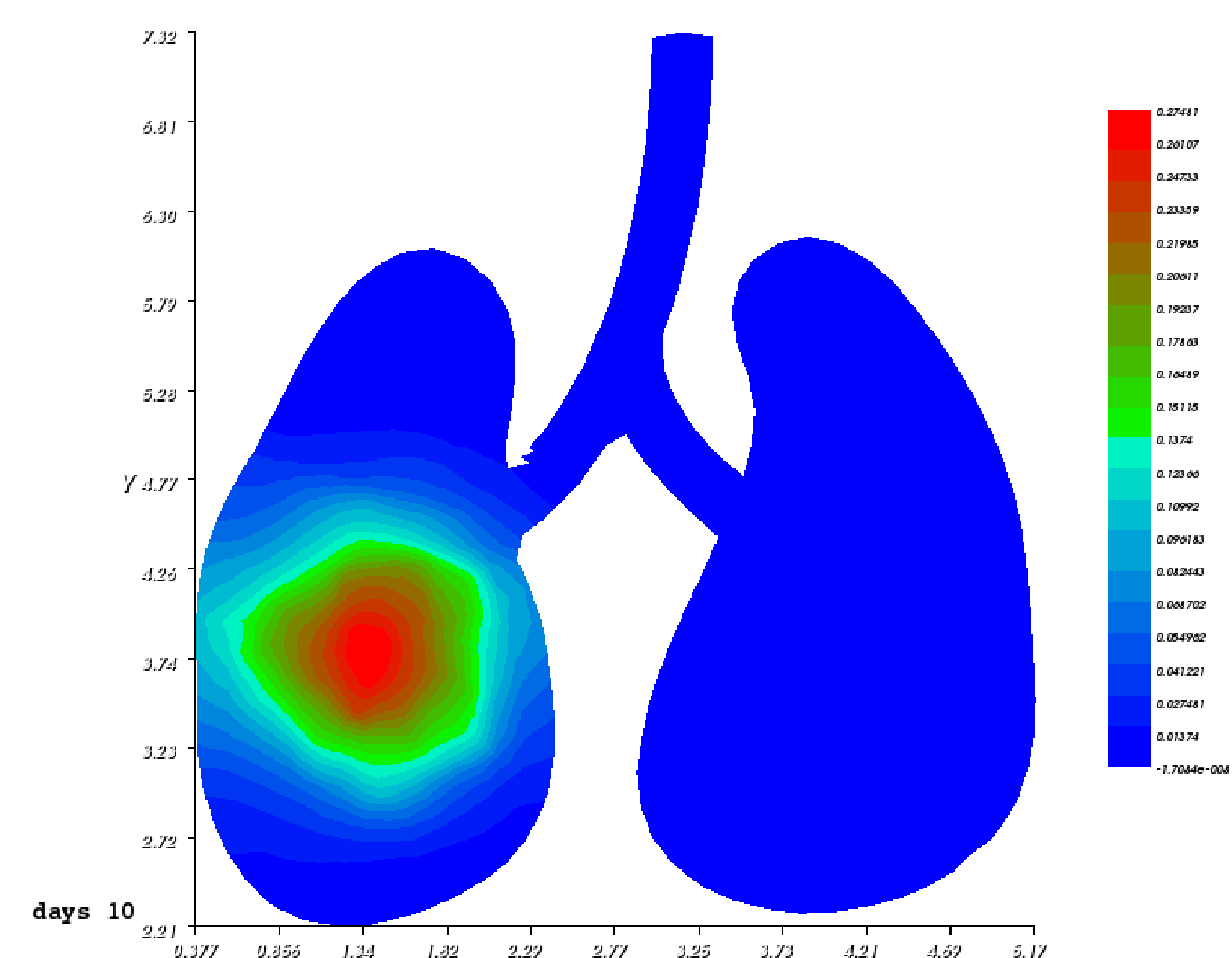}
        \includegraphics[width=0.245\textwidth]{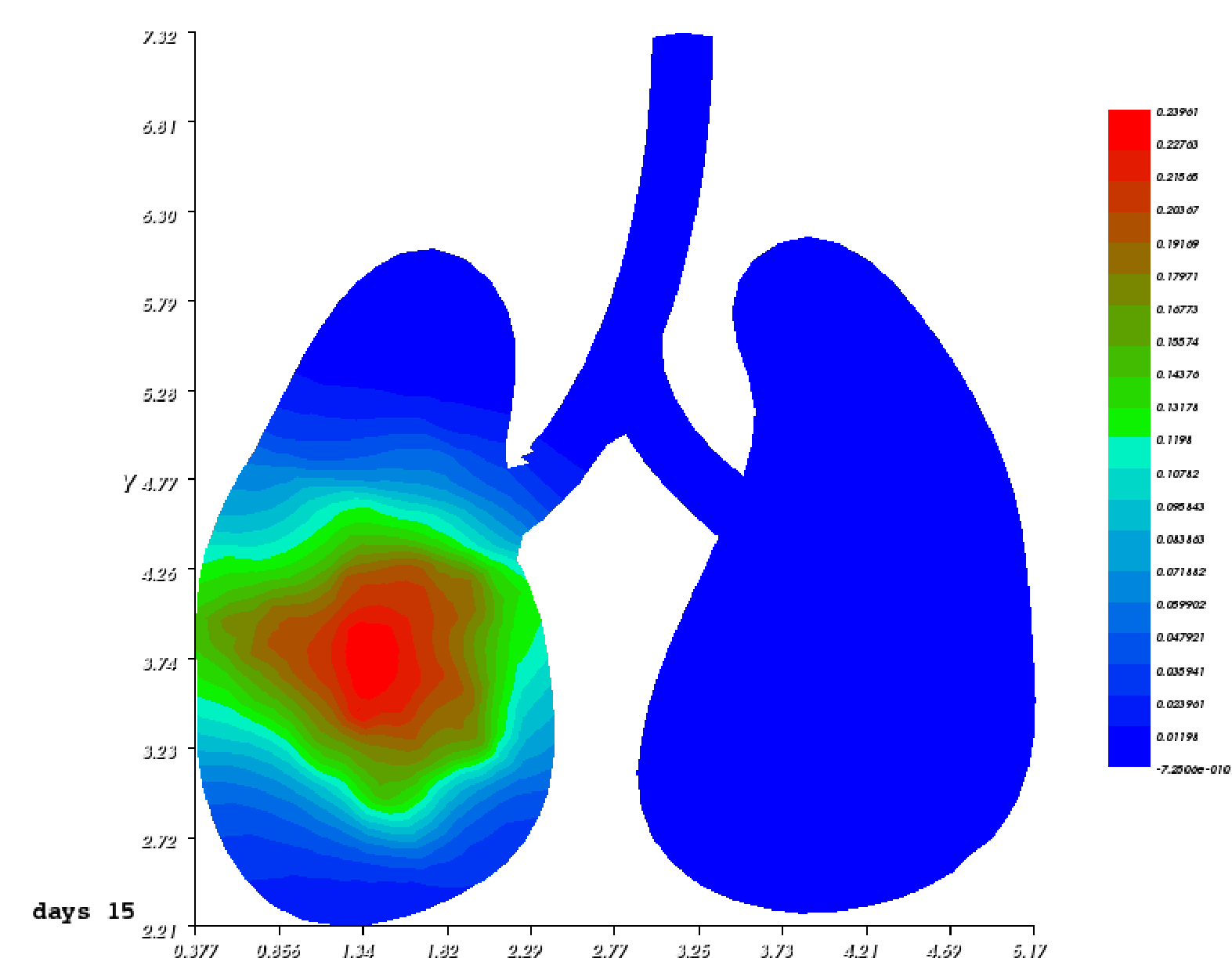}
        \includegraphics[width=0.245\textwidth]{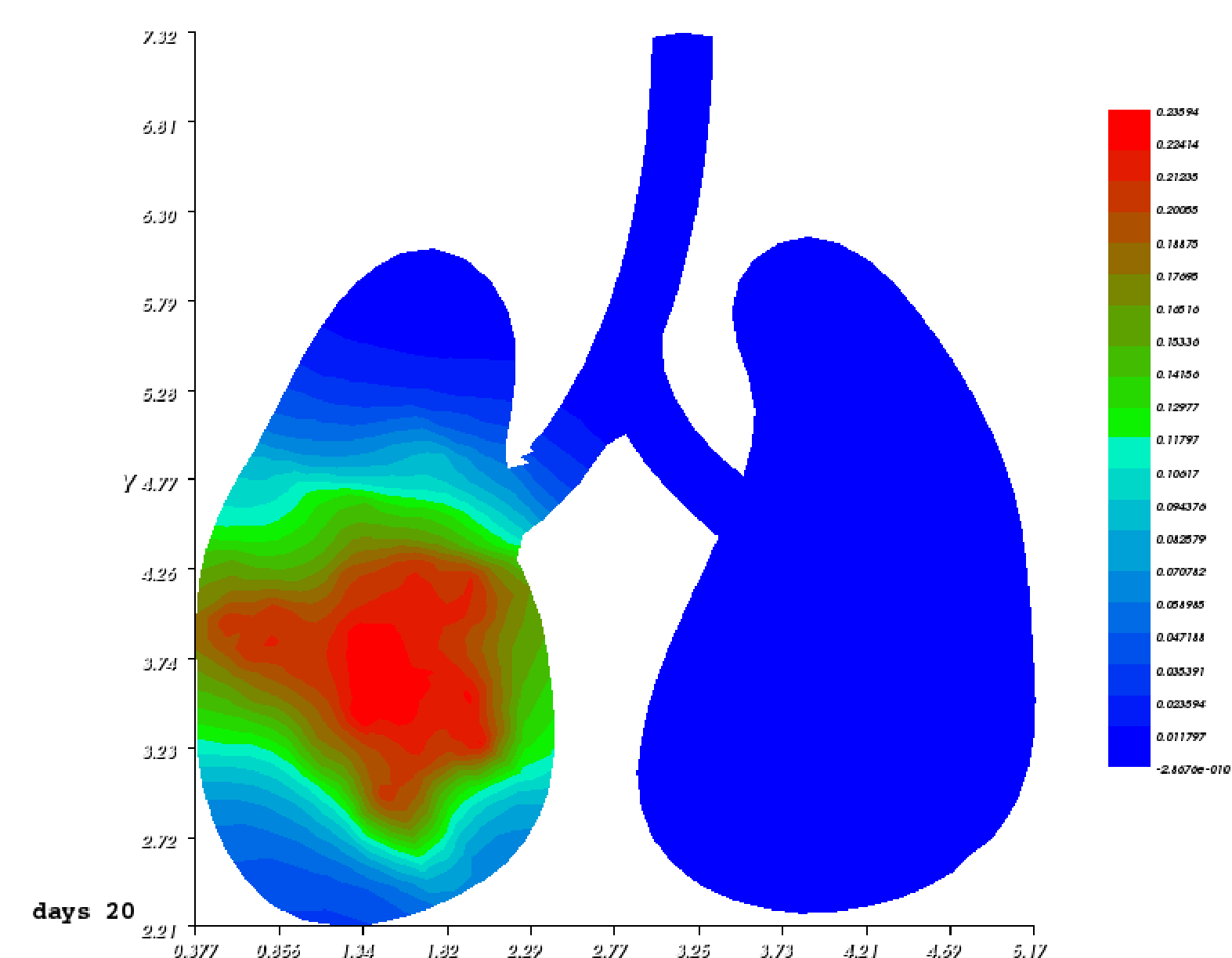}
          \caption{The tumour density evolution $\mathbf{u}$ without treatment, $i.e.,$ $\mathbf{u}  =\mathcal{F}(0)$ on 5th, 10th, 15th, 20th day, respectively, from left to right.}
      \label{tumst}  
   \end{figure} 
   \begin{figure}
   \centering
        \includegraphics[width=0.295\textwidth]{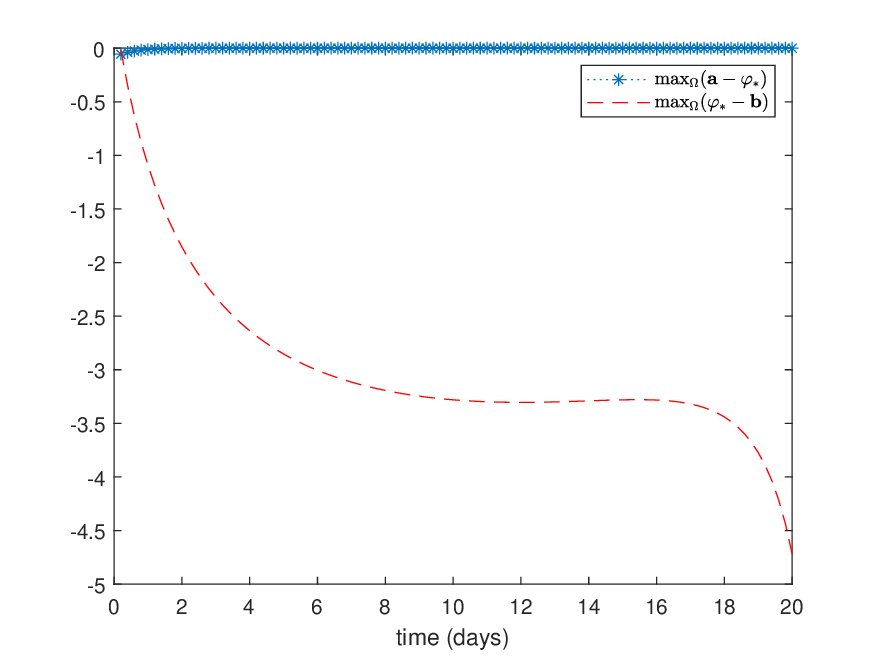}
        \includegraphics[width=0.295\textwidth]{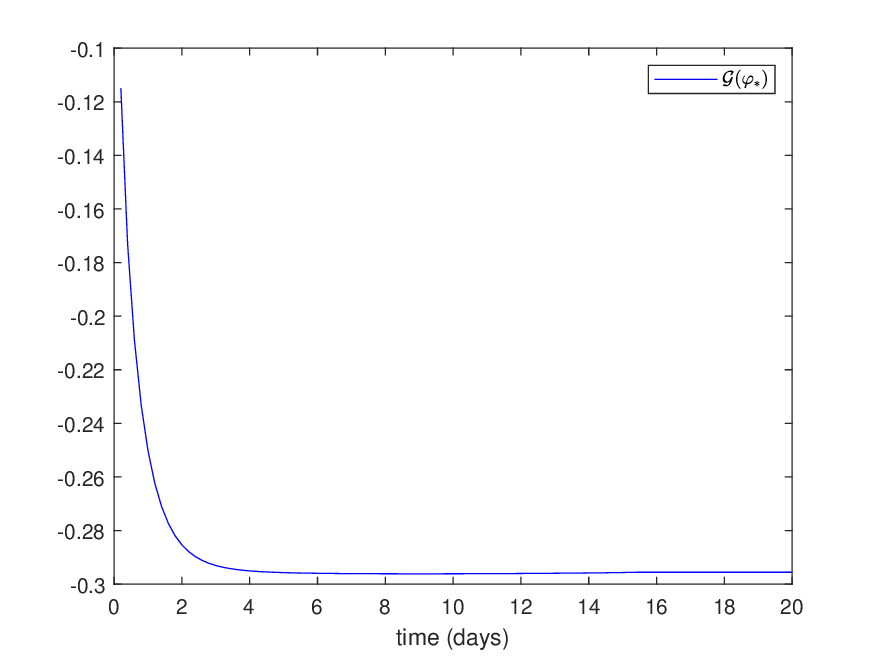}
         \caption{Evolution of the feasibility of the control $\varphi_{*}$ on the left and the state $\mathbf{u_{*}}$ on the right.}
          \label{figerac1}
   \end{figure}
    \begin{figure}
\centering
    \includegraphics[width=0.295\textwidth]{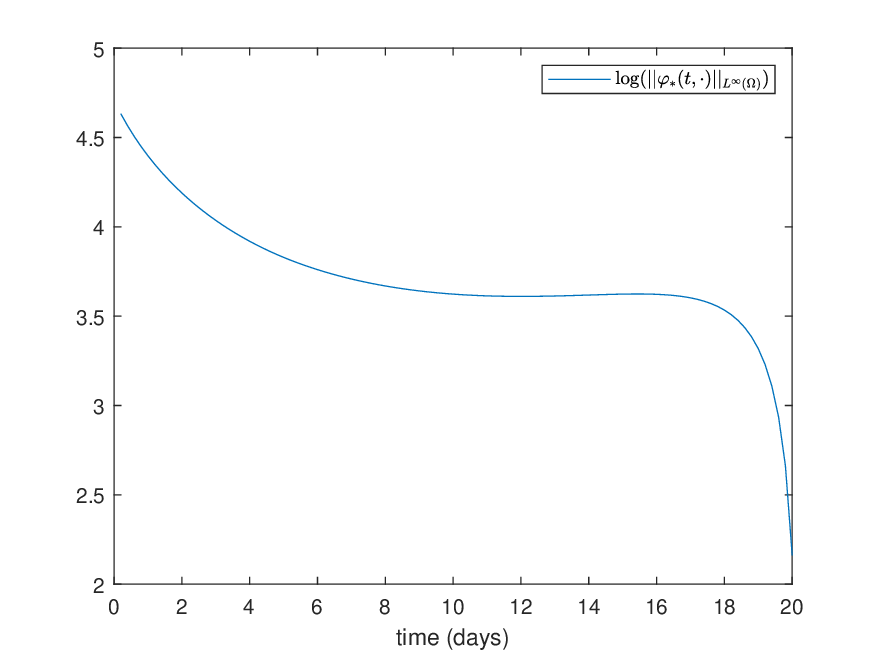}
        \includegraphics[width=0.295\textwidth]{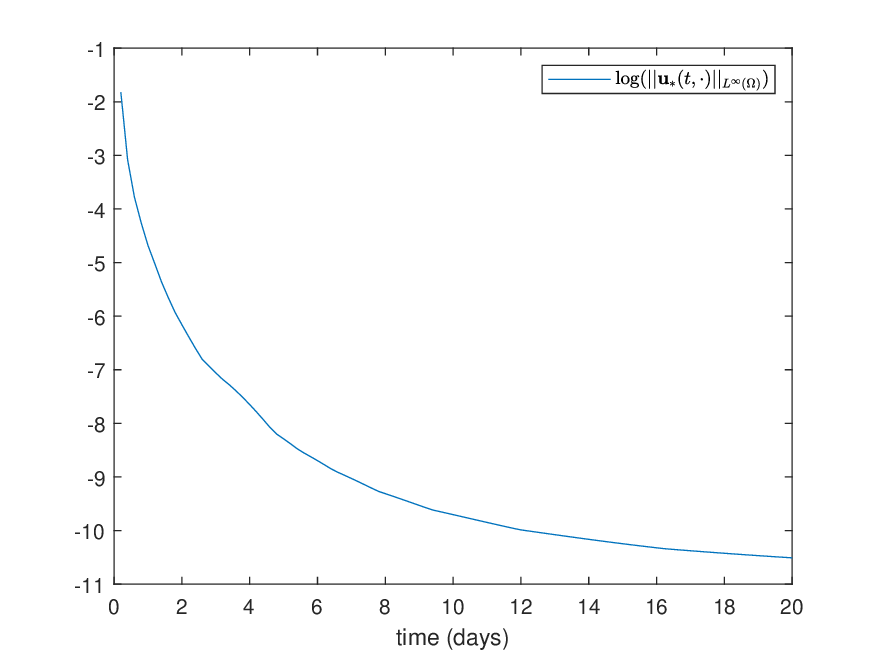}
           \caption{Evolution of $\log(\lVert\varphi_{*}\rVert_{\mL^{\infty}(\Omega)})$ on the left and $\log(\lVert\mathbf{u}_{*}\rVert_{\mL^{\infty}(\Omega)})$ on the right.}
 \label{1dec}
   \end{figure} 
    \begin{table}[htbp]
   \label{max tab1}
     \caption{The maximum and minimum value of the optimal tumour density. }
    \centering
   \begin{tabular}{|c|c|c|c|c|}
\hline
Time $(t)$ &  5{th} day&  10{th} day &  15{th} day & 20{th} day\\  
  \hline
 $\max_{\Omega}\mathbf{u_{*}}$& $0.00045$& $6.10\cdot 10^{-5}$& $3.54\cdot 10^{-5}$ & $2.72\cdot 10^{-5}$\\
  \hline  
  $\min_{\Omega}\mathbf{u_{*}}$& $1.55\cdot 10^{-7}$& $4.17\cdot 10^{-10}$& $8.78\cdot 10^{-11}$ & $1.14\cdot 10^{-12}$\\
  \hline
\end{tabular} 
\end{table}
\begin{figure}
\centering
        \includegraphics[width=0.22\textwidth]{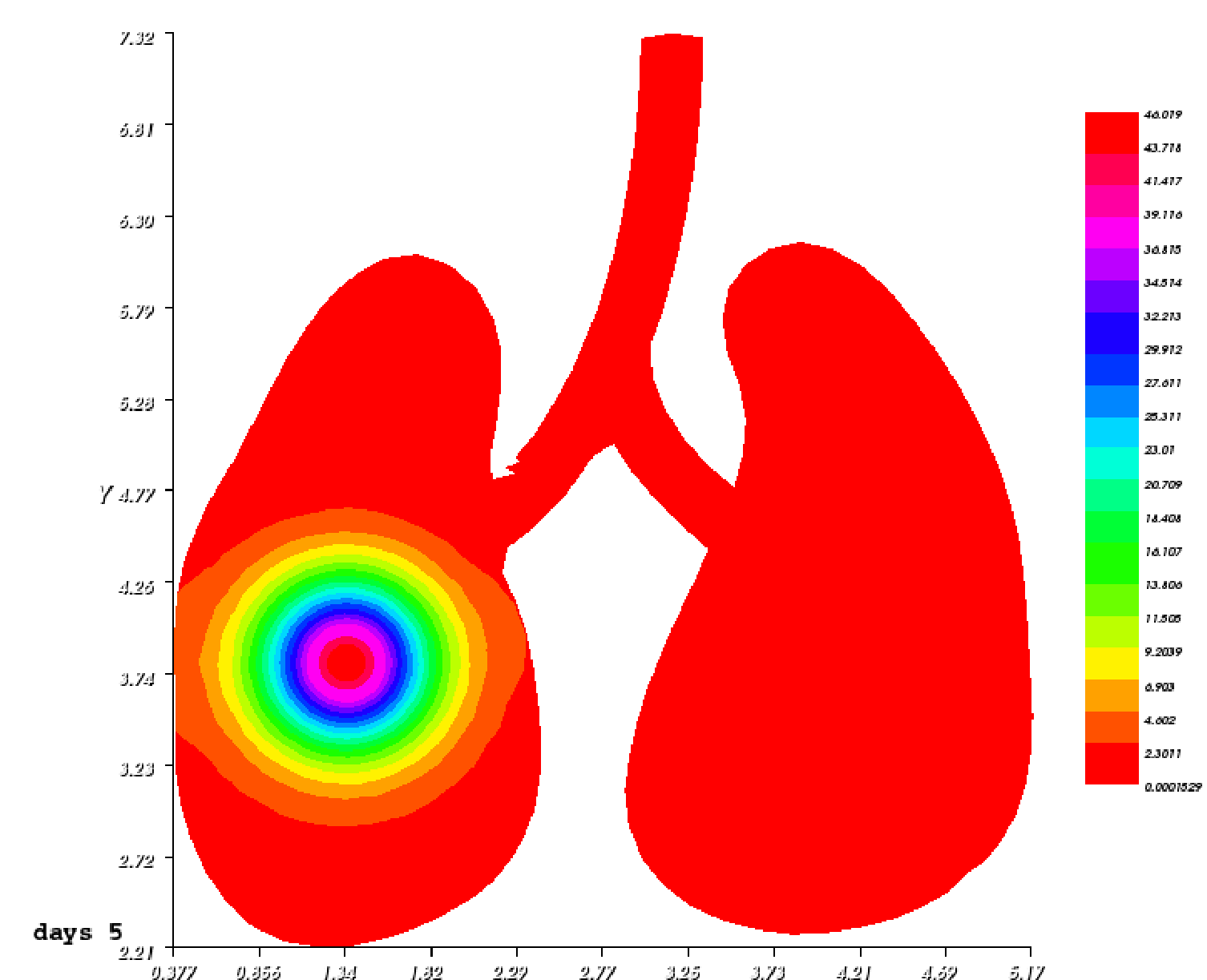}  
\includegraphics[width=0.22\linewidth]{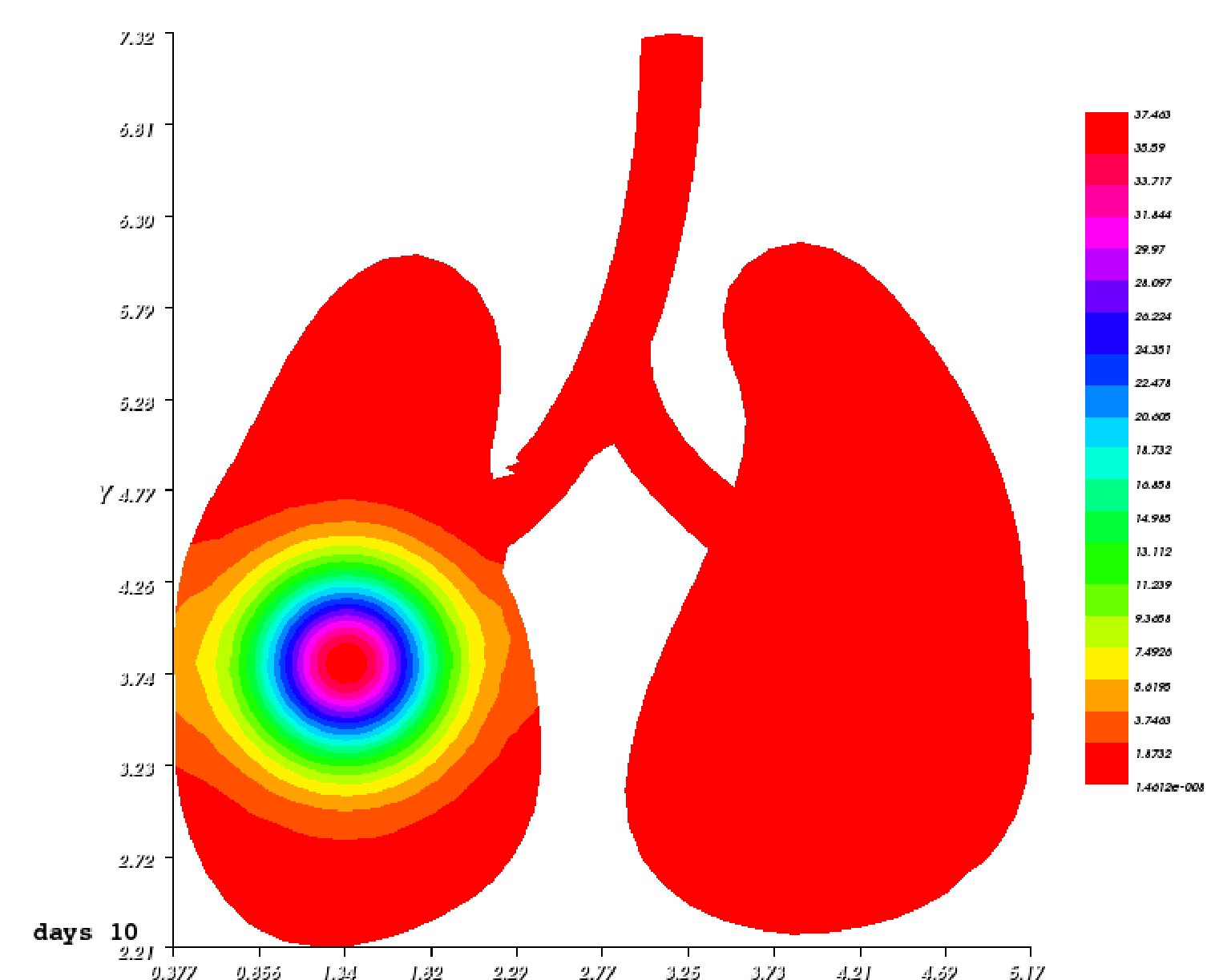}     
\includegraphics[width=0.22\textwidth]{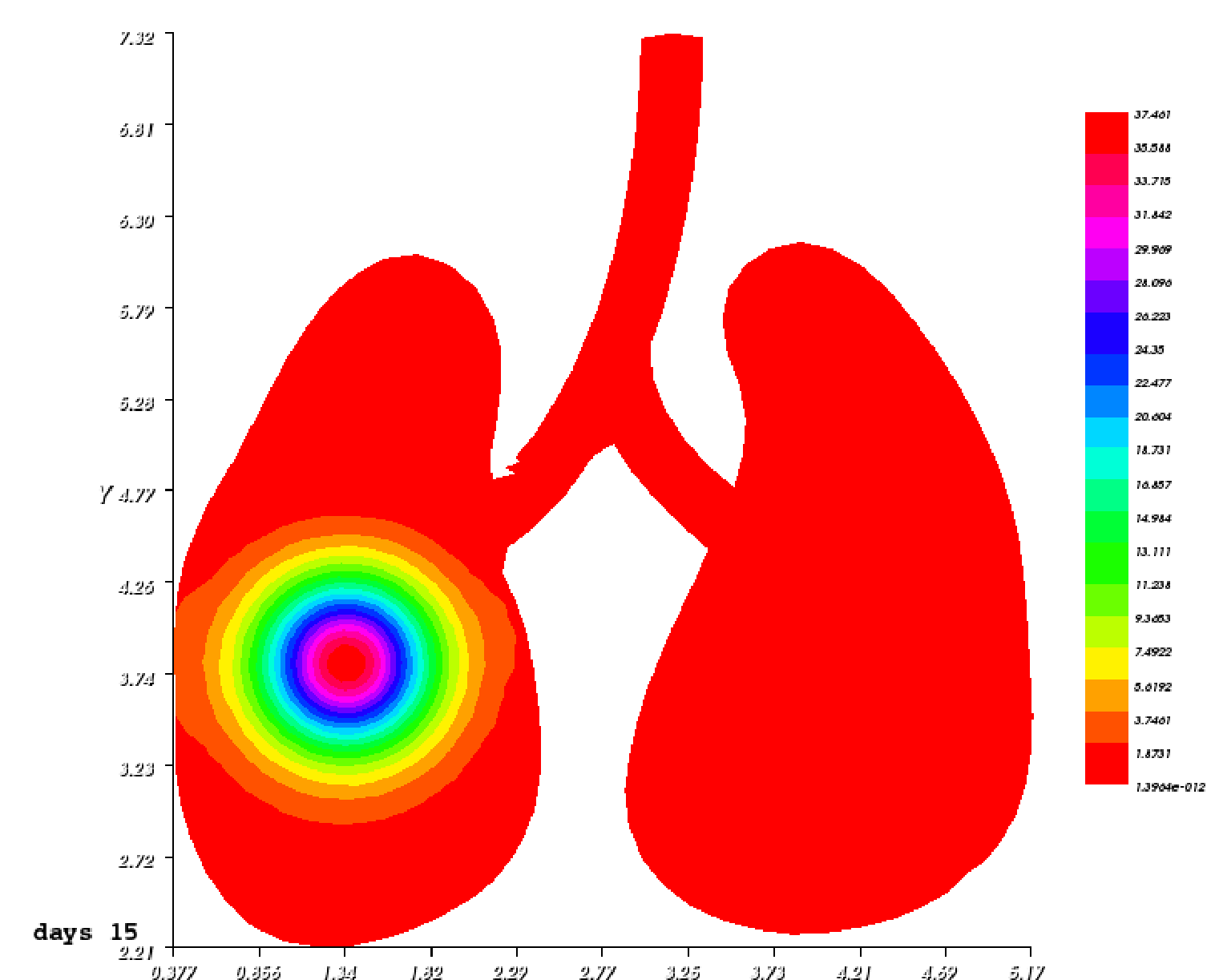}
 \includegraphics[width=0.22\textwidth]{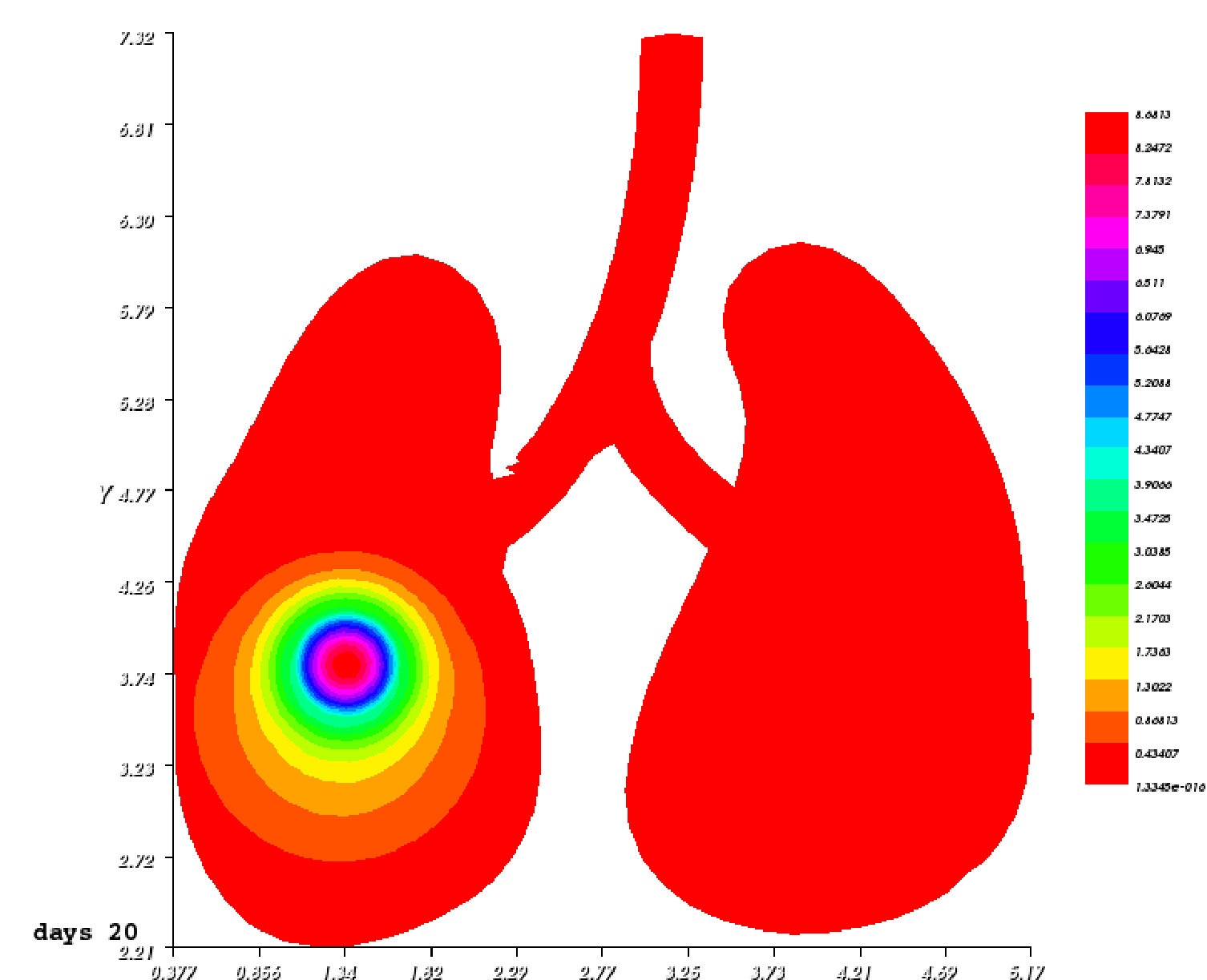}  
        \includegraphics[width=0.22\textwidth]{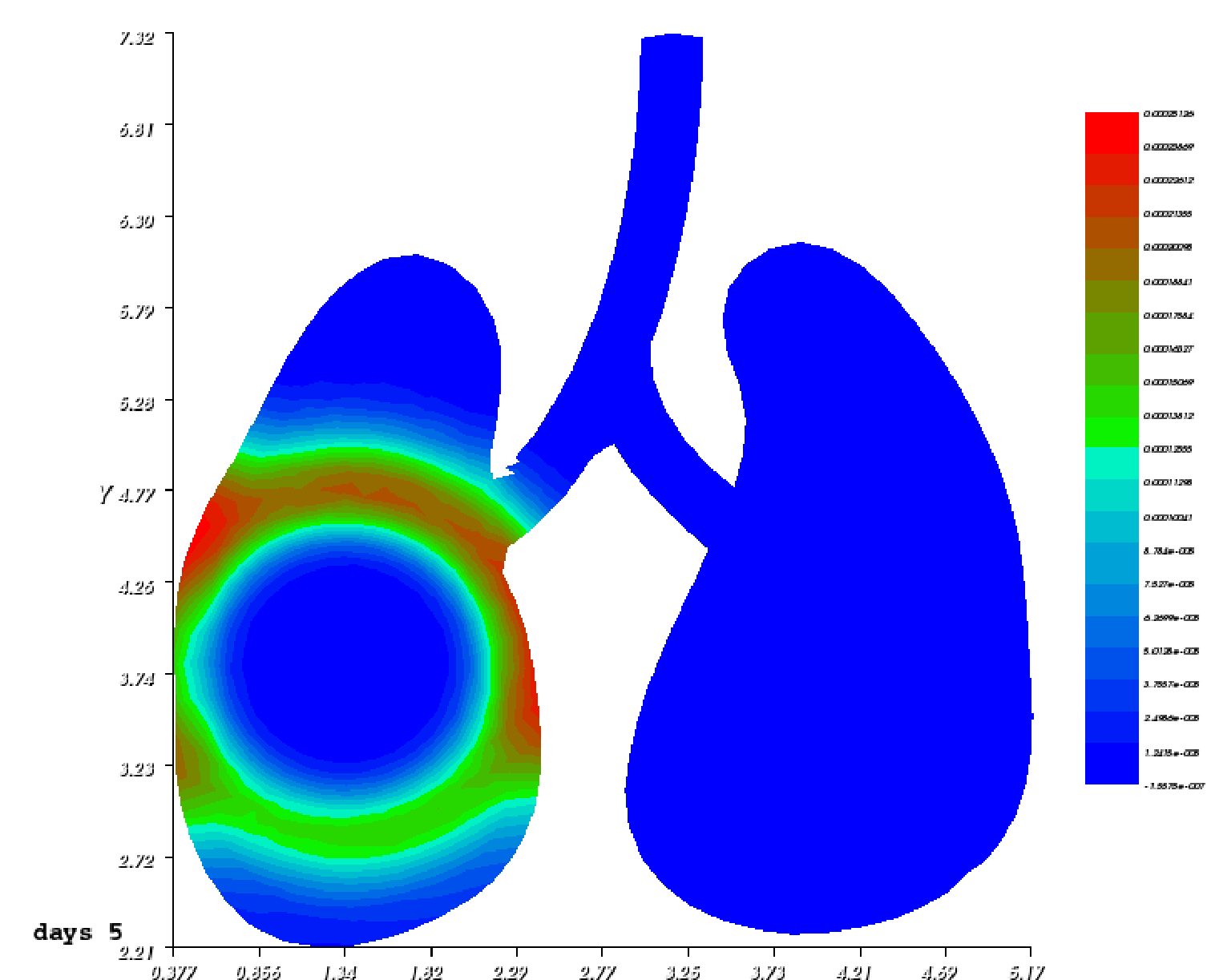}          
        \includegraphics[width=0.22\textwidth]{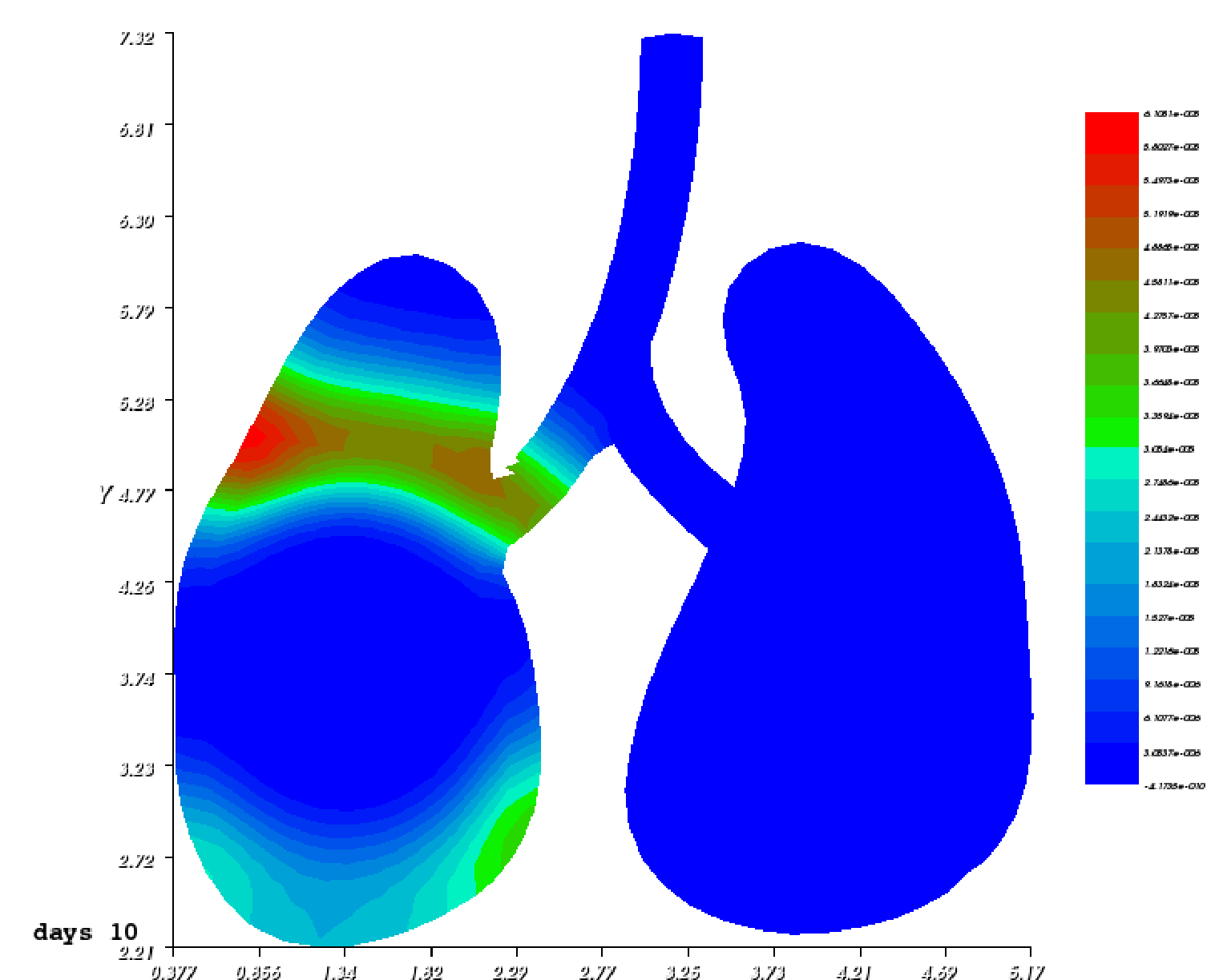}
        \includegraphics[width=0.22\textwidth]{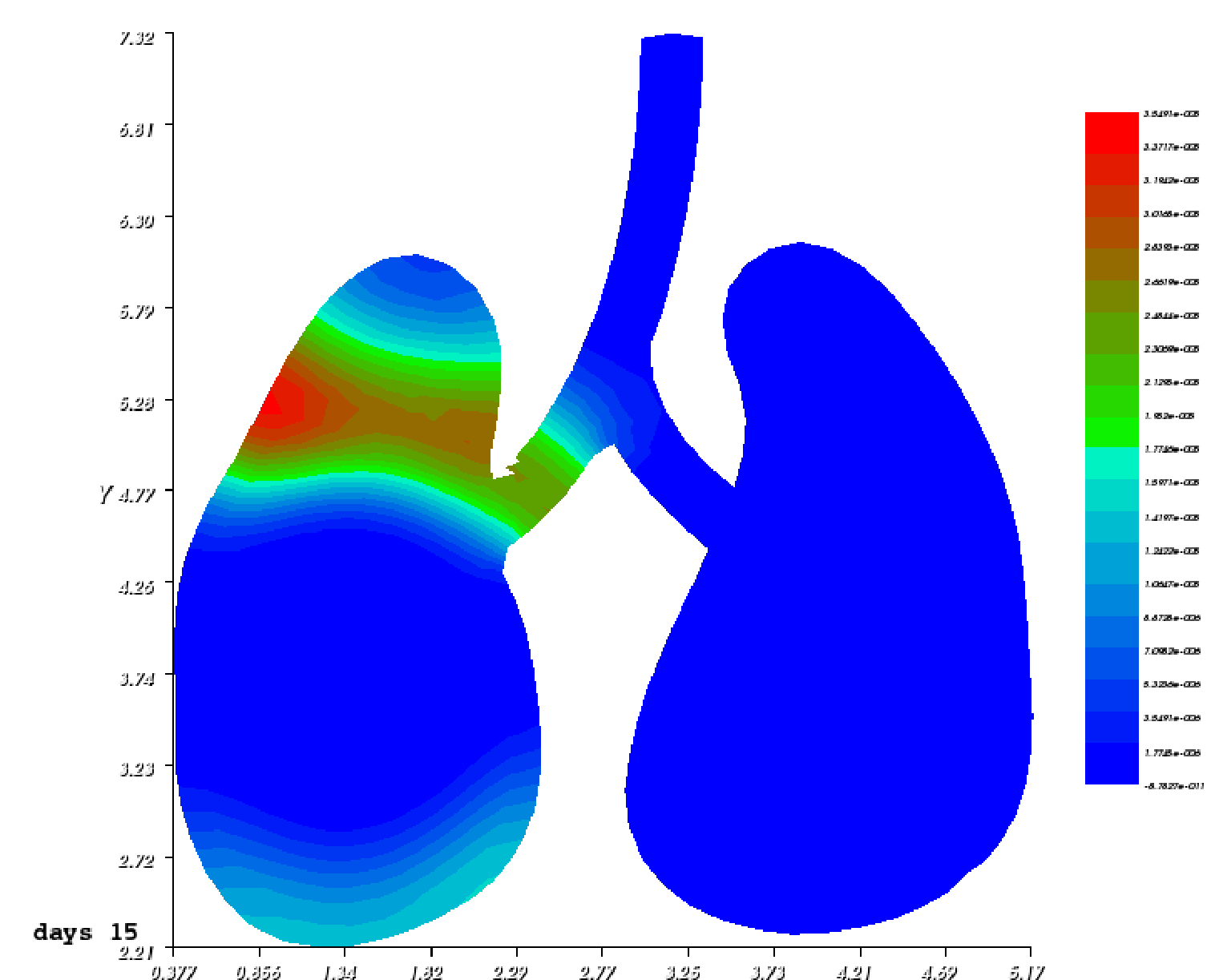} 
        \includegraphics[width=0.22\textwidth]{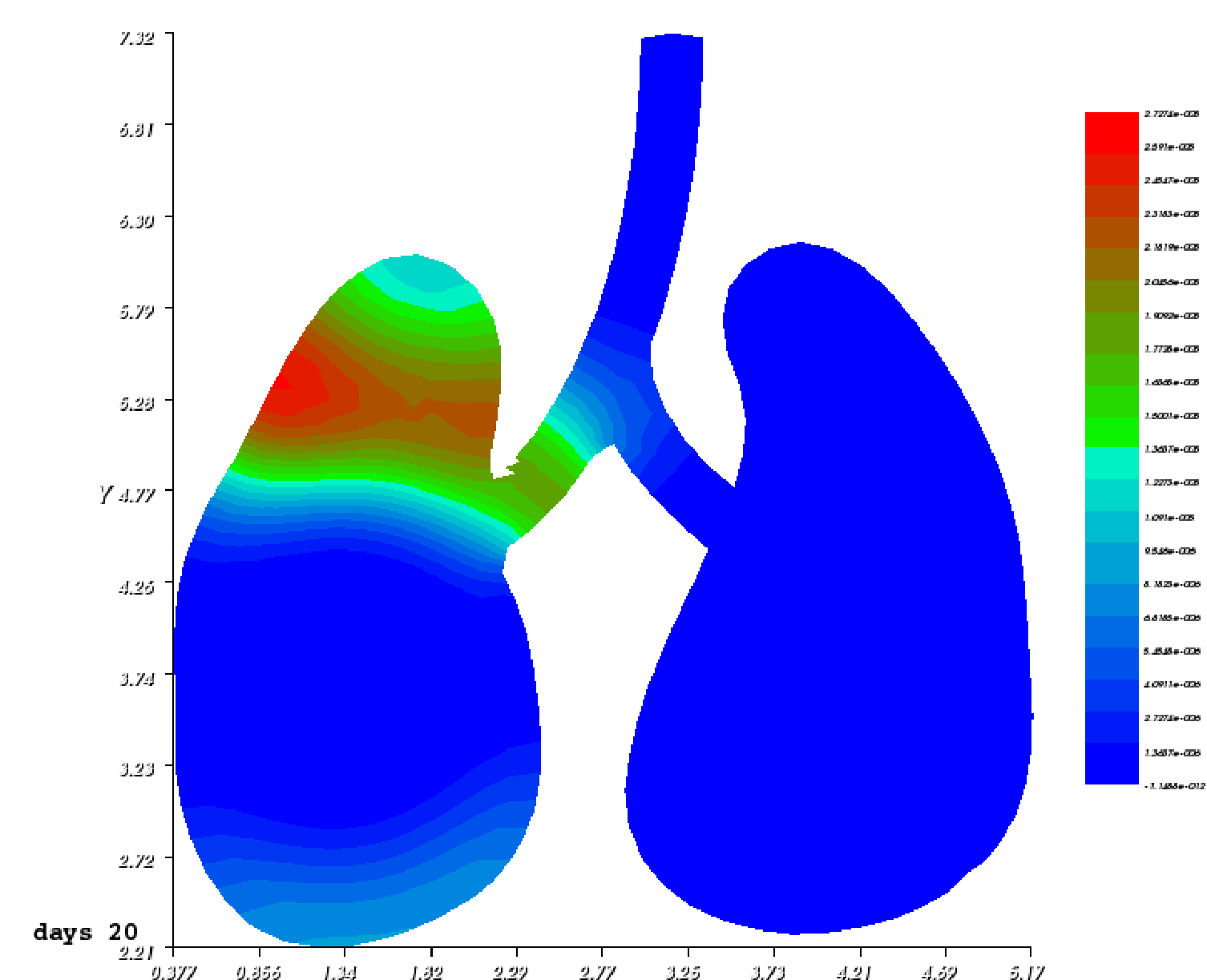}
         \caption{
 The optimal drug $\varphi_{*}$ at the top and the optimal tumour density $ 
         \mathbf{u}_{*}$ at the bottom.
          }
     \label{contAC}
   \end{figure} 
   
   For this case, the results are promising, and all the figures are consistent and meet our expectations.
 \subsubsection{Optimal solutions for the eradication of right lobe lung tumours}
For this case, we consider  
$u_{0}({\x})=\eta_{1}\e^{-\frac{1}{\varepsilon}[(x-x_{0})^{2}+(y-y_{0})^{2}]} \mbox{ in $\Omega,$ }$ representing the initial tumour density centred in $(x_{0},y_{0})=(3.8,4.5).$ Here, the mesh consists of
 $n_{me}=1444$ elements and $n_{mp}=797$ points. The control  parameters are fixed as follows: $\lambda=5\cdot 10^{-2}$, $\beta_{1}=2\cdot 10^{4}$, $\beta_{2}=10^{5}$, $\beta_{3}=10^{4}$, $(x_{1},y_{1})=(1.3,3.5)$.\\
  Figure \ref{2tumst} illustrates the progression of tumour density in the absence of treatment, showing an increase in tumour density within the right lobe, affecting the right bronchus. Figure \ref{figerac2} demonstrates the feasibility of the state and control constraints.  Figure \ref{2contAC} depicts the optimal drug concentration $\varphi_{}$ and the optimal tumour density $\mathbf{u_{}}$ over several days. As expected, $\mathbf{u_{*}}$ decreases as showed in Figure \ref{2dec}.
\begin{figure}
  \centering
        \includegraphics[width=0.245\textwidth]{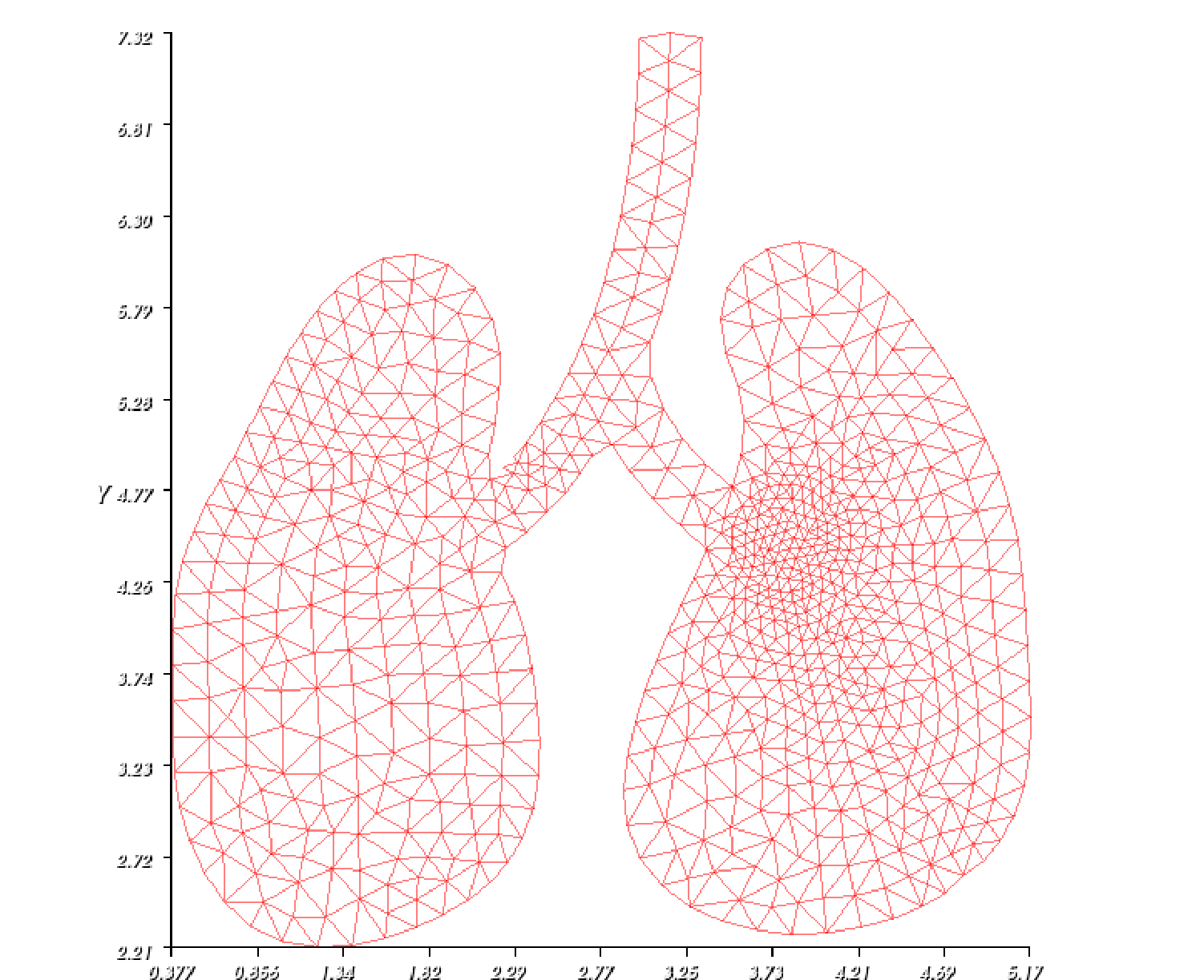}
        \includegraphics[width=0.245\textwidth]{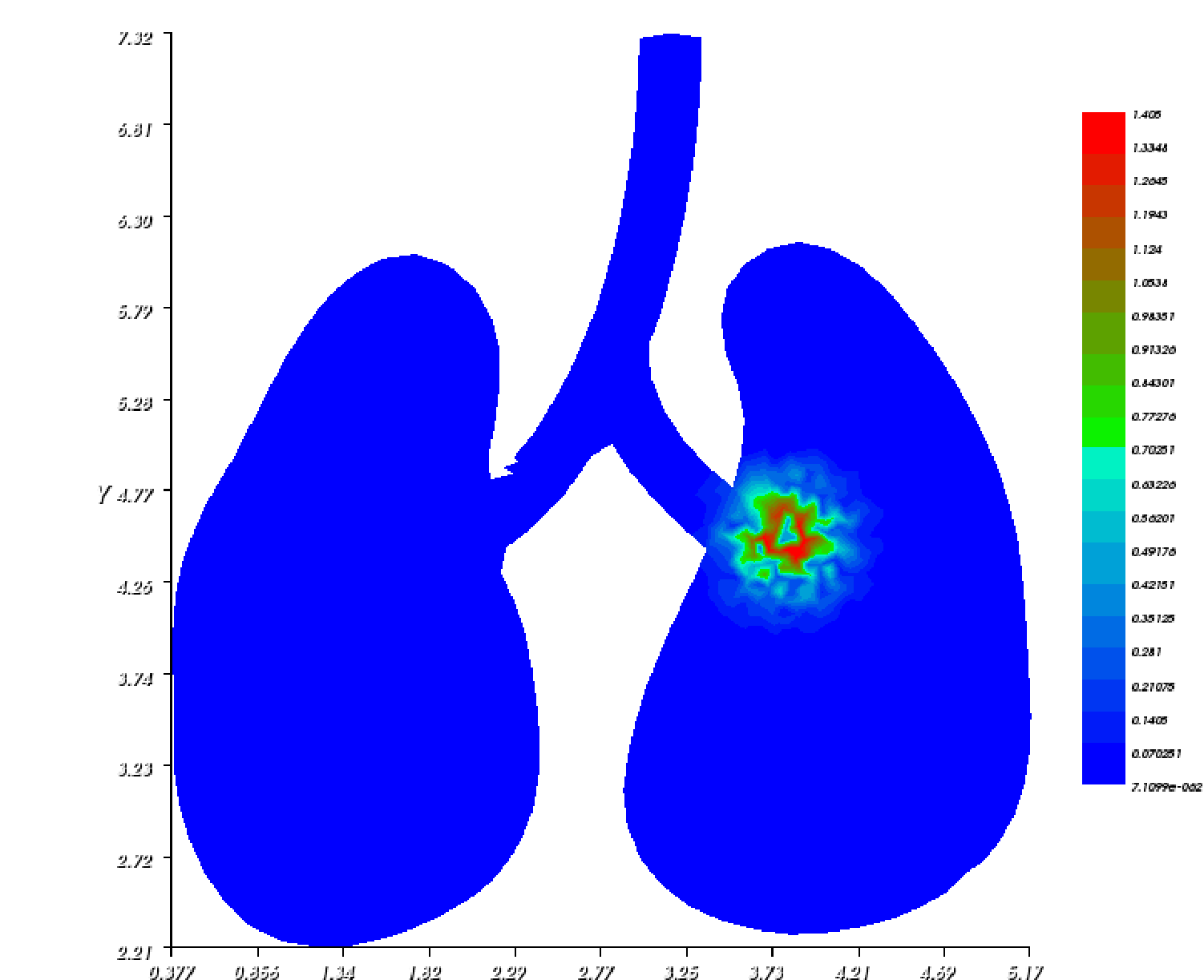}
         \caption{The lung domain $\Omega$ on the left and the initial tumour density cells $u_{0}$ on 
         the right.} 
       \label{2u0}
    \end{figure}
    \begin{figure}
\centering 
        \includegraphics[width=0.245\textwidth]{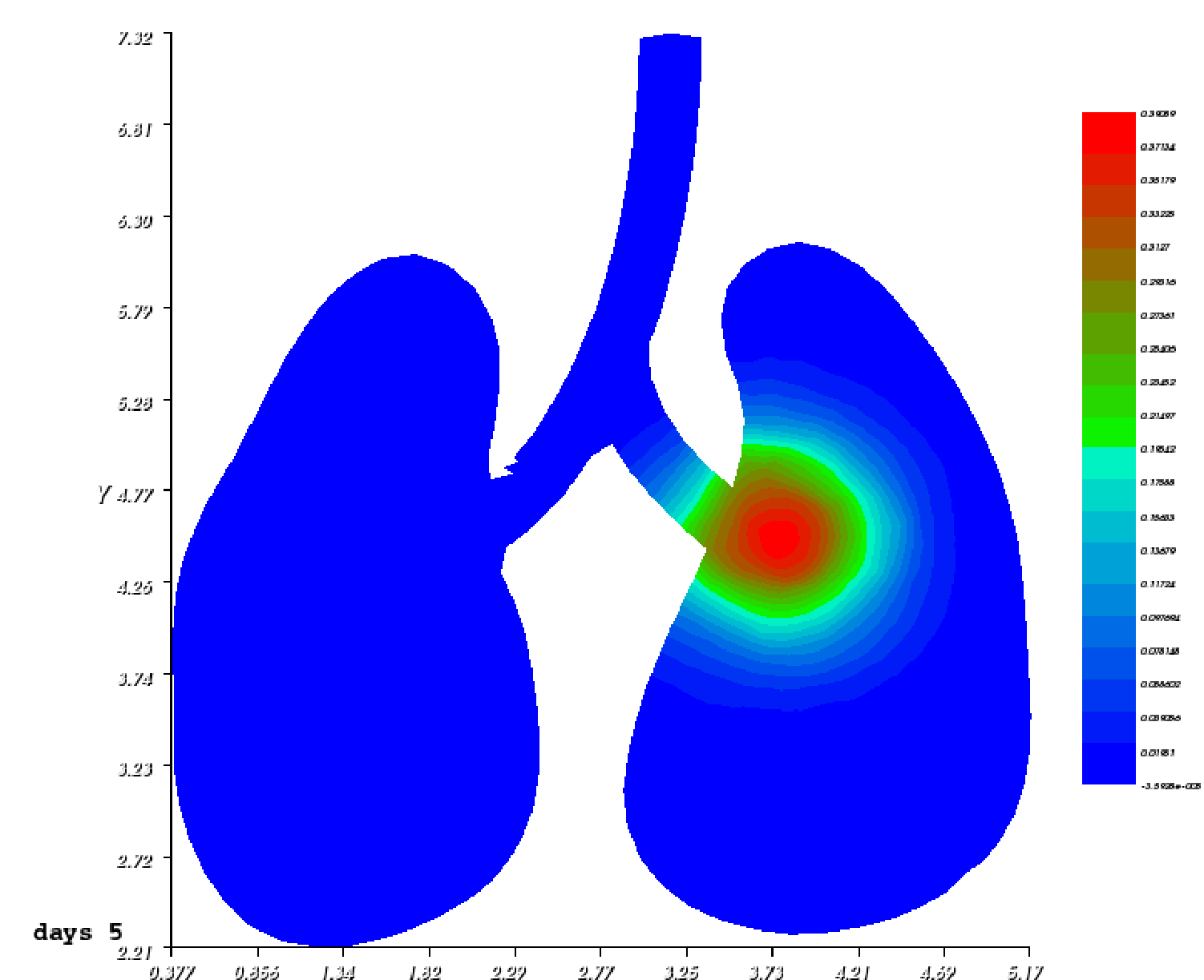}
        \includegraphics[width=0.245\textwidth]{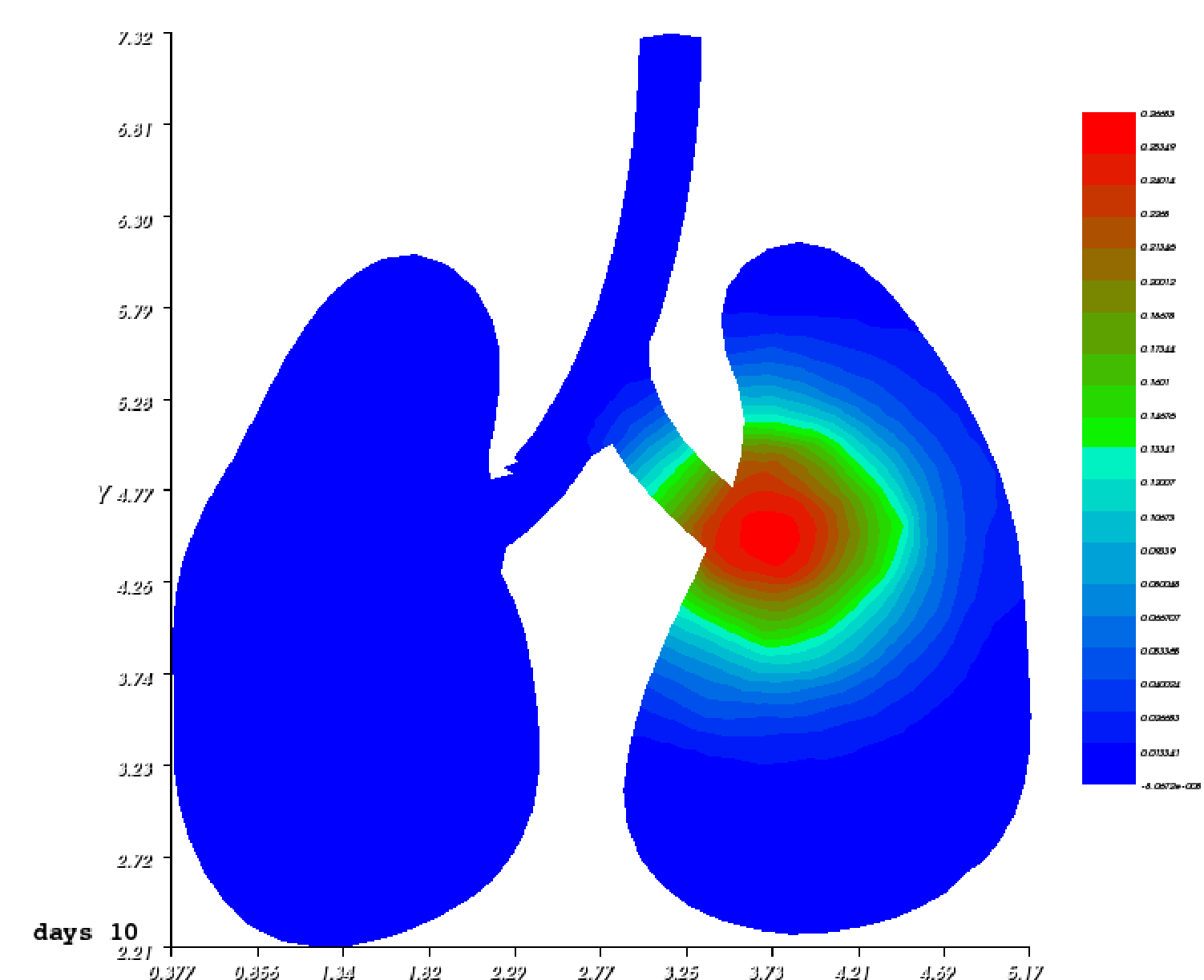}
        \includegraphics[width=0.245\textwidth]{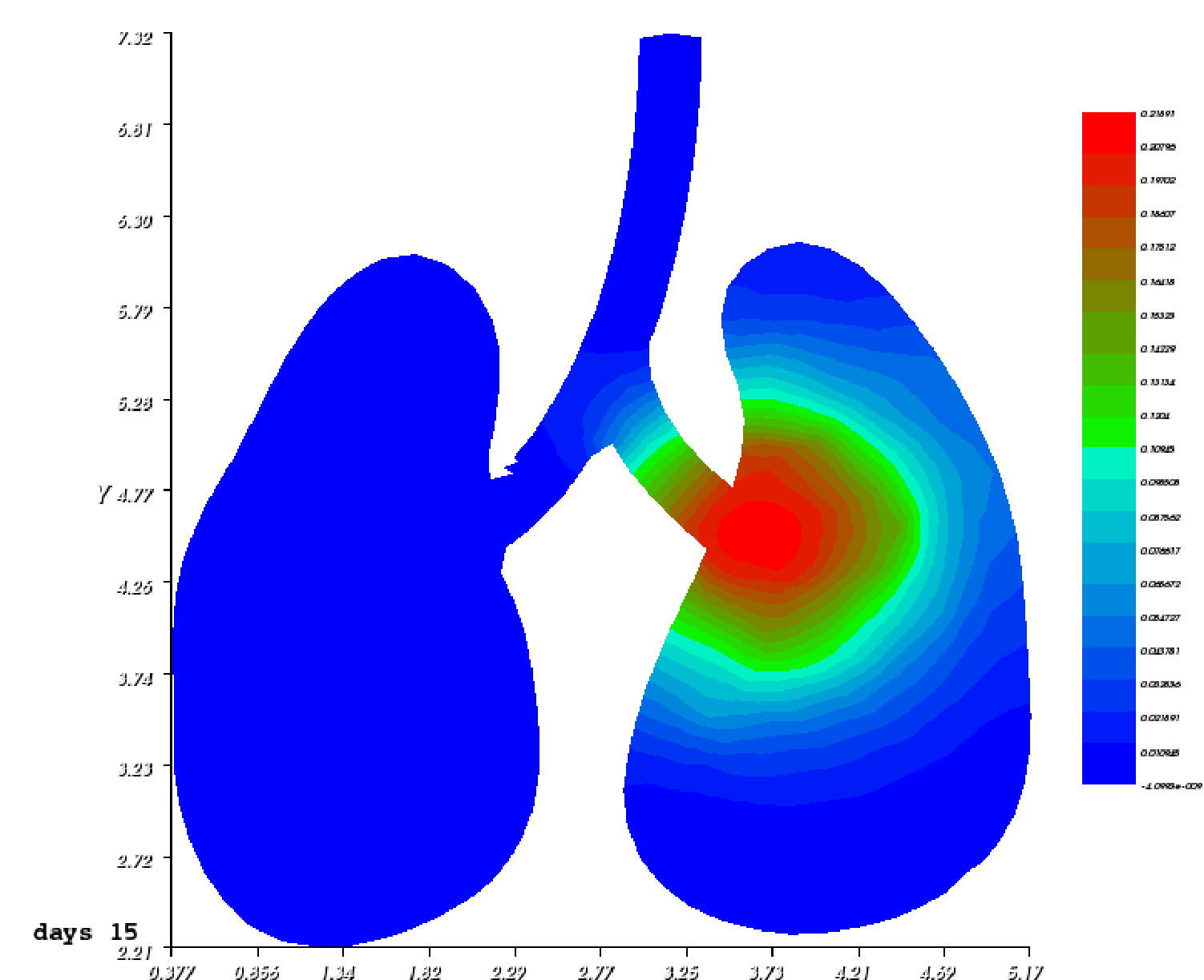}
        \includegraphics[width=0.245\textwidth]{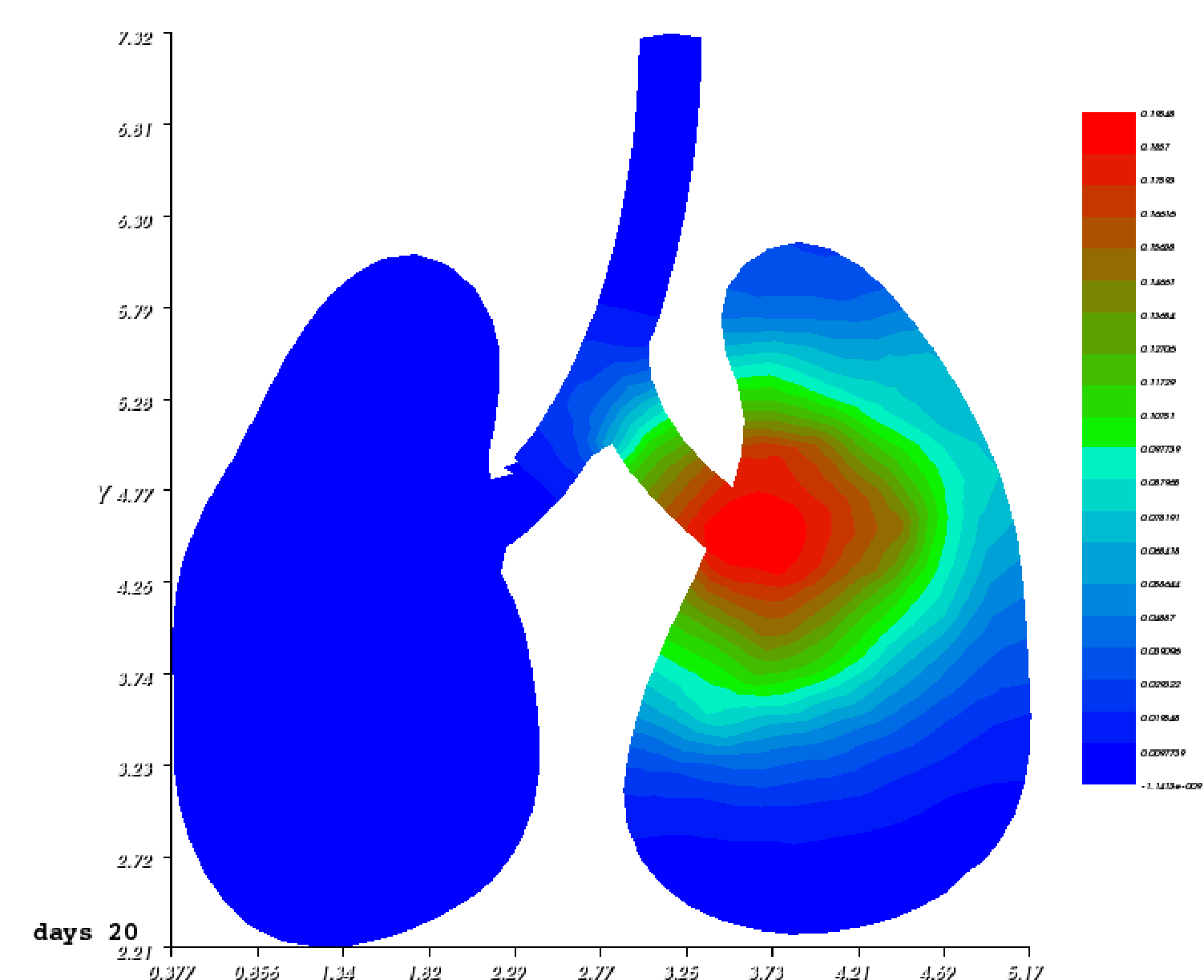}
         \caption{The tumour density evolution $\mathbf{u}$ without treatment, $i.e.,$ $\mathbf{u}  =\mathcal{F}(0)$ on 5th, 10th, 15th, 20th day, respectively, from left to right.}
      \label{2tumst}
   \end{figure}
     \begin{figure} 
\centering
        \includegraphics[width=0.29\textwidth]{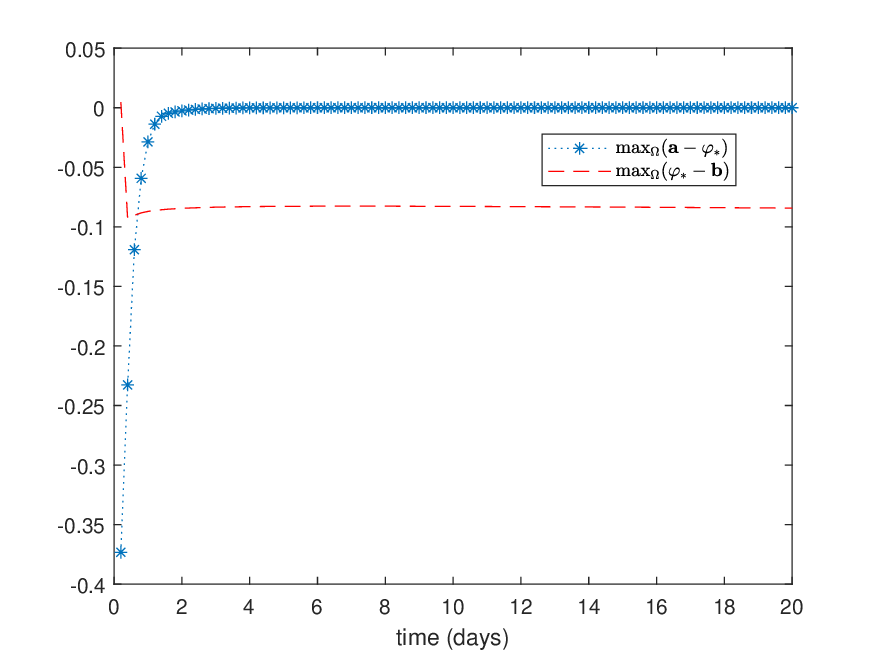}
        \includegraphics[width=0.29\textwidth]{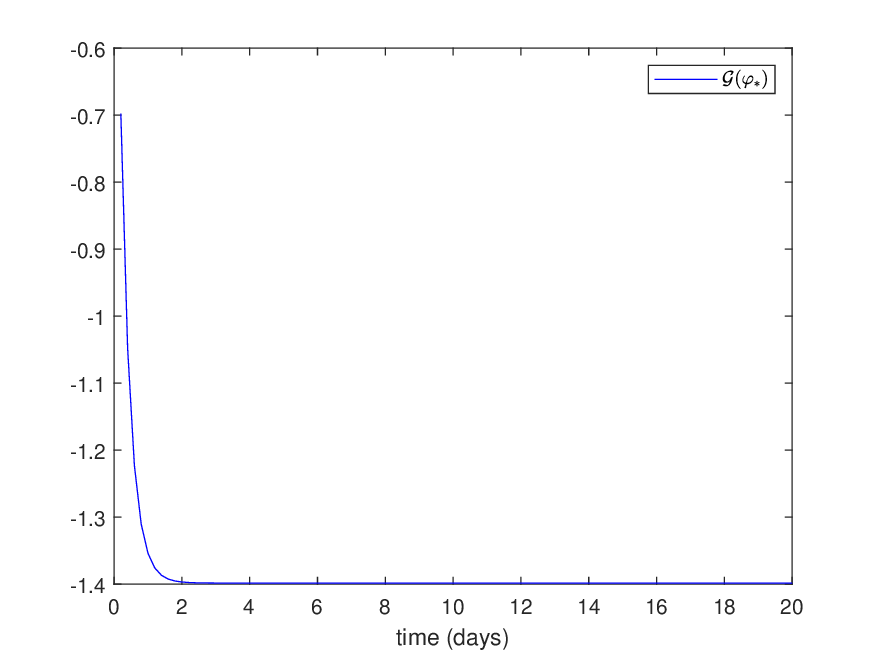}
         \caption{Evolution of the feasibility of the control $\varphi_{*}$ on the left and the state $\mathbf{u_{*}}$ on the right.}
         \label{figerac2}
   \end{figure}
    \begin{figure}
\centering
    \includegraphics[width=0.29\textwidth]{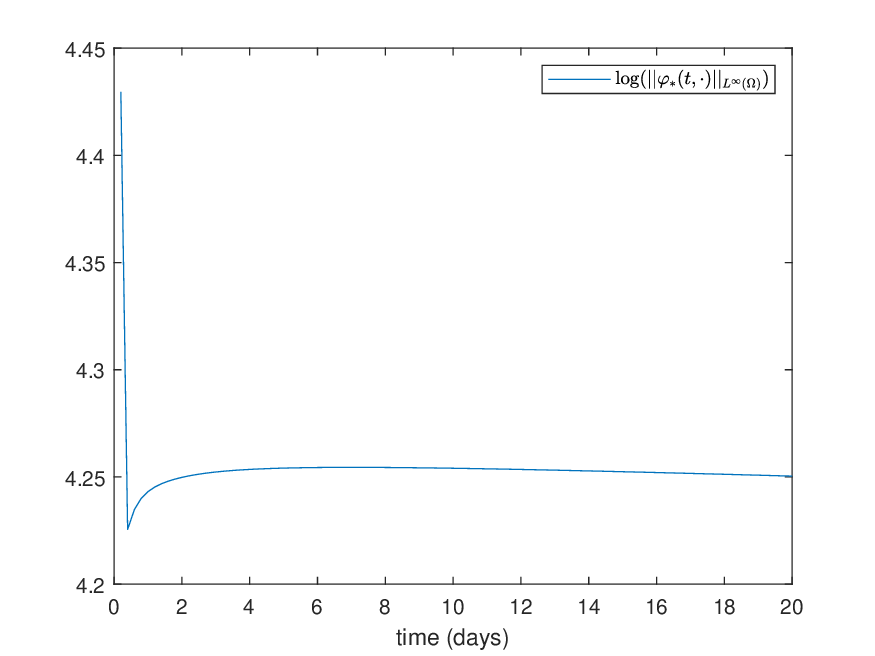}
        \includegraphics[width=0.29\textwidth]{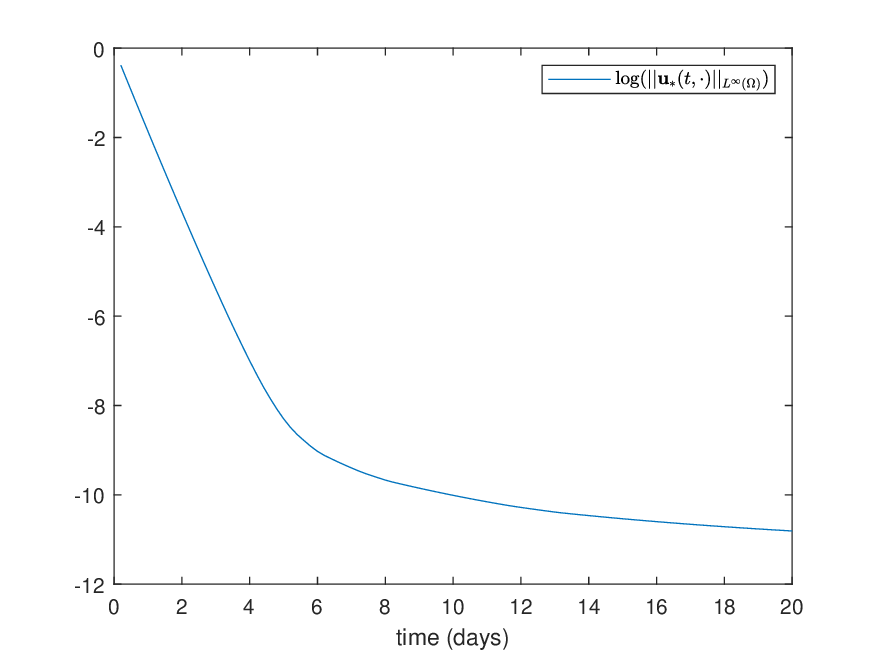}
           \caption{Evolution of $\log(\lVert\varphi_{*}\rVert_{\mL^{\infty}(\Omega)})$ on the left and $\log(\lVert\mathbf{u}_{*}\rVert_{\mL^{\infty}(\Omega)})$ on the right.}
 \label{2dec}
   \end{figure} 
       \begin{table}[htbp]
   \label{max tab2}
     \caption{The maximum and minimum value of the optimal tumour density.}
    \centering
   \begin{tabular}{|c|c|c|c|c|}
\hline
Time $(t)$ &  5{th} day&  10{th} day &  15{th} day & 20{th} day\\  
  \hline
 $\max_{\Omega}\mathbf{u_{*}}$& $0.00067$& $4.49\cdot 10^{-5}$& $2.65\cdot 10^{-5}$ & $2.09\cdot 10^{-5}$\\
  \hline  
  $\min_{\Omega}\mathbf{u_{*}}$& $1.17\cdot 10^{-7}$& $5.37\cdot 10^{-11}$& $6.83\cdot 10^{-12}$ & $1.037\cdot 10^{-12}$\\
  \hline
\end{tabular} 
\end{table}
\begin{figure}
\centering
        \includegraphics[width=0.22\textwidth]{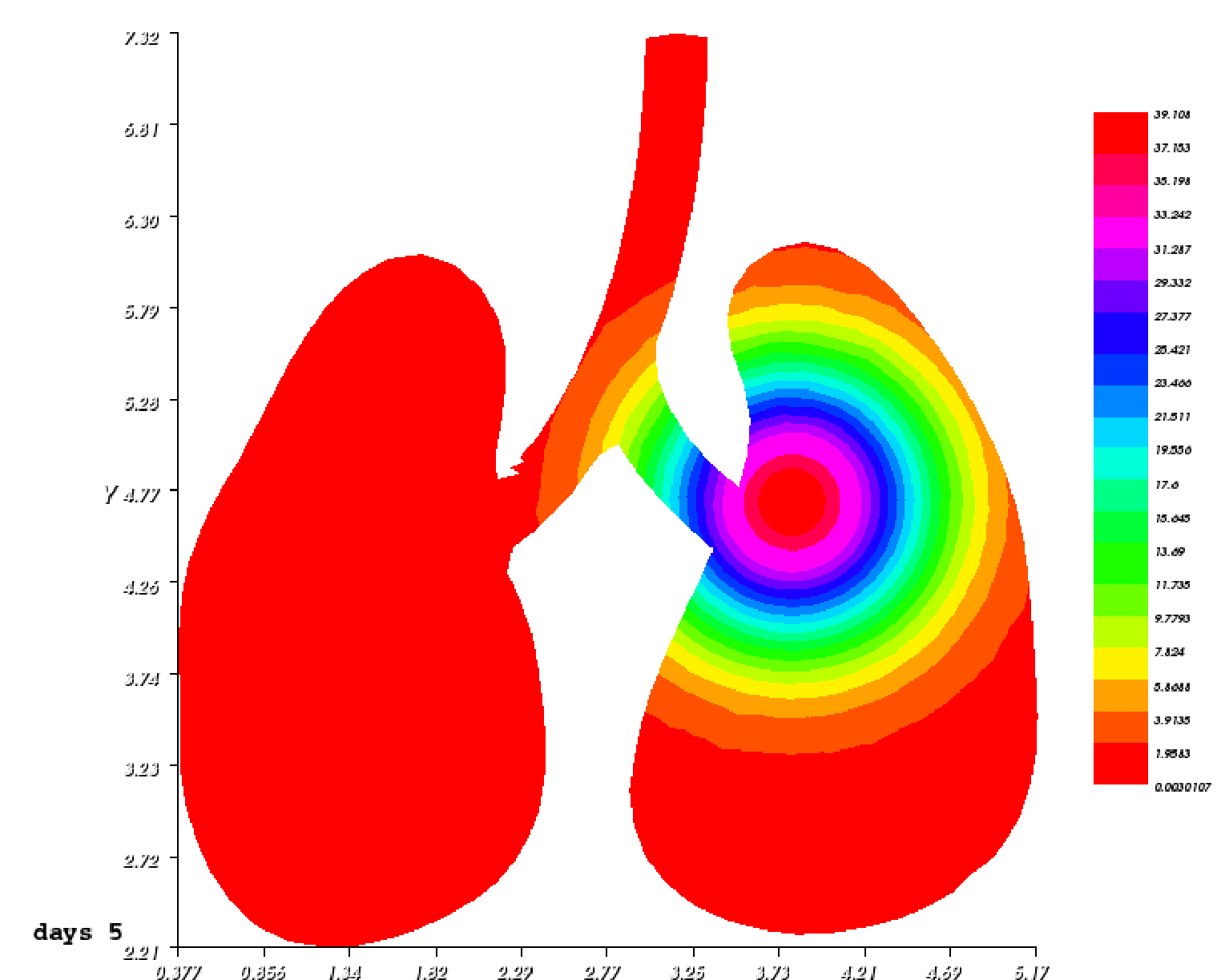}    
\includegraphics[width=0.22\linewidth]{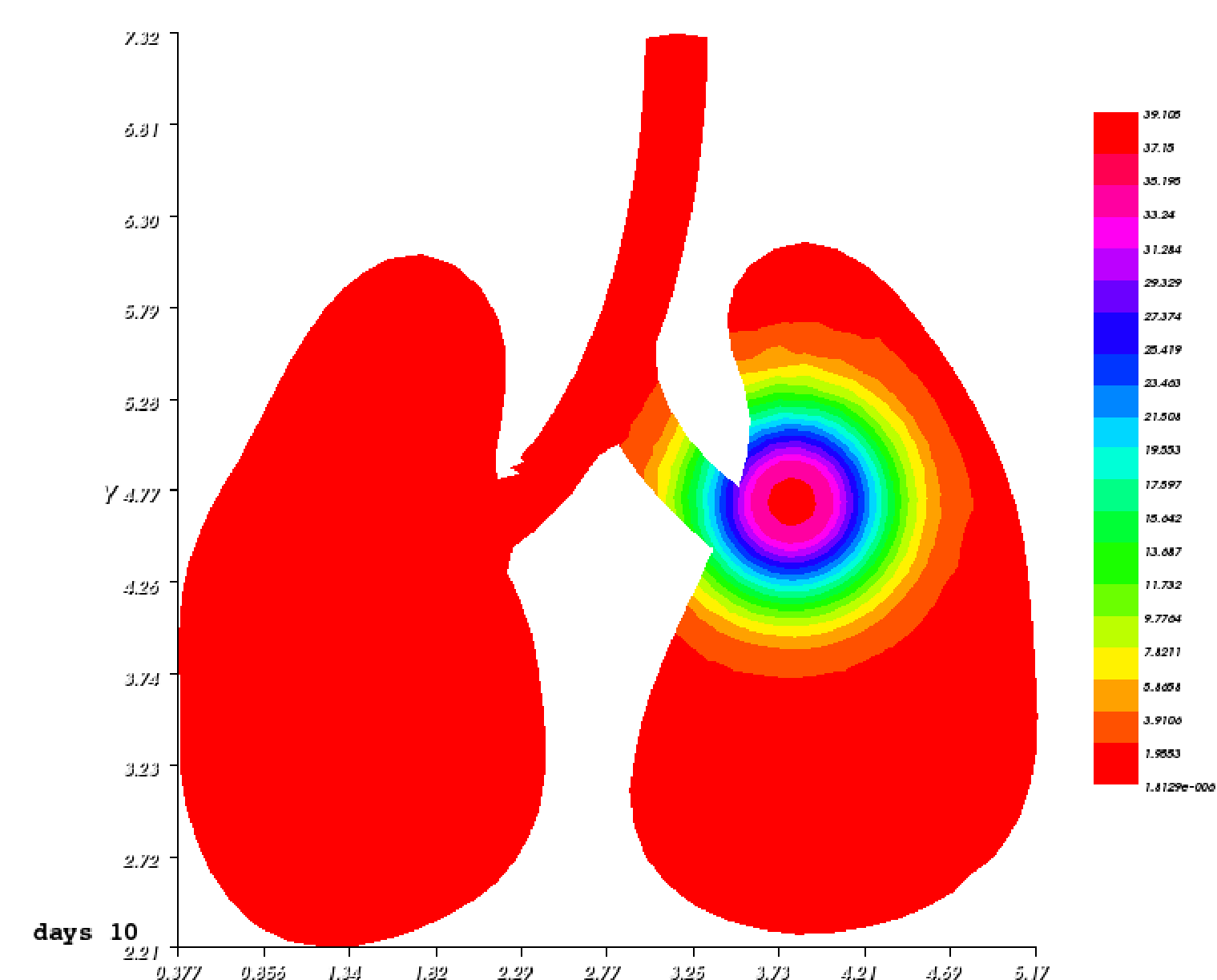}   
\includegraphics[width=0.22\textwidth]{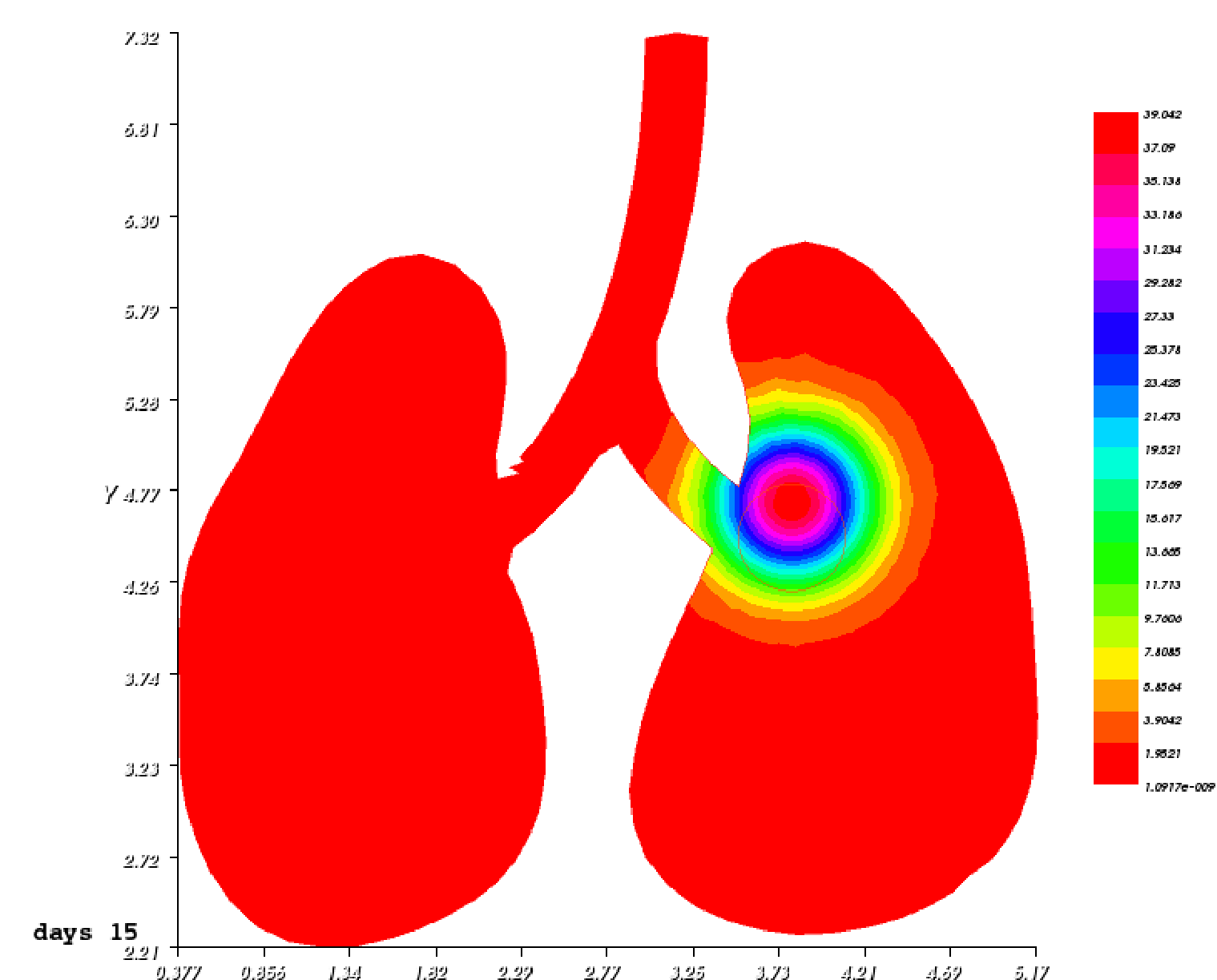}
\includegraphics[width=0.22\textwidth]{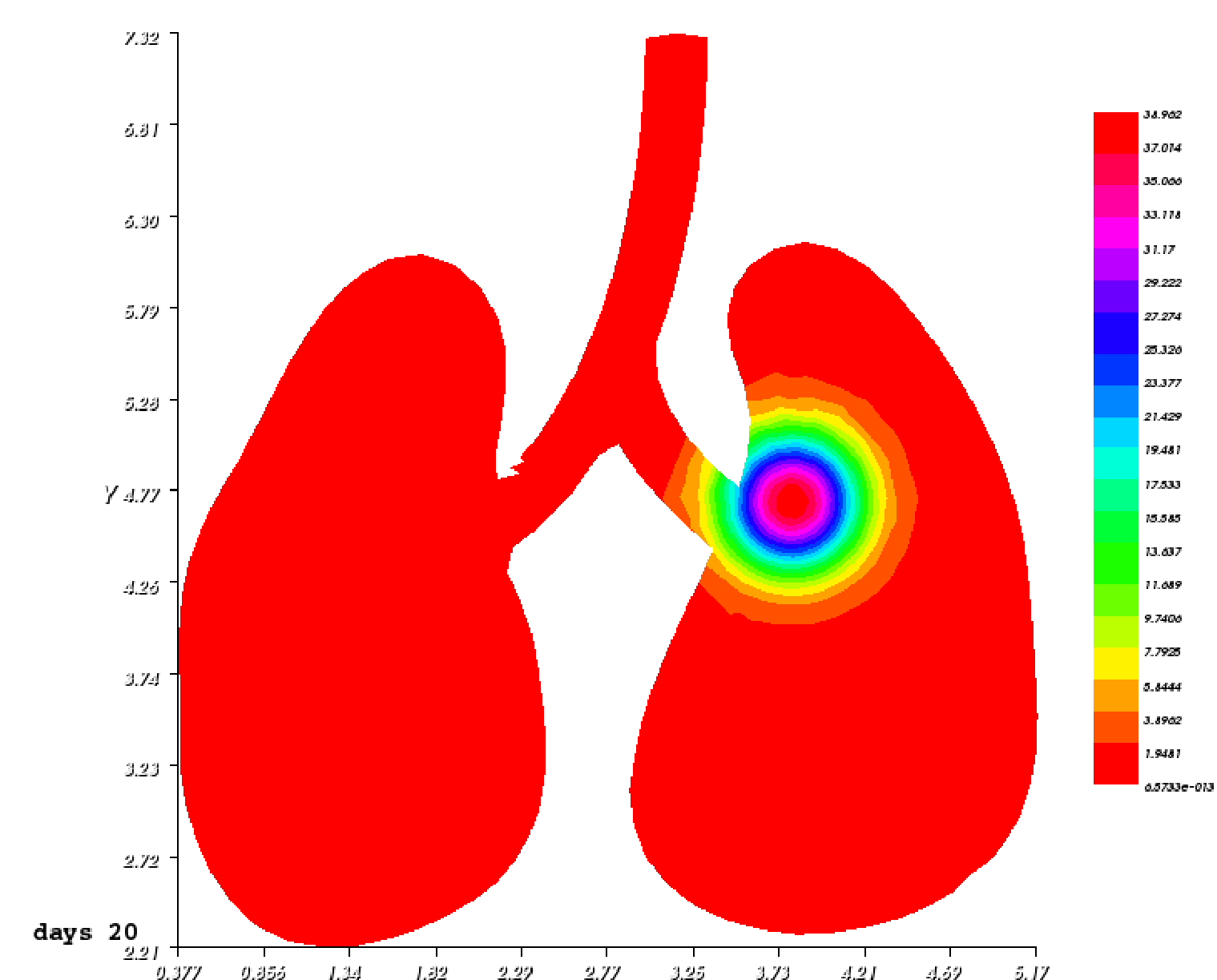} 
 \includegraphics[width=0.22\textwidth]{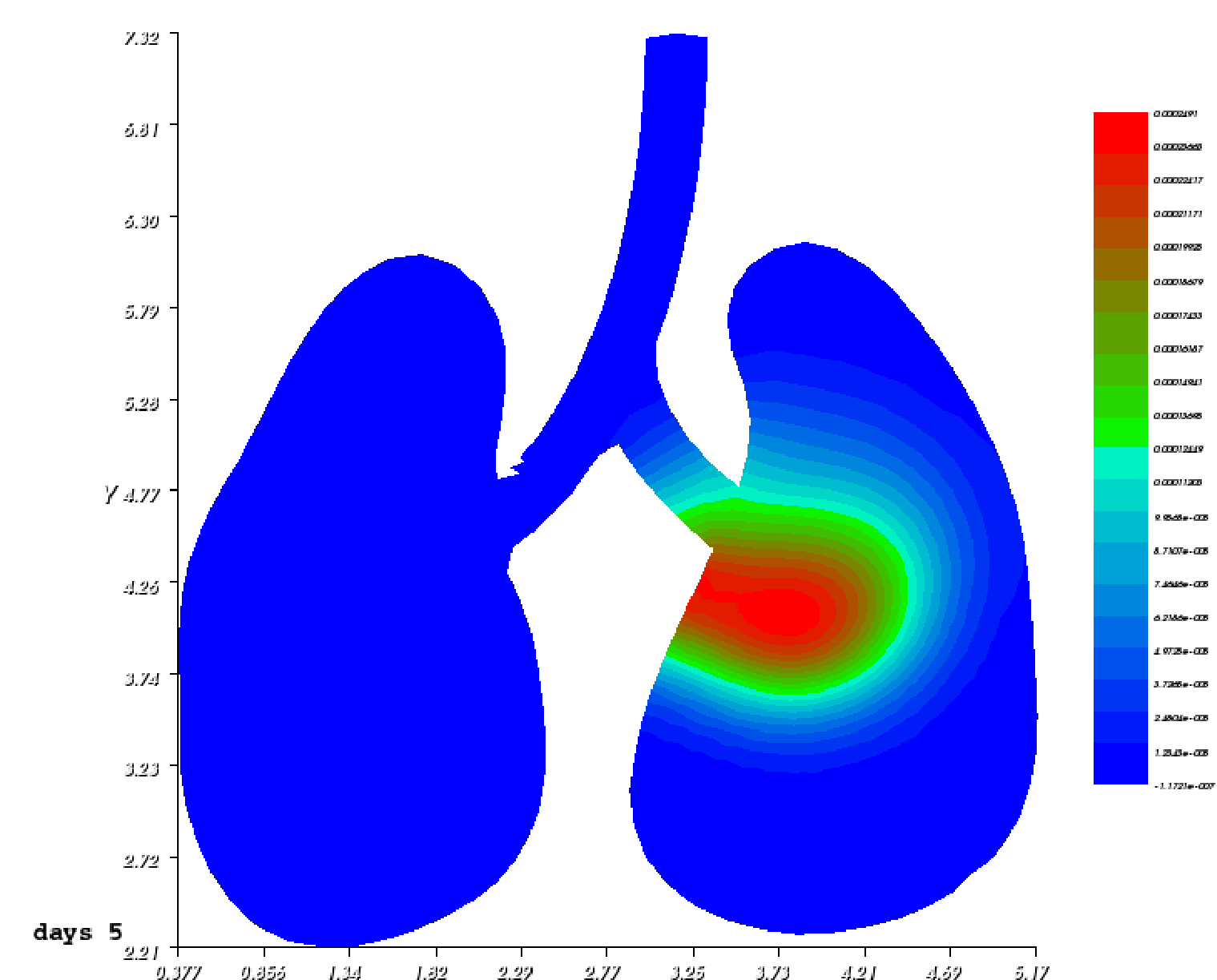}           
  \includegraphics[width=0.22\textwidth]{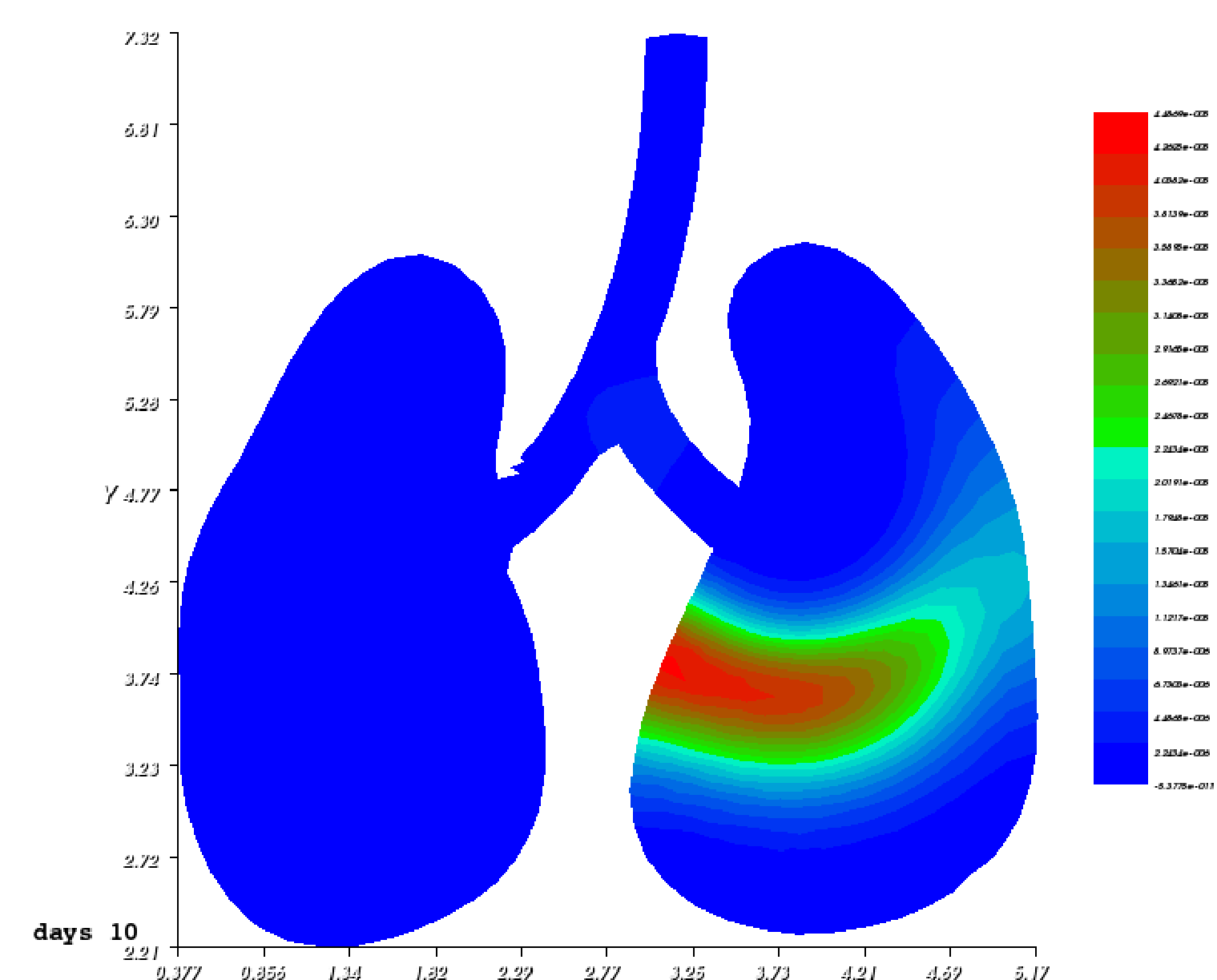}     
        \includegraphics[width=0.22\textwidth]{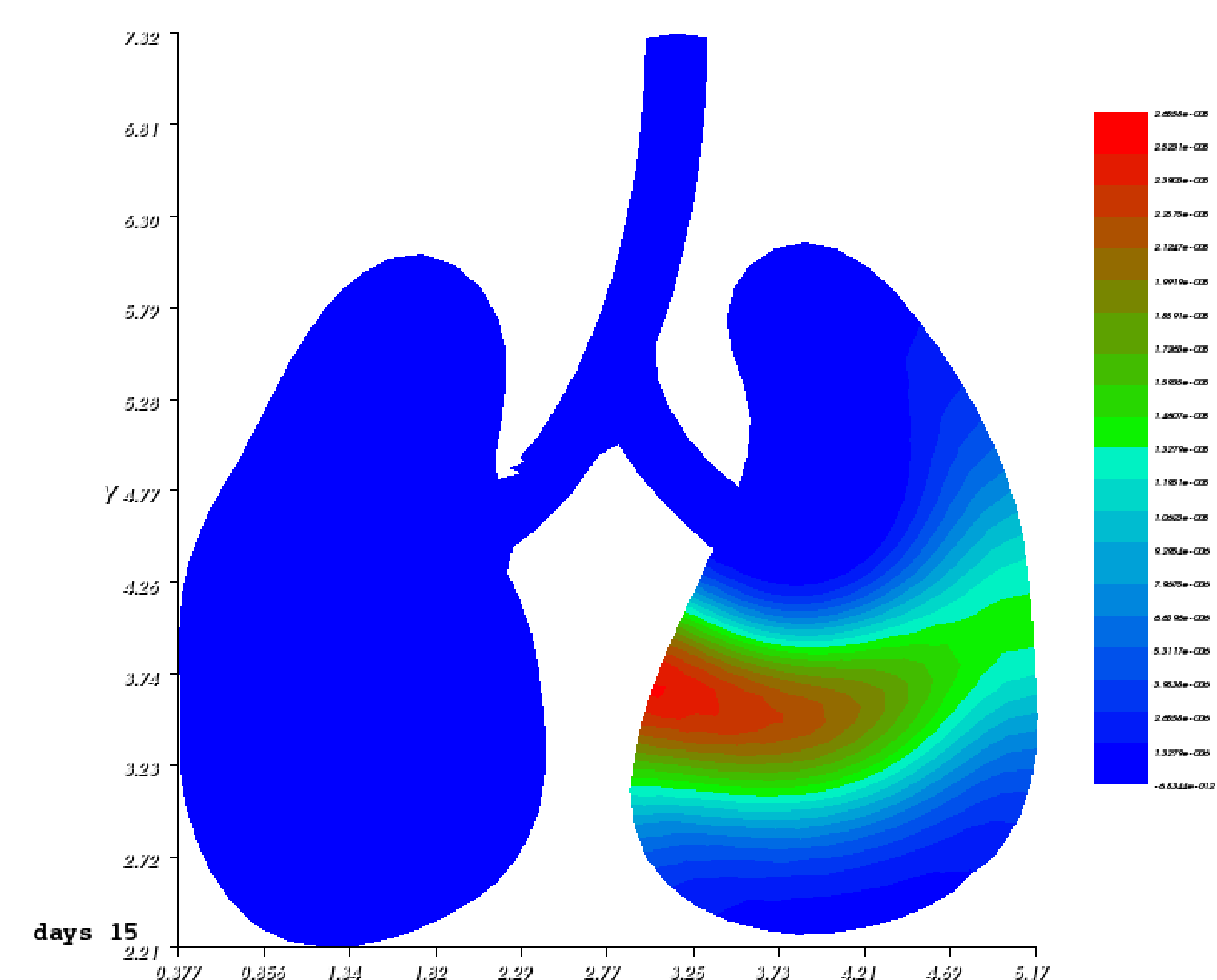}  
        \includegraphics[width=0.22\textwidth]{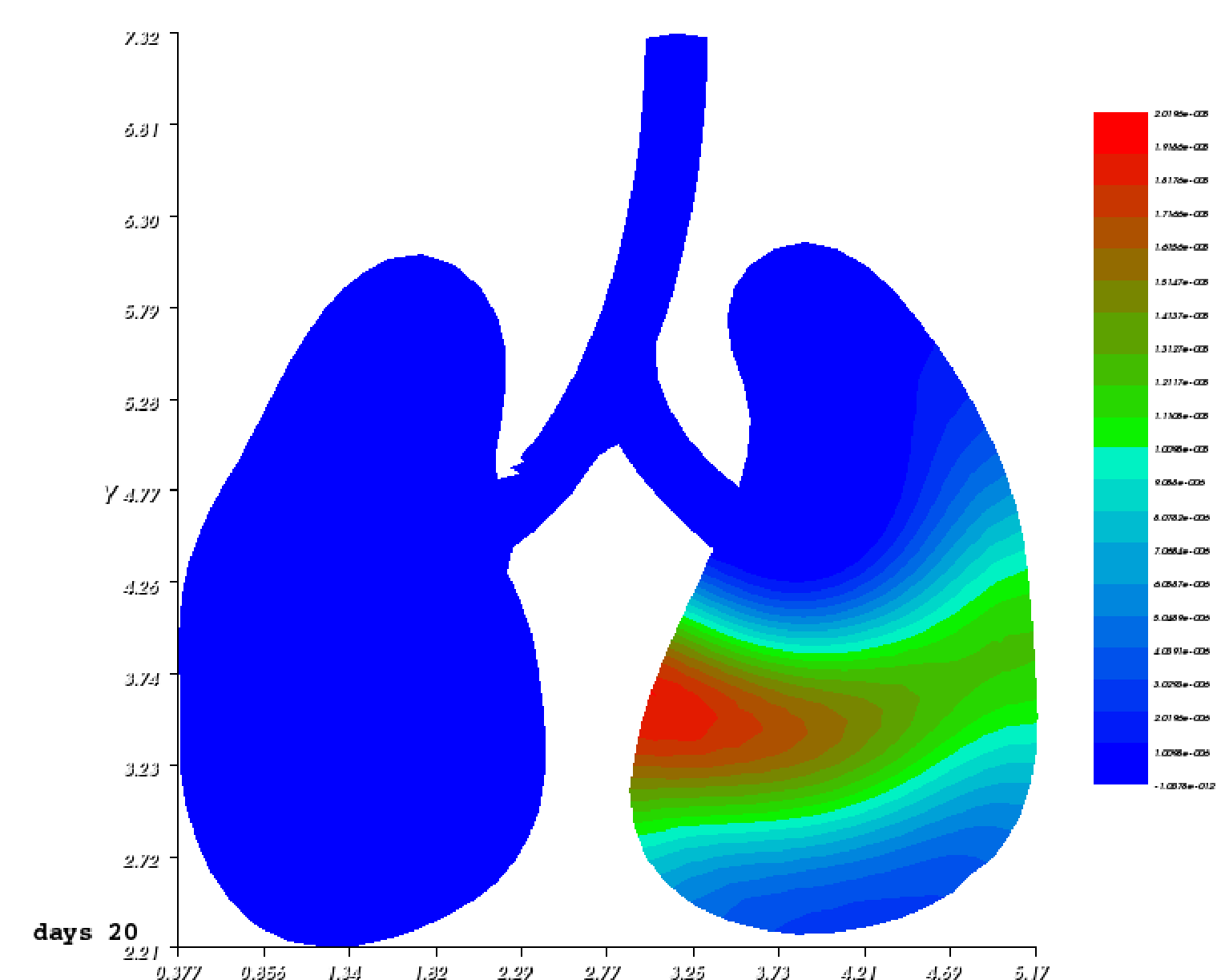} 
          \caption{
          The optimal drug $\varphi_{*}$ at the top and the optimal tumour density $ 
         \mathbf{u}_{*}$ at the bottom.
          }
     \label{2contAC}
   \end{figure} 
   \\
   Once again, for this case, all the figures are consistent and meet our expectations.
\section{Conclusion and Perspectives}\label{setting5}
This article presents a promising methodology and strategy for the optimal treatment of cancer with chemotherapy, through optimal control of drug concentration.
We adopted a nonlinear mathematical model of tumour response to treatment, coupled with an optimal control problem with state and control constraints.
The paper begins with a mathematical analysis of the model, followed by a detailed formulation of the optimal control problem and a discussion of the optimality conditions. 
Techniques for the numerical resolution of the control problem have been proposed (a penalty method is applied). Numerical simulations of optimal solutions to the control problem, aimed at eradicating two specific cases of lung cancer, confirm our expectations and highlight the importance of state constraints in personalized therapeutic processes.

In future work, we plan to adapt our approach to realistic protocols personalized with clinical data. From a clinical perspective, the chemotherapy plan is carefully structured according to a predetermined schedule specifying the treatment days. The duration and mode of administration depend on the drugs used and the specific therapeutic goals. Each treatment session is followed by a rest period, allowing the body to recover, leading to a discontinuous therapeutic approach. We will develop a therapeutic strategy that is both discontinuous in time and localized in space, targeting specific intervention time to treat various types of cancer using realistic and stochastic data.

\appendix
\section{The proofs of the Results of Section \ref{setting2}} \label{setting6}
\subsection{Proof of Theorem \ref{Théorème principal pour le bien posé} }\label{setting61}
The local well-posedness of the problem \eqref{Eqstate} is derived from a more general framework of \cite[Theo 2.2]{ACRB1999}. Before applying this result, we  prove that our problem \eqref{Eqstate} fits well into this framework by showing in the sense of Definition \ref{def epreg}, that $f_1$ is $\epsilon_1$-regular and $\tilde{f}_2$ is $\epsilon_2$-regular with respect to $(\mathcal{X}^{1}_q(\Omega),\mathcal{X}^{0}_q(\Omega))$ for some $\epsilon_1, \epsilon_2 \ge 0$, as stated in the following Lemma. 
\begin{lemma}
\label{epsgf1f2}
For a given function $\varphi$ in $\mathcal{U}_{ad}$, we have that the function $f_1$ is $\epsilon_1-regular,$ and $\tilde{f}_2$ is $\epsilon_2-regular$ in the sense of the definition \ref{def epreg}, where
\begin{enumerate}
\item[(i)] for $r=0:$ \\
$\epsilon_1=\frac{m}{2q}\Big(1-\frac{q}{\mathbf{p}_c(\rho_1-1)}\Big),$ $\gamma_1=\rho_1\epsilon_1 +1-\frac{m(\rho_1 -1)}{2q};$ $\epsilon_2=\frac{q+\rho_2 +1}{4q\rho_2},$ $\gamma_2=\rho_2\epsilon_2 +\frac{1}{2}-\frac{m(\rho_2 -1)}{2q}$,\\ with $ \min(\gamma_{1};\gamma_{2})=:\underline{\gamma}>\Bar{\epsilon}:=\max(\epsilon_{1};\epsilon_{2}).$
\item[(ii)] for $r=1$ and $q<m:$\\
$\epsilon_1=\frac{m-q}{2q}-\frac{m}{2\mathbf{p}_c(\rho_1 -1)},$ $\gamma_1=\rho_1\epsilon_1 +\frac{(m+q)-\rho_1(m-q)}{2q};$ $\epsilon_2=\frac{m+1}{2q\rho_2},$ $\gamma_2=\rho_2\epsilon_2 +\frac{m-\rho_2(m-q)}{2q}$ with $ \underline{\gamma}>\Bar{\epsilon}.$
\item[(iii)] for $r=1$ and $q\ge m:$\\
$\epsilon_1=0,$ $\gamma_1=\frac{1}{2};$ $\epsilon_2=\frac{m+1}{4q},$ $\gamma_2=\frac{m}{2q}$ with $ \underline{\gamma}>\Bar{\epsilon}.$
\end{enumerate}
\end{lemma}
\begin{pf}
Let $\varphi$ be in $\mathcal{U}_{ad}.$
\begin{enumerate}
\item[(i)]  Case $r=0.$ Let's
\begin{gather*}
 \epsilon_1=\frac{m}{2q}\Big(1-\frac{q}{\mathbf{p}_c(\rho_1-1)}\Big), \gamma_1=\rho_1\epsilon_1 +1-\frac{m(\rho_1 -1)}{2q}, s_1=\frac{mq}{m-2q(\gamma_1-1)}, \ell_1=\frac{mq}{m-2q\epsilon_{1}}.
 \end{gather*}
 According to assumptions $(I_1)-\rm{(i)}$  and $(I_2)-\rm{(i)}$ in \textit{\textbf{(H1)}}-\textit{\textbf{(H2)}}, we have
   $\epsilon_1\in]0;\frac{m}{2q}[,$ $\gamma_1\in[\rho_1\epsilon_1;1[,$ $\ell_1=s_1\rho_1$ and $\mathbf{p}_c=s_1\rho^{*}_1.$ 
    Now, let $u,v$ be in $\mathcal{X}^{1+\epsilon_1}_q(\Omega),$
for all $t>0$
$$\lVert f_{1}(t,\x,\varphi,u)\lVert_{\mathcal{X}^{\gamma_1}_q(\Omega)}\leq \lVert \alpha_{0}\varphi(t,\cdot)u(t,\cdot)\lVert_{\mathcal{X}^{\gamma_1}_q(\Omega)} +\lVert\mK_{2}(t,\cdot)u(t,\cdot)\lVert_{\mathcal{X}^{\gamma_1}_q(\Omega)}+\lVert\mK_{1}(t,\cdot)\text{h}_{1}(u(t,\cdot))\lVert_{\mathcal{X}^{\gamma_1}_q(\Omega)}.$$
From $\textit{\textbf{(H2)}},$ we have $-\frac{m}{2q^{*}}<\gamma_1 -1\leq 0,$ then using Hölder's inequality and  injections \eqref{injection xlq}:
\begin{gather}
\label{phiestim}
     \begin{split}
     \lVert\alpha_{0}\varphi(t,\cdot)u(t,\cdot)\lVert_{\mathcal{X}^{\gamma_1}_q(\Omega)}&\leq  C\bigg(\int_{\Omega}|\varphi(t,\x)|^{s_1}|u(t,\x)|^{s_1}\d x\bigg)^{\frac{1}{s_1}}\\
     &\leq C\bigg(\int_{\Omega}|\varphi(t,\x)|^{s_1 \rho^{*}_1}\d\x\bigg)^{\frac{1}{\rho^{*}_1 s_1}}\bigg(\int_{\Omega}|u(t,\x)|^{s_1\rho_1}\d \x\bigg)^{\frac{1}{\rho_1 s_1}}\\
     &\leq C\lVert \varphi\lVert_{\mL^{\infty}(0, T; \mL^{\mathbf{p}_c} (\Omega))} \lVert u(t,\cdot)\lVert_{\mL^{\ell_1} (\Omega)}\\
     &\leq C\lVert \mathbf{b}\lVert_{\mL^{\mathbf{p}_{c}}(\Omega)} \lVert u\lVert_{\mathcal{X}^{1+\epsilon_{1}}_q(\Omega)}\\
     &\leq C\lVert u\lVert_{\mathcal{X}^{1+\epsilon_{1}}_q(\Omega)}.
     \end{split}
 \end{gather}
 Similarly, by using \eqref{injection xlq}, Hölder's inequality, and the definition of $\text{h}_{1}$ in \eqref{hi}, we obtain
\begin{gather*}
  \lVert\mK_1(t,\cdot)\text{h}_{1}(u(t,\cdot))\lVert_{\mathcal{X}^{\gamma_1}_q(\Omega)}\leq C\lVert u\rVert^{\rho_1}_{\mathcal{X}^{1+\epsilon_{1}}_q(\Omega)},\\
  \lVert\mK_2(t,\cdot)u(t,\cdot)\lVert_{\mathcal{X}^{\gamma_1}_q(\Omega)}\leq C\lVert u\rVert_{\mathcal{X}^{1+\epsilon_{1}}_q(\Omega)}.
  \end{gather*}
  Then, the nonlinear operator \( f_{1}(.,.,\varphi,.): ]0; \infty[ \times \Omega \times \mathcal{X}^{1+\epsilon_{1}}_q(\Omega) \longrightarrow \mathcal{X}^{\gamma_{1}}_q(\Omega) \) is well-defined. Moreover, by applying Hölder's inequality, injections \eqref{injection xlq} and the following inequality,
\begin{gather}
\label{estim classic}
    ||a|^{\rho}-|b|^{\rho}|\leq C|a-b|(|a|^{\rho-1}+|b|^{\rho-1}+1),\mbox{ $\forall(a,b)\in\RR^{2}$ and $\rho> 1,$ }
    \end{gather}
    we obtain ($\forall u,v\in\mathcal{X}^{1+\epsilon_{1}}_q(\Omega)$ with $\nu_1(t)=Ct^{\gamma_1-\epsilon_{1}}$): 
    \begin{gather}
 \lVert f_1(t,\x,\varphi,u)\!-\!f_1(t,\x,\varphi,v)\lVert  _{\mathcal{X}^{\gamma_1}_q(\Omega)}\!\leq C\lVert  u\!-\!v\lVert_{\mathcal{X}^{1+\epsilon_{1}}_q(\Omega)}\big( \lVert  u\lVert^{\rho_{1}-1}_{\mathcal{X}^{1+\epsilon_{1}}_q(\Omega)}\!+\lVert  v\lVert^{\rho_{1}-1}_{\mathcal{X}^{1+\epsilon_{1}}_q(\Omega)}\!+\nu_1(t)t^{\epsilon_{1}-\gamma_1}\!\big).
    \end{gather}
It implies that $f_1$ is $\epsilon_{1}-reguar$ relative to $(\mathcal{X}^{1}_q(\Omega),\mathcal{X}^{0}_q(\Omega)).$\\
 Furthermore, putting $\epsilon_2=\frac{q+\rho_2 +1}{4q\rho_2},$ $\gamma_2=\rho_2\epsilon_2 +\frac{1}{2}-\frac{m(\rho_2 -1)}{2q}$, and using $\textit{\textbf{(H1)}}-\textit{\textbf{(H2)}},$  we obtain $\epsilon_2\in]\frac{1}{2q};\frac{m}{2q}[$ and $\gamma_2\in]\rho_2\epsilon_2;1[.$\\
 Let $u$ be in $\mathcal{X}^{1+\epsilon_2}_{q}(\Omega)$ and $\psi$ be in $\mathcal{X}^{-\gamma_2}_{q^{*}}(\Omega).$ Setting $\iota=\frac{(m-1)q}{\rho_{2}(m-2q\epsilon_{2})}$ and using \eqref{opérateur trace}, we can deduce that
\begin{gather*}
\begin{split}
 &\mathcal{X}^{1+\epsilon_2}_{q}(\Omega)\hookrightarrow \mH^{2\epsilon_2}_{q}(\Omega),\\
 &\gamma_{\Gamma}: \mH^{2\epsilon_2}_{q}(\Omega){\longrightarrow}\mL^{\omega\rho_2 }(\Gamma),\\
 & \mathcal{X}^{-\gamma_2}_{q^{*}}(\Omega)=\mathcal{E}^{1-\gamma_2}_{q^{*}}(\Omega)\hookrightarrow \mH^{2(1-\gamma_2)}_{q^{*}}(\Omega),\\
 & \gamma_{\Gamma}: \mH^{2(1-\gamma_2)}_{q^{*}}(\Omega){\longrightarrow}\mL^{\iota^{*}}(\Gamma).
  \end{split}
 \end{gather*}
Using Hölder's inequality, and the definition of $\text{h}_{2}$ in \eqref{hi}, we have (since $f_2$ is the restriction of $\tilde{f}_{2}$ on $\Gamma$)
\begin{gather*}
    \begin{split}
   \langle \tilde{f}_{2}(t,\x,u),\psi \rangle_{\mathcal{X}^{\gamma_2}_{q}(\Omega),\mathcal{X}^{-\gamma_2}_{q^{*}}(\Omega)}&=\int_{\Gamma}f_2(t,\x,u)\cdot\psi(\x)\d\x\\
   &=\int_{\Gamma}\mK_3(t,\x)\text{h}_{2}(u(t,\x))\cdot\psi(\x)\d\x\\
   &\leq \int_{\Gamma}|\mK_3(t,\x)\lVert  u|^{\rho_2}\cdot|\psi(\x)|\d\x\\
   &\leq C\int_{\Gamma}|u|^{\rho_2}\cdot|\psi(\x)|\d\x\\
   &\leq C\lVert u\lVert^{\rho_2}_{\mL^{\iota\rho_2}(\Gamma)}\lVert  \psi\lVert_{\mL^{\iota^{*}}(\Gamma)}\\
   &\leq C\lVert u\lVert^{\rho_2}_{\mathcal{X}^{1+\epsilon_2}_{q}(\Omega)}\lVert  \psi\lVert_{\mathcal{X}^{-\gamma_2}_{q^{*}}(\Omega)}, \mbox{ $\forall\psi\in\mathcal{X}^{-\gamma_2}_{q^{*}}(\Omega).$ }
    \end{split}
\end{gather*} 
Then, we get
\begin{gather*}
\lVert\tilde{f}_{2}(t,\x,u)\lVert_{\mathcal{X}^{\gamma_2}_{q}(\Omega)}\leq C\lVert u\lVert^{\rho_2}_{\mathcal{X}^{1+\epsilon_2}_{q}(\Omega)}.
\end{gather*}
Using \eqref{estim classic}, we can deduce that $\tilde{f}_{2}$ is $\epsilon_{2}-regular$ relative to $(\mathcal{X}^{1}_q(\Omega),\mathcal{X}^{0}_{q}(\Omega)).$ 
Moreover, from $\textit{\textbf{(H1)}}-\textit{\textbf{(H2)}},$  we have $\epsilon_1<\epsilon_2$ and $\gamma_2<\gamma_1.$ Given that \(\epsilon_2 < \rho_2 \epsilon_2 \leq \gamma_2\), we can conclude that \(\underline{\gamma} = \gamma_2 > \epsilon_2 = \bar{\epsilon}\). 
\item[(ii)]  Case $r=1$ and $q<m.$ By setting
\begin{gather*}
 \epsilon_1=\frac{m-q}{2q}-\frac{m}{2\mathbf{p}_c(\rho_1-1)}, \gamma_1=\rho_1\epsilon_1 +\frac{(m+q)-\rho_1(m-q)}{2q}, s_1=\frac{mq}{m+q-2q\gamma_1},\\
  \ell_1=\frac{mq}{m-q-2q\epsilon_{1}},\epsilon_2=\frac{m+1}{2q\rho_2}, \gamma_2=\rho_2\epsilon_2 +\frac{m-\rho_2(m-q)}{2q},
 \end{gather*}
  and from $\textit{\textbf{(H1)}}-\textit{\textbf{(H2)}},$ we have $\epsilon_1\in[0;\frac{m-q}{2q}[,$ $\gamma_1\in[\rho_1\epsilon_1;1[,$ $\ell_1=s_1\rho_1,$ $\mathbf{p}_c=s_1\rho^{*}_1,$ $\epsilon_2\in]\frac{1}{2q};\frac{m}{2q}]$ and $\gamma_2\in[\rho_2\epsilon_2;1[.$ Using injections \eqref{injection wlq}, Hölder's inequality and the relation  \eqref{estim classic}, we can proceed similarly to the previous case and we can conclude that $f_1$ is $\epsilon_1-regular$, and $\tilde{f}_{2}$ is $\epsilon_{2}-regular$ relative to $(\mathcal{X}^{1}_q(\Omega),\mathcal{X}^{0}_q(\Omega))$, and that $\underline{\gamma}=\gamma_2>\epsilon_2=\Bar{\epsilon}.$
 \item[(iii)]  Case $r=1$ and $q\ge m.$ By setting $\epsilon_1=0,$ $\gamma_1=\frac{1}{2},$ $s_1=q,$ $\ell_1=\rho_1 q,$ $\epsilon_2=\frac{m+1}{4q},$ $\gamma_2=\frac{m}{2q}$, and proceeding similarly to the first case, we can deduce that 
$f_1$ is $\epsilon_1-regular$ and $\tilde{f}_{2}$ is $\epsilon_{2}-regular$ relative to $(\mathcal{X}^{1}_q(\Omega),\mathcal{X}^{0}_q(\Omega))$ and that  $\underline{\gamma}=\frac{m}{2q}>\frac{m+1}{4q}=\Bar{\epsilon}$. 
\end{enumerate}
This completes the proof.
\end{pf}
Now, we can give the proof of Theorem \ref{Théorème principal pour le bien posé}.

Let $\varphi$ be in $\mathcal{U}_{ad}$ and $u_0$ be in $\mW^{r,q}(\Omega)$ $(r\in\{0,1\}).$ According to Lemma \ref{epsgf1f2}, the function $f_1$ is $\epsilon_1 -regular$ and  $\tilde{f}_2$ is $\epsilon_2 -regular$ relative to $(\mathcal{X}^{1}_q(\Omega),\mathcal{X}^{0}_q(\Omega))$ for given $\epsilon_1,\epsilon_2$ such that $\underline{\gamma}>\Bar{\epsilon}.$ By applying \cite[Theo 2.2]{ACRB1999},  there exists a time $\tau_0>0$ and a unique $\Bar{\epsilon}-regular$ mild solution $\mathbf{u}$ on $[0;\tau_0]$ to the problem \eqref{Eqstate} and that $\mathbf{u}$ is the classical solution of \eqref{Eqstate} on $[0;\tau_0],$ which concludes the local existence solution part of the Theorem \ref{Théorème principal pour le bien posé}.
The global existence  of $\mathbf{u}$ at any finite horizon $T>0,$ can be obtained using the same technique as that used in \cite[Sec. 8]{RodB2002}. We give the general idea while omitting the details of the proof.

First, we fix a finite horizon $T$ and for any time $t\in (0; T),$ we establish an a priori estimate on $\mathbf{u}$ under a condition depending only on the data of \eqref{Eqstate}. In our case, when $\rho_1>2\rho_2$, we establish an a priori estimate without any condition. If $\rho_1 +1=2\rho_2$, a priori estimate of the solution can be established when $ 4\nu_0(q-1)\inf_{\mathcal{Q}}{\mK_{1}}\ge c_{\Omega}(\inf_{\Sigma}{\mK_{3}})^{2}(\rho_2 +q-1)^{2},$ $c_{\Omega}$ is a constant dependent only on $\Omega.$ Then, from the a priori estimates established, we easily show that in both cases that $\lVert\mathbf{u}(t)\rVert_{\mL^{q}(\Omega)}$ is bounded for a finite time and $\nabla(|\mathbf{u}|^{\frac{q}{2}})\in\mL^{2}((0; T)\times\Omega)$ for any $T>0,$ which concludes the proof.
\begin{rmk}
We do not consider the case  $\rho_1 +1<2\rho_2$, as this case may lead to blow-up phenomena  (cf. \cite{RodB2002}).
\end{rmk}
\subsection{Proof of Theorem \ref{Théorème principal pour le bien posé 2}}\label{preuve bien pose2}
We prove the uniqueness of the solution and stability of the problem \eqref{Eqstate}. For this, let $\mathbf{u}_1,$ $resp.$ $\mathbf{u}_2$ the classical solution of \eqref{Eqstate} with respect to $(u_{0,1},\varphi_{1}),$ $resp.$ $(u_{0,2},\varphi_{2})\in \mW^{r,q}(\Omega)\times\mathcal{U}_{ad}$ ($r\in\{0,1\}$).
  Then, according to \eqref{mild form}, we have ($\forall t\in (0; T)$): 
\begin{gather*}
\mathbf{u}_1(t,\x)-\mathbf{u}_2(t,\x)=\e^{\tilde{A}t}(u_{0,1}-u_{0,2})+\int^{t}_{0}\e^{\tilde{A}(t-s)}\Big(f_{1}(s,\x,\varphi_1(s),\mathbf{u}_1 (s))-f_{1}(s,\x,\varphi_2(s),\mathbf{u}_2 (s))\\
+\tilde{f}_{2}(s,\x,\mathbf{u}_1)-\tilde{f}_{2}(s,\x,\mathbf{u}_2)\Big)\d s.
\end{gather*}
\begin{enumerate}
\item Let \(\theta \in [0, \underline{\gamma}[\), using Hölder's inequality, we obtain
\begin{gather}
\label{dif0}
\begin{split}
t^{\theta}\lVert \mathbf{u}_1(t,\cdot)&-\mathbf{u}_2(t,\cdot)\lVert_{\mathcal{X}^{1+\theta}_q(\Omega)}\leq t^{\theta}\lVert \e^{\tilde{A}t}(u_{0,1}-u_{0,2})\lVert_{\mathcal{X}^{1+\theta}_q(\Omega)}\\
& +Ct^{\theta}\int^{t}_0 \lVert \e^{\tilde{A}(t-s)}\Big(f_{1}(s,\x,\varphi_1(s),\mathbf{u}_1 (s))-f_{1}(s,\x,\varphi_2(s),\mathbf{u}_2 (s))\Big)\lVert_{\mathcal{X}^{1+\theta}_q(\Omega)}\d s\\
&+Ct^{\theta}\int^{t}_0 \lVert \e^{\tilde{A}(t-s)}\Big(\tilde{f}_{2}(s,\x,\mathbf{u}_1 (s))-\tilde{f}_{2}(s,\x,\mathbf{u}_2 (s))\Big)\lVert_{\mathcal{X}^{1+\theta}_q(\Omega)}\d s.
\end{split}
\end{gather}
Moreover, according to \eqref{termes sources principales}: 
\begin{gather*}
\begin{split}
f_{1}(s,\x,&\varphi_1(s),\mathbf{u}_1(s))-f_{1}(s,\x,\varphi_2(s),\mathbf{u}_2 (s))=-\mK_1(s,\x)\big(\text{h}_{1}(u_{1}(s))-\text{h}_{1}(u_{2}(s))\big)\\
&+\mK_{2}(s,\x)\big(\mathbf{u}_1(s)-\mathbf{u}_2(s)\big)
-\alpha_{0}\varphi_{1}(s,\x)\big(\mathbf{u}_1(s) -\mathbf{u}_2(s)\big)
-\alpha_{0}\big(\varphi_1(s,\x) -\varphi_2(s,\x)\big)\mathbf{u}_2(s).
\end{split}
\end{gather*}
Then, using Definition \ref{def epreg}, and \eqref{phiestim}, we have
\begin{gather}
\label{dif1}
\begin{split}
&\lVert f_1(s,\varphi_1(s),\mathbf{u}_1 (s))-f_1(s,\varphi_2,\mathbf{u}_2 (s))\lVert_{\mathcal{X}^{\underline{\gamma}}_q(\Omega)}\\
&\leq\!\!C\lVert \mathbf{u}_1(s)\!\!-\!\!\mathbf{u}_2(s)\lVert_{\mathcal{X}^{1+\Bar{\epsilon}}_q(\Omega)}\Big(\!\lVert \mathbf{u}_1(s)\lVert^{\rho_1 -1}_{\mathcal{X}^{1+\Bar{\epsilon}}_q(\Omega)}\!\!+\!\lVert \mathbf{u}_2(s)\lVert^{\rho_1 -1}_{\mathcal{X}^{1+\Bar{\epsilon}}_q(\Omega)}
+ 2\!\Big)
+C\lVert \varphi_{1} -\varphi_{2} \Vert_{\mathcal{U}_{ad}}\lVert \mathbf{u}_2(s)\lVert_{\mathcal{X}^{1+\Bar{\epsilon}}_q(\Omega)}
\\
&\leq 2C\lVert \mathbf{u}_1(s)-\mathbf{u}_2(s)\lVert_{\mathcal{X}^{1+\Bar{\epsilon}}_q(\Omega)}(1+s^{-\Bar{\epsilon}(\rho_1 -1)})
+ Cs^{-\Bar{\epsilon}}\lVert \varphi_{1} -\varphi_{2} \Vert_{\mathcal{U}_{ad}}.
\end{split}
\end{gather} 
Similarly, we have
 \begin{gather}
 \label{dif2}
 \lVert \tilde{f}_{2}(s,u_1 (s))-\tilde{f}_{2}(s,u_2 (s))\lVert_{\mathcal{X}^{\underline{\gamma}}_q(\Omega)}
\leq C\lVert u_1(s)-u_2(s)\lVert_{\mathcal{X}^{1+\Bar{\epsilon}}_q(\Omega)}(1+2s^{-\Bar{\epsilon}(\rho_2 -1)}).
 \end{gather}
Substituting \eqref{dif1} and \eqref{dif2} in \eqref{dif0},  we get
 \begin{gather}
 \label{dif3}
\begin{split}
t^{\theta}\lVert& \mathbf{u}_1(t,\cdot)-\mathbf{u}_2(t,\cdot)\lVert_{\mathcal{X}^{1+\theta}_q(\Omega)}\leq t^{\theta}\lVert \e^{\tilde{A}t}(u_{0,1}-u_{0,2})\lVert_{\mathcal{X}^{1+\theta}_q(\Omega)}\\
&+
CMt^{\theta}\int^{t}_0 (t-s)^{\underline{\gamma}-(1+\theta)}\lVert f_{1}(s,\x,\varphi_1(s),\mathbf{u}_1(s))-f_{1}(s,\x,\varphi_2(s),\mathbf{u}_2(s))\lVert_{\mathcal{X}^{\underline{\gamma}}_q(\Omega)}\\
&+CMt^{\theta}\int^{t}_0 (t-s)^{\underline{\gamma}-(1+\theta)} \lVert \tilde{f}_{2}(s,\x,\mathbf{u}_1 (s))-\tilde{f}_{2}(s,\x,\mathbf{u}_2 (s))\lVert_{\mathcal{X}^{\underline{\gamma}}_q(\Omega)}
 \\
&\leq M \lVert u_{0,1}-u_{0,2}\lVert_{\mathcal{X}^{1}_q(\Omega)}+ CM\lVert \varphi_{1} -\varphi_{2} \Vert_{\mathcal{U}_{ad}}t^{\theta}\int^{t}_0 (t-s)^{\underline{\gamma}-(1+\theta)}s^{-\Bar{\epsilon}}\d s\\
&+CMt^{\theta}\int^{t}_0\Big((t-s)^{\underline{\gamma}-(1+\theta)}s^{-\Bar{\epsilon}}(2+s^{-\Bar{\epsilon}(\rho_{1} -1)}+2s^{-\Bar{\epsilon}(\rho_{2}-1)})\Big)s^{\Bar{\epsilon}}\lVert u_1(s)-u_2(s)\lVert_{\mathcal{X}^{1+\Bar{\epsilon}}_q(\Omega)}\d s.
 \end{split}
\end{gather}
Let $\theta\in[0;\underline{\gamma}[$, and consider the function
\begin{gather*}
\mathbf{\psi}_{\theta}(s):=(t-s)^{\underline{\gamma}-(1+\theta)}s^{-\Bar{\epsilon}}(2+s^{-\Bar{\epsilon}(\rho_{1} -1)}+2s^{-\Bar{\epsilon}(\rho_{2}-1)}).
\end{gather*}
Using the fact that $\underline{\gamma}>\Bar{\epsilon}$ and that $\forall t\in]0; T],$
\begin{gather*}
\begin{split}
 t^{\theta}\int^{t}_0 (t-s)^{\underline{\gamma}-(1+\theta)}s^{-\Bar{\epsilon}}\d s &=t^{\underline{\gamma}-\Bar{\epsilon}}\mathbf{B}(\underline{\gamma}-\theta;1-\Bar{\epsilon})\\
 &\leq T^{\underline{\gamma}-\Bar{\epsilon}}\mathbf{B}(\underline{\gamma}-\theta;1-\Bar{\epsilon})\\
 &\leq C,
 \end{split}
 \end{gather*}
we have
 \begin{gather}
 \label{dif4}
 \begin{split}
t^{\theta}\!\!\!\int^{t}_{0}\!\!\!\mathbf{\psi}_{\theta}(s)\d s &\!=\!2t^{\theta}\!\!\!\int^{t}_{0}\!\!\!(t-s)^{\underline{\gamma}-(1+\theta)}s^{-\Bar{\epsilon}}\d s \!+t^{\theta}\!\!\!\int^{t}_{0}\!\!\!(t-s)^{\underline{\gamma}-(1+\theta)}s^{-\Bar{\epsilon\rho_1}}\d s \!+2t^{\theta}\!\!\!\int^{t}_{0}\!\!\!(t-s)^{\underline{\gamma}-(1+\theta)}s^{-\Bar{\epsilon\rho_2}}\d s\\
&\leq 2T^{\underline{\gamma}-\Bar{\epsilon}}\mathbf{B}(\underline{\gamma}-\theta;1-\Bar{\epsilon})+T^{\underline{\gamma}-\Bar{\epsilon}}\mathbf{B}(\underline{\gamma}-\theta; 1-\rho_1 \Bar{\epsilon}) +2T^{\underline{\gamma}-\Bar{\epsilon}}\mathbf{B}(\underline{\gamma}-\theta,1-\rho_2\Bar{\epsilon})\\
&\leq C.
\end{split}
\end{gather}
The estimate \eqref{dif3} can be rewritten as follows:
\begin{gather}
\begin{split}
\label{stab}
  t^{\theta}\lVert \mathbf{u}_1(t,\cdot)-\mathbf{u}_2(t,\cdot)\lVert_{\mathcal{X}^{1+\theta}_q(\Omega)}\leq &M \lVert u_{0,1}-u_{0,2}\lVert_{\mathcal{X}^{1}_q(\Omega)}
  +CM\lVert \varphi_{1} -\varphi_{2} \Vert_{\mathcal{U}_{ad}}\\
  &+CMt^{\theta}\int^{t}_0  \mathbf{\psi}_{\theta}(s)s^{\Bar{\epsilon}}\lVert \mathbf{u}_1(s)-\mathbf{u}_2(s)\lVert_{\mathcal{X}^{1+\Bar{\epsilon}}_q(\Omega)}\d s.
\end{split}
\end{gather}
In particular, when \(\theta = \Bar{\epsilon} < \underline{\gamma}\), we can apply Grönwall's Lemma to obtain
\begin{gather}
\begin{split}
\label{stab2}
t^{\Bar{\epsilon}}\lVert \mathbf{u}_1(t,\cdot)-\mathbf{u}_2(t,\cdot)\lVert_{\mathcal{X}^{1+\Bar{\epsilon}}_q(\Omega)}
&\leq\Big( M \lVert u_{0,1}-u_{0,2}\lVert_{\mathcal{X}^{1}_q(\Omega)}+CM\lVert \varphi_{1} -\varphi_{2} \Vert_{\mathcal{U}_{ad}}\Big)\e^{CMt^{\Bar{\epsilon}}\int^{t}_{0} \mathbf{\psi}_{\Bar{\epsilon}}(s)\d s}\\
    &\leq C_{0}\Big(\lVert u_{0,1}-u_{0,2}\lVert_{\mathcal{X}^{1}_q(\Omega)}+\lVert \varphi_{1} -\varphi_{2} \Vert_{\mathcal{U}_{ad}}\Big).
\end{split}
\end{gather}
Therefore, substituting \eqref{stab2} in \eqref{stab} and using \eqref{dif4}, we obtain ($\forall\theta\in[0;\underline{\gamma}[$):
  \begin{gather}
   \label{stab3} 
  \begin{split}
  t^{\theta}\lVert \mathbf{u}_1(t,\cdot)-\mathbf{u}_2(t,\cdot)\lVert_{\mathcal{X}^{1+\theta}_q(\Omega)}\leq {c}\Big( \lVert u_{0,1}-u_{0,2}\lVert_{\mathcal{X}^{1}_q(\Omega)}+\lVert \varphi_{1} -\varphi_{2} \Vert_{\mathcal{U}_{ad}}\Big).
   \end{split}
  \end{gather}
\item For the maximal regularity $(\theta=\underline{\gamma}),$ we proceed as follows.
Let $\tilde{\theta}\in[0;\underline{\gamma}[,$ $\tau\in]0; t[$ and $t\in]0; T].$\\
 We have for ($i=1,2$, and $f:=f_1 +\tilde{f}_{2}$):
\begin{gather*}
    \mathbf{u}_{i}(t,\cdot)=\e^{\tilde{A}(t-\tau)}\mathbf{u}_{i}(\tau,\cdot)+\int^{t}_{\tau}\e^{\tilde{A}(t-s)}f(s,\x,\mathbf{u}_{i}(s))\d s \\
    \mathbf{u}_{i}(t+\varepsilon,\cdot)=e^{\tilde{A}(t+\varepsilon-\tau)}\mathbf{u}_{i}(\tau,\cdot)+\int^{t+\varepsilon}_{\tau}\e^{\tilde{A}(t+\varepsilon-s)}f(s,\x,\mathbf{u}_{i}(s))\d s,
\end{gather*}
Using Holder's inequality and \eqref{defM}, we obtain
\begin{gather*}
\begin{split}
\lVert \mathbf{u}_{1}(t+\varepsilon,\cdot)&-\mathbf{u}_{2}(t+\varepsilon,\cdot)\lVert_{\mathcal{X}^{1+\underline{\gamma}}_q(\Omega)}\leq M(t+\varepsilon-\tau)^{\tilde{\theta}-\underline{\gamma}}\lVert \mathbf{u}_{1}(\tau,\cdot)-\mathbf{u}_{2}(\tau,\cdot)\lVert_{\mathcal{X}^{1+\tilde{\theta}}_q(\Omega)}\\&+\int^{t+\varepsilon}_{\tau} \mathbf{\psi}_{\tilde{\theta}}(s)s^{\Bar{\epsilon}}\lVert \mathbf{u}_1(s)-\mathbf{u}_2(s)\lVert_{\mathcal{X}^{1+\Bar{\epsilon}}_q(\Omega)}
+CM\lVert \varphi_{1}-\varphi_{2} \Vert_{\mathcal{U}_{ad}}.
\end{split}
\end{gather*}
Now, when $\varepsilon\to 0,$ we have
\begin{gather*}
\begin{split}
\lVert \mathbf{u}_{1}(t,\cdot)-\mathbf{u}_{2}(t,\cdot)\lVert_{\mathcal{X}^{1+\underline{\gamma}}_q(\Omega)}&\leq M(t-\tau)^{\tilde{\theta}-\underline{\gamma}}\lVert \mathbf{u}_{1}(\tau,\cdot)-\mathbf{u}_{2}(\tau,\cdot)\lVert_{\mathcal{X}^{1+\tilde{\theta}}_q(\Omega)}\\
&+\int^{t}_{\tau} \mathbf{\psi}_{\tilde{\theta}}(s)s^{\Bar{\epsilon}}\lVert \mathbf{u}_1(s)-\mathbf{u}_2(s)\lVert_{\mathcal{X}^{1+\Bar{\epsilon}}_q(\Omega)}\d s
+CM\lVert \varphi_{1} -\varphi_{2} \Vert_{\mathcal{U}_{ad}}.
\end{split}
\end{gather*}
So, choosing $t-\tau=\tau=\frac{t}{2}$ for all $t\in ]0; T]$, and using \eqref{dif4} and \eqref{stab3}, we obtain
 \begin{gather}
 \label{ane1}
  \begin{split}
  t^{\underline{\gamma}}\lVert \mathbf{u}_1(t,\cdot)-\mathbf{u}_2(t,\cdot)\lVert_{\mathcal{X}^{1+\underline{\gamma}}_q(\Omega)}\leq c\Big(\lVert u_{0,1}-u_{0,2}\lVert_{\mathcal{X}^{1}_q(\Omega)}+\lVert \varphi_{1} -\varphi_{2} \Vert_{\mathcal{U}_{ad}}\Big).
   \end{split}
  \end{gather} 
Then, from \eqref{stab3} and \eqref{ane1}, we conclude that for all $\theta\in[0;\underline{\gamma}]$ and $t\in]0; T],$
  \begin{gather}
  \label{sas}
  \begin{split}
  \lVert \mathbf{u}_1(t,\cdot)-\mathbf{u}_2(t,\cdot)\lVert_{\mathcal{X}^{1+\theta}_q(\Omega)}\leq Ct^{-\theta}\Big( \lVert u_{0,1}-u_{0,2}\lVert_{\mathcal{X}^{1}_q(\Omega)}+\lVert \varphi_{1} -\varphi_{2} \Vert_{\mathcal{U}_{ad}}\Big),
   \end{split}
  \end{gather}
 which completes the proof of first estimate of \eqref{stabilité dt}.
\item We prove the second estimate of \eqref{stabilité dt} as follows. From \eqref{mild form}, we obtain (with $f=f_1 +\tilde{f}_2 $):
  \begin{gather}
  \label{gla}
  \begin{split}
  \partial_t\mathbf{u}_1(t,\cdot)-\partial_t \mathbf{u}_2(t,\cdot)=\tilde{A}(\mathbf{u}_{1}(t)-\mathbf{u}_{2}(t))+\Big(f(t,\x,\mathbf{u}_{1})-f(t,\x,\mathbf{u}_{2})\Big),
  \end{split}
\end{gather}
Let $\beta\in [0;\underline{\gamma}[$ and $t\in (0; T).$ Using the definition \ref{def epreg}, we have 
\begin{gather}
\label{dif8}
\begin{split}
    t^{\beta}\lVert f(s,\x,\mathbf{u}_1(s))\!-\!f(s,\x,\mathbf{u}_2(s))\lVert_{\mathcal{X}^{\beta}_q(\Omega)}&\leq  2Ct^{\beta-\Bar{\epsilon}}(1+t^{-\Bar{\epsilon}(\rho_1 -1)})t^{\Bar{\epsilon}}\lVert \mathbf{u}_1(t,\cdot)-\mathbf{u}_2(t,\cdot)\lVert_{\mathcal{X}^{1+\Bar{\epsilon}}_q(\Omega)}\\
&+ Ct^{\beta-\Bar{\epsilon}}\lVert \varphi_{1} -\varphi_{2} \Vert_{\mathcal{U}_{ad}}\\
&\leq C\Big(\lVert u_{0,1}-u_{0,2}\lVert_{\mathcal{X}^{1}_q(\Omega)}+\lVert \varphi_{1} -\varphi_{2} \Vert_{\mathcal{U}_{ad}}\Big).
\end{split}
\end{gather}
So, according to \eqref{sas}, \eqref{gla} and \eqref{dif8}:
 \begin{gather*}
 \label{dif5}
 \begin{split}
  t^{\beta}\lVert\partial_t \mathbf{u}_1(t,\cdot)-\partial_t \mathbf{u}_2(t,\cdot)\lVert_{\mathcal{X}^{\beta}_q(\Omega)}
  &\leq t^{\beta}\lVert\mathbf{u}_{1}(t)-\mathbf{u}_{2}(t)\lVert_{\mathcal{X}^{\beta+1}_q(\Omega)}+t^{\beta}\lVert f(s,\x,\mathbf{u}_1(s))-f(s,\x,\mathbf{u}_2(s))\lVert_{\mathcal{X}^{\beta}_q(\Omega)}\\
 &\leq C\Big( \lVert u_{0,1}-u_{0,2}\lVert_{\mathcal{X}^{1}_q(\Omega)}
 +\lVert \varphi_{1} -\varphi_{2} \Vert_{\mathcal{U}_{ad}}\Big).
  \end{split}
\end{gather*}
Therefore, for all $\beta\in[0;\underline{\gamma}[$, and $t\in(0; T)$:
\begin{gather*}
     \lVert \partial_t\mathbf{u}_1(t,\cdot)-\partial_t \mathbf{u}_2(t,\cdot)\lVert_{\mathcal{X}^{\beta}_q(\Omega)}\leq Ct^{-\beta}\Big(\lVert u_{0,1}-u_{0,2}\lVert_{\mathcal{X}^{1}_q}+\lVert \varphi_{1} -\varphi_{2} \Vert_{\mathcal{U}_{ad}}\Big),
\end{gather*}
\end{enumerate}
 which completes the proof.
\subsection{Proof of Proposition \ref{principe de maximum}}\label{preuve principe maximum} 
Let $u_0 \in\mW^{r,q}(\Omega)$ ($r\in\{0,1\}$, $q>1$) such that $0\leq u_{0}\leq c_{s}$ and $\varphi\in\mathcal{U}_{ad}$.\\
We consider $\rho_{i}$ are non-integer values for $i=1,2$. The proof is the same when $\rho_{i} \in \NN$.
\begin{enumerate}
 \item First, we prove that if $u_{0}\ge 0,$ $a.e.$ in $\Omega,$ then $\mathbf{u}(t,\cdot)\ge 0,$ for all $t\in (0; T)$ and $a.e.$ in $\Omega.$ We have $\mathbf{u}(t,\cdot)\in\mW^{1,q}(\Omega)$ is the classical solution of \eqref{Eqstate} on $(0; T)$ corresponding to the control $\varphi\in\mathcal{U}_{ad}.$ From \cite{GTru1983} for instance, we have $\mathbf{u}^{-}\in\mW^{1,q}(\Omega),$ and thus $\mathbf{u}^{-}|\mathbf{u}^{-}|^{q-2}\in\mW^{1,q^{*}}(\Omega).$\\
Multiplying the first equation of \eqref{Eqstate} by $v=-\mathbf{u}^{-}|\mathbf{u}^{-}|^{q-2}$, and integrating by part over $\Omega,$ we obtain ($\forall t\in(0; T )$):
\begin{gather*}
\begin{split}
   \frac{1}{q}\frac{\d  }{\d  t}\lVert \mathbf{u}^{-}(t,\cdot)\rVert^{q}_{\mL^{q}(\Omega)}&+(q-1)\int_{\Omega}\mD(t,\x)|\nabla{\mathbf{u}^{-}}(t,\x)|^{2}|\mathbf{u}^{-}(t,\x)|^{q-2}\d\x+\alpha_{0}\int_{\Omega}\varphi(t,\x) |\mathbf{u}^{-}(t,\x)|^{q}\d\x \\
   &+\int_{\Omega}\mK_{1}(t,\x)|\mathbf{u}(t,\x)|^{\rho_{1}-1}|\mathbf{u}^{-}(t,\x)|^{q}\d\x+\int_{\Gamma}\mK_{3}(t,\x)|\mathbf{u}(t,\x)|^{\rho_{2}-1}|\mathbf{u}^{-}(t,\x)|^{q}\d\x\\
   &=\int_{\Omega}\mK_{2}(t,\x)|\mathbf{u}^{-}(t,\x)|^{q}\d\x.
   \end{split}
\end{gather*}
Since all the terms on the left-hand side are positive, we can deduce that
\begin{gather*}
\frac{1}{q}\frac{\d}{\d t}\lVert \mathbf{u}^{-}(t,\cdot)\rVert^{q}_{\mL^{q}(\Omega)}\leq \overline{\mK_{2}}\int_{\Omega}|\mathbf{u}^{-}(t,\x)|^{q}\d\x.
\end{gather*}
According to Grönwall Lemma, for all time $t\in(0; T)$ 
\begin{gather*}
 \lVert \mathbf{u}^{-}\rVert^{q}_{\mL^{q}(\Omega)} (t)\leq \e^{q\overline{\mK_{2}}T}\lVert \mathbf{u}^{-}_{0}\rVert^{q}_{\mL^{q}(\Omega)}.
\end{gather*} 
Using the assumption $u_{0}\ge 0,$ we have  $u^{-}_{0}=0.$ Then, we can deduce $\mathbf{u}^{-}=0$, and that $\mathbf{u}(t,\cdot)\ge 0$ for all $t\in (0; T )$ and $a.e.$ in $\Omega.$
\item Next, we  prove that for all $t \in (0; T ),$ $ \mathbf{u}(t,\cdot)\leq c_{s}$ $a.e.$ in $\Omega.$  Let $\Omega^{+}:=\{\x\in\Omega:\hspace{0.1cm} u-c_{s}>0\}.$ By multiplying the first equation of \eqref{Eqstate} by $w=-(\mathbf{u}-c_{s})^{+}|(\mathbf{u}-c_{s})^{+}|^{q-2}$ and then integrating by part over all $\Omega,$ we obtain $a.e.$ in $(0; T )$
\begin{gather*}
\begin{split}
   \frac{1}{q}&\frac{\d}{\d t}\lVert (\mathbf{u}-c_{s})^{+}\rVert^{q}_{\mL^{q}(\Omega)}+(q-1)\!\int_{\Omega}\!\!\mD(t,\x)|\nabla{(\mathbf{u}-c_{s})^{+}}|^{2}|(\mathbf{u}-c_{s})^{+}|^{q-2}\d\x+\!\alpha_{0}\!\!\int_{\Omega}\!\!\varphi(t,\x) |(\mathbf{u}-c_{s})^{+}|^{q}\d\x\\
   & +c_{s}\int_{\Gamma}\mK_{3}(t,\x)|\mathbf{u}|^{\rho_{2}-1}(\mathbf{u}-c_{s})^{+}|(\mathbf{u}-c_{s})^{+}|^{q-2}\d\x
  +\int_{\Omega}\mK_{1}(t,\x)|\mathbf{u}|^{\rho_{1}-1}|(\mathbf{u}-c_{s})^{+}|^{q}\d\x\\ 
&+\alpha_{0}c_{s}\int_{\Omega}\varphi(t,\x)(\mathbf{u}-c_{s})^{+}|(\mathbf{u}-c_{s})^{+}|^{q-2}\d\x+\int_{\Gamma}\mK_{3}(t,\x)|\mathbf{u}|^{\rho_{2}-1}|(\mathbf{u}-c_{s})^{+}|^{q}\d\x\\
  & =\int_{\Omega}\!\!\mK_{2}(t,\x)|(\mathbf{u}-c_{s})^{+}|^{q}\d\x
   \!+\!c_{s}\!\!\int_{\Omega}\!\!\mK_{2}(t,\x)|(\mathbf{u}-c_{s})^{+}|^{q-1}\d\x
   \!-\!c_{s}\!\!\int_{\Omega}\!\!\mK_{1}(t,\x)|\mathbf{u}|^{\rho_{1}-1}|(\mathbf{u}-c_{s})^{+}|^{q-1}\d\x
     \end{split}
\end{gather*}
Since all the terms in the left-hand side are positive, we deduce that
\begin{gather*}
\begin{split}
 \frac{1}{q}\frac{\d}{\d t}\lVert (\mathbf{u}-c_{s})^{+}\rVert^{q}_{\mL^{q}(\Omega)} &\!\leq\! \overline{\mK_{2}}\!\!\int_{\Omega}\!\!|(\mathbf{u}-c_{s})^{+}|^{q}\d\x\!+\!c_{s}\!\!\int_{\Omega}\!\!\overline{\mK_{2}}|(\mathbf{u}-c_{s})^{+}|^{q-1}\d\x\!-\!c_{s}\!\!\int_{\Omega}\!\!\underline{\mK_{1}}|\mathbf{u}|^{\rho_{1}-1}|(\mathbf{u}-c_{s})^{+}|^{q-1}\\
 &\leq \overline{\mK_{2}}\int_{\Omega}|(\mathbf{u}-c_{s})^{+}|^{q}\d\x+c_{s}\int_{\Omega^{+}}\bigg(\overline{\mK_{2}}-\underline{\mK_{1}}c^{\rho_{1}-1}_{s}\bigg)|(\mathbf{u}-c_{s})^{+}|^{q-1}\d\x.
\end{split}
\end{gather*}
For $\overline{\mK_{2}}-\underline{\mK_{1}}c^{\rho_{1}-1}_{s}\leq 0,$ $i.e.,$ $c_{s}=\bigg{(}\frac{\overline{\mK_{2}}}{{\underline{\mK_{1}}}}\bigg{)}^{\frac{1}{\rho_{1}-1}},$ we can deduce that
\begin{gather*}
\frac{1}{q}\frac{\d}{\d t}\lVert (\mathbf{u}-c_{s})^{+}\rVert^{q}_{\mL^{q}(\Omega)}\leq \overline{\mK_{2}}\int_{\Omega}|(\mathbf{u}-c_{s})^{+}|^{q}\d\x,
\end{gather*}
and using Gronwall's Lemma again, we have ($\forall t\in(0; T )$): 
\begin{gather*}
 \lVert (\mathbf{u}-c_{s})^{+}\rVert^{q}_{\mL^{q}(\Omega)} (t)\leq \e^{q\overline{\mK_{2}}T}\lVert (u_{0}-c_{s})^{+}\rVert^{q}_{\mL^{q}(\Omega)}.
\end{gather*}
Since $u_{0}\leq c_{s},$ then $(u_{0}-c_{s})^{+}=0$.
Consequently, $(\mathbf{u} - c_{s})^{+} = 0$, which implies $\mathbf{u}(t, \cdot) \leq c_{s}$  for all $t\in (0; T )$ and $a.e.$ in $\Omega.$
\end{enumerate}
This ends the proof.




\bibliographystyle{cas-model2-names}

%
\bibliography{biblio_cas_refs}

\end{document}